\documentclass[11pt]{article}

\usepackage[utf8]{inputenc}
\usepackage[T1]{fontenc}
\usepackage{times}
\usepackage{amsmath,amsfonts,amssymb,amsthm,euscript,pifont}
\usepackage{mathrsfs}
\usepackage{upgreek}
\usepackage {epigraph}
\usepackage{authblk}
\usepackage[pagebackref,colorlinks=true]{hyperref}
\hypersetup{
    linkcolor = {blue},
    citecolor={blue}
    }

\usepackage{fullpage}
\usepackage{soul}
\usepackage{xspace}
\usepackage{float}
\restylefloat{table}
\usepackage{graphicx} 
\usepackage{caption}
\usepackage{subcaption}
\usepackage{multirow}
\usepackage{bbm}
\usepackage{fancyhdr}
\usepackage{anysize} 
\usepackage[euler]{textgreek}

\usepackage[nottoc]{tocbibind} \usepackage{ colortbl}
\usepackage{cite}
\usepackage{ifpdf}
\usepackage{url} 
\usepackage{numprint} 
\usepackage{ marvosym } 

\usepackage{color}
\definecolor{rltred}{rgb}{0.75,0,0}
\definecolor{rltgreen}{rgb}{0.75,0,0}
\definecolor{rltblue}{rgb}{0,0,0.75}
\definecolor{cColor}{rgb}{0.75,0,0.2}

\newcommand{\cxi}{\check{x}^{(i)}}

\newcommand{\eps}{\varepsilon}
\renewcommand{\P}{\mathbb{P}}
\newcommand{\Var}{\mathbb{V}\text{ar}}

\newcommand{\F}{\mathcal{F}}
\newcommand{\R}{\mathbb{R}}

\newcommand{\N}{\mathbb{N}}

\newcommand{\E}{\mathbb{E}}

\usepackage{dsfont}
\newcommand{\ind}[1]{\mathds{1}_{#1}}

\newcommand{\hX}{\widehat{X}}

\newcommand{\dt}{\Delta t}

\usepackage{dsfont}

\newcommand{\new}[1]{{#1}}

\newcommand{\tV}{\widetilde{V}}
\newcommand{\tx}{\widetilde{x}}
\newcommand{\tm}{\widetilde{m}}
\newcommand{\tn}{\widetilde{n}}
\renewcommand{\th}{\widetilde{h}}
\newcommand{\tX}{\widetilde{X}}
\newcommand{\ty}{\widetilde{y}}
\newcommand{\hV}{\widehat{V}}
\newcommand{\hx}{\widehat{x}}
\newcommand{\hm}{\widehat{m}}
\newcommand{\hn}{\widehat{n}}
\newcommand{\hh}{\widehat{h}}
\newcommand{\hy}{\widehat{y}}

\newcommand{\gNa}{g_\text{Na}}
\newcommand{\gL}{g_\text{L}}
\newcommand{\gK}{g_\text{K}}

\newcommand{\VNa}{V_\text{Na}}
\newcommand{\VL}{V_\text{L}}
\newcommand{\VK}{V_\text{K}}
\newcommand{\Vrev}{V_\text{rev}}
\newcommand{\Rmax}{R_\text{max}}

\newcommand{\JCh}{J_{\text{Ch}}}
\newcommand{\JE}{J_{\text{E}}}
 
\newcommand{\D}{\Delta}

\newcommand{\Vmax}{V^\text{max}}

\newcommand{\meanV}{\bar{V}^N}
\newcommand{\meanY}{\bar{y}^N}
\newcommand{\meanX}{\bar{x}^N}
\newcommand{\meanM}{\bar{m}^N}
\newcommand{\meanN}{\bar{n}^N}
\newcommand{\meanH}{\bar{h}^N}
\newcommand{\varV}{\bar{S}^V}
\newcommand{\varM}{\bar{S}^m}
\newcommand{\varN}{\bar{S}^n}
\newcommand{\varH}{\bar{S}^h}
\newcommand{\varX}{\bar{S}^x}

\newtheorem{theo}{Theorem}[section]
\newtheorem{hypothesis}[theo]{Hypothesis}
\newtheorem{lem}[theo]{Lemma}

\newtheorem{prop}[theo]{Proposition}
\newtheorem{corollary}[theo]{Corollary}

\newtheorem{rem}[theo]{Remark} 
\theoremstyle{definition}

\numberwithin{equation}{section}
\begin{document}
\pagestyle{plain}

\title{Synchronization of stochastic mean field networks of Hodgkin-Huxley neurons with noisy channels.   }

\author{Mireille Bossy\thanks{INRIA Sophia Antipolis M{\'e}diterran{\'e}e, France, 
Mireille.Bossy@inria.fr}, \;\;\; Joaqu\'in Fontbona\thanks{Center for Mathematical Modeling, University of  Chile. fontbona@dim.uchile.cl. Supported by CMM-Basal Conicyt  AFB 170001,  Nucleo Milenio NC 130062 and Fondecyt Grant 1150570} \; \; and \;\; H\'ector Olivero\thanks{ CIMFAV, Facultad de Ingenier\'ia, Universidad de Valpara\'iso, Chile. hector.olivero@uv.cl. Partially supported by Nucleo Milenio NC 130062,  Beca Chile Postdoctorado and Fondecyt Postdoctoral Grant 3180777.
 \\ \indent This research was  partially supported by the supercomputing infrastructure of NLHPC (ECM-02), Conicyt.}}%
\date{\today}

\maketitle

\abstract{In this work we are interested in a mathematical model \new{of  the collective behavior of a fully connected network of  finitely many neurons,  when their number  and when time go to infinity. } We assume that every neuron follows a stochastic version of the Hodgkin-Huxley model, and that pairs of neurons interact  through both electrical and chemical synapses, the global  connectivity being  of mean field type.  When the  leak conductance is strictly positive, we prove that if the  initial voltages are uniformly bounded and 
 the electrical interaction  between neurons is strong enough,  then, uniformly in the number of neurons,  the whole system synchronizes exponentially fast    as time goes to infinity,  up to some   error controlled by (and vanishing with) the channels noise level. Moreover,  we prove  that if the random initial condition is exchangeable,   on every bounded time interval  the propagation of chaos property for this system holds  (regardless of the interaction intensities). Combining these results, we deduce that the nonlinear  McKean-Vlasov equation describing an infinite network of such neurons  concentrates,  as \new{ time} goes to infinity, around the dynamics of a single Hodgkin-Huxley  neuron  with  \new{ chemical neurotransmitter channels}. Our results are illustrated and complemented with  numerical simulations. }

\medskip




\section{Introduction}
The  dynamics of  a neuron's  voltage is the result of the passage of  ions  through its membrane. This  ion  flux takes place through specific proteins which act as gated channels.  According to the Hodgkin-Huxley model of a nerve neuron  \cite{Hodgkin1952},  the  coupled behavior of the voltage of the  neuron $V_t$, with the proportions  $m_t$, $h_t$ and $n_t$ of open channels of the different ions involved in this process (respectively activation Sodium channels, deactivation Sodium channels and activation Potassium channels),  can be described by the following system of ordinary differential equations: 
\begin{equation}\label{eq:DeterministicHHmodel}
\begin{aligned}
 V_t& =  V_0 + \int_{0}^{t}F(V_s,m_s,n_s,h_s)ds\\
 x_t &= x_0 + \int_{0}^{t}{\rho_x(V_s)(1-x_s)-\zeta_x(V_s)x_s \, ds} 
\end{aligned}
\end{equation}
where, here and in the sequel,  $x$ generically represents the $m,n,h$ components and $F:\R\times [0,1]^4\to \R$, defined by 
\begin{equation}\label{eq:HHOriginalDynamic}
F(V,m,n,h)=I - \gK n^4(V-\VK) - \gNa m^3h(V-\VNa) -\gL (V-\VL), 
\end{equation}
represents the effect on the voltage of  the ionic channels  and of an external current  $I$  (assumed constant  for simplicity). The rate functions $\rho_x$ and $\zeta_x$,  originally considered in     \cite{Hodgkin1952}, have some  generic form given in \eqref{def:rho_m,n} and \eqref{def:rho_h} below; see also Figure \ref{figure:RateFunctions} and  Tables  \ref{tab:ParametersVoltage}  and \ref{tab:RateFunctionsSingleNeuronModel} for their shape  and for biologically meaningful values of the parameters.  We refer the reader to  Ermentrout and Terman \cite{Ermentrout:2010aa} for a  concise discussion on  the Hodgkin-Huxley (HH in the sequel) model  and its deduction, as well as  for general  background on mathematical   models. 

From a mathematical point of view, system \eqref{eq:DeterministicHHmodel} defines a rich dynamical system, the properties  of which has been extensively studied. As an example of its various possible behaviors,   Figure \ref{figure:SingleNeuronDifferentImputCurrents}  below illustrates different  possible responses of system \eqref{eq:DeterministicHHmodel} to the value of the input current  $I$  \new{ (no oscillations, oscillations  of various types,  damping) }  all other parameters  of the model being fixed. See e.g. \cite{Ermentrout:2010aa} and Izhikevich \cite{Izhikevich:2007aa},   and references therein for detailed accounts on dynamical properties of \eqref{eq:DeterministicHHmodel} and related neuron models.  
Lower dimensional  dynamics have  also been  proposed as simpler alternatives to \eqref{eq:DeterministicHHmodel}, the most important  ones being the  FitzHugh-Nagumo model (FitzHugh \cite{FitzHugh:1961aa}, Nagumo et al. \cite{Nagumo:1962aa})  and the Morris-Lecar  model (Morris and Lecar \cite{Morris:1981aa}). 
These  are able to reproduce some of the dynamical  features of the HH system \eqref{eq:DeterministicHHmodel} and are easier to study from the mathematical point of view, but  they are less realistic regarding  some of its relevant  features. 

 A different approach  to model the electric activity of neurons are integrate-and-fire models,  introduced in Lapicque \cite{Lapicque:1907aa}. In  these models   the electric potential evolves according to some ordinary differential equation until it reaches a certain fixed threshold; the  neuron  then  emits a potential spike and  the voltage is reset to some   reference value,  from which its evolution restarts following the same dynamics. We refer the reader to  Burkitt \cite{Burkitt2006}, \cite{Burkitt:2006aa} for a review of this class of models. 

In the last decade, there has been  an increasing interest  of the  mathematical and computational neuroscience communities in understanding the role of stochasticity in neurons' dynamics, as well as in  mathematical models for it. We refer the reader to Goldwyn et al. \cite{goldwyn2011stochastic} and to Goldwyn and Shea-Brown \cite{goldwyn2011and}  for a  discussion on different ways in which  randomness might be introduced in the HH model,  their biological interpretation  and  their pertinence. See also  \cite[Chapter 10]{Ermentrout:2010aa} for general background on this issue. 
  Classically, random models arise in the form of  finite Markov chains describing a discrete number of   open  gates  which approximate the ion channel dynamics, or by  directly introducing   Gaussian  additive or multiplicative white noise   (that is, a Brownian motion or a stochastic integral with respect to it) in the voltage or ion channels dynamics in \eqref{eq:DeterministicHHmodel}.  More recently,   hybrid (also called piecewise deterministic) Markov processes have also been proposed as microscopic counterparts of the HH or other deterministic models. In this setting, the channel variables are replaced  by discrete continuous time processes whose    jump rates depend on the voltage,  while keeping a continuous description for the latter, see Austin  \cite{Austin:2008aa}, Pakdaman et al.  \cite{K.-Pakdaman:2010aa}  and references therein.   We also refer the reader to  Bossy et al. \cite{Bossy2015},   Dangerfield et al. \cite{PhysRevE.85.051907}, Wainrib \cite{Wainrib:2010aa} and  Sacerdote and Giraudo \cite{Sacerdote2013} for further discussion on stochastic models in this context, the latter one  in the case   of integrate-and-fire models.
  
In the present work we  consider  stochastic versions of the HH model  \eqref{eq:DeterministicHHmodel} which arise as diffusive scaling limits of hybrid models  of the type studied in  \cite{Austin:2008aa} and  \cite{K.-Pakdaman:2010aa}. More precisely, we are interested in networks of $N$  such neurons in mean field interaction, which can be described  by a system of stochastic differential equations of the form: 
\begin{equation*}
\begin{aligned}
& V_t^{(i)}=  V^{(i)}_0 + \int_{0}^{t}F(V_s^{(i)}, m_s^{(i)},n_s^{(i)},h_s^{(i)})  ds\\
& \qquad\qquad\qquad-  \int_{0}^{t} \frac{1}{N}\sum_{j=1}^{N}\JE (V_s^{(i)}-V_s^{(j)}) 
 -\frac{1}{N}\sum_{j=1}^{N}{\JCh y^{(j)}_s (V_s^{(i)}-\Vrev)} ds,\\
& x^{(i)}_t = x^{(i)}_0+\int_{0}^{t} b_x(V^{(i)}_s,x^{(i)}_s)  ds + \int_{0}^{t}{\sigma_x(V_s^{(i)},x_s^{(i)})dW_s^{x,i}}, ~~ x=m,n,h,y. 
\end{aligned}
\end{equation*}
Here, $(W^{x,i}\,:   i=1,\dots, N, x=m,n,h,y )$  are one dimensional Brownian motions  and the interaction between neurons account for the effect of  electrical and chemical synapses (the biological  interpretation of the interactions  terms  and in particular of  the variables $y^{(i)}$ is given in next section).  We refer to equation  \eqref{eq:HH-model}  below for the explicit  form of the system we will consider, and for  Hypothesis \ref{hyp:MainHypotheses}  for our assumptions on its coefficients.

The collective behavior of neurons, and the way it emerges  from their individual  features and synaptic activity, is indeed a central question in neuroscience.  In particular, considerable efforts have been devoted to understanding synchronization of neurons, an ubiquitous  phenomenon  seemingly related  to the generation of rhythms (such as the respiratory one or the heartbeat) but also to more complex neurologic functionalities. For instance, at the brain level,  synchronization has been connected to memory formation, see Axmacher et al. \cite{Axmacher2006170}, but also to  disorders such as epileptic seizures, see Jiruska et al. \cite{TJP:TJP5487}. Since the neuroscience  literature on this topic is  huge,  it is not our intention to   thoroughly comment on it here, and we refer the reader to  \cite[Chapters 8,9]{Ermentrout:2010aa}, \cite[Chapter 10]{Izhikevich:2007aa} for further discussion on biological roles of  neuron synchronization,  and mathematical approaches to it. \new{For a broad perspective on synchronization, we also refer the reader to Pikovsky et al. \cite{pikovsky2003synchronization}. } 

\new{
In this paper we are interested in synchronization due to a strong enough coupling between neurons.  This phenomenon differs from synchronization owed to common noise addressed e.g.\ in Marella and Ermentrout \cite{marella2008class} (where  uncoupled oscillators subject to a common noise are observed to get synchronized) or  in Pikovsky \cite{pikovskii1984synchronization} (where synchronization results from the action of a random forcing term). In our case,  noise is unshared  by the interacting neurons and does therefore  not contribute to their synchronization; indeed, it actually  prevents the perfect asymptotic synchronization of the network. (This phenomenon might be compared to noise-induced deviations from stable cycles in noisy oscillators, see  \cite{berglund2004noise,berglund2014noise} for a large deviations approach to that problem.) 
 Moreover, we will understand and quantify synchronization in terms of the empirical variance of the system of neurons and of its dissipation, an approach which does not rely on  the stability properties of individual neuron dynamics nor, in particular, on the existence of some oscillatory  limiting behavior.

In the case of interacting oscillators, a central mathematical  tool is  phase reduction. Introduced by Kuramoto \cite{Kuramoto:1984aa},  it is based on the  idea that stable periodic solutions of a nonlinear oscillator can be parametrized by its phase in the limit cycle.} Kuramoto's model  has proved  useful to  understand synchronization mechanisms  of simple coupled oscillators  (see \cite{Ermentrout:2010aa} and \cite{Izhikevich:2007aa} ), or for ensembles of population of neurons with intrinsic and extrinsic noise (see Bressloff and Ming Lai \cite{Bressloff2011} and the references therein), or even in the limiting case  of infinite oscillators  with noise and mean field interaction (see Bertini et al.  \cite{BertiniGiacominPakdaman} and the discussion in \cite[Chapter 4]{Wainrib:2010aa}). However, to our knowledge,  applications of these ideas to  networks of HH-type neurons have so far been restricted to small deterministic networks  and  ``weak coupling'' regimes (see  e.g. Hansel and Mato \cite{hansel1993patterns} and Hansel et al. \cite{hansel1993phase} ).  We refer the reader to Ostojic et al. \cite{Ostojic2008} for  synchronization results in the case of integrate-and-fire networks.

A related question is the  asymptotic behavior of networks  when the number of neurons tends to infinity.  In that sense, networks of $N$ neurons  in mean field interaction, in which every neuron experiences a pairwise interaction of strength-order $1/N$ with each  other, provide a mathematically  tractable (though not completely realistic) framework to address this question. 
Indeed, in a mean field network, the  evolutions of  finitely many  neurons are expected to become independent  as $N$  goes to infinity, a property known as propagation of chaos. In the case of exchangeable particles, this  is equivalent to the convergence of the dynamics of the empirical law of the system to some deterministic flow of probability laws, typically described by a nonlinear  McKean-Vlasov partial differential equation  (also termed ``mean field'' equation in this context); see   M\'el\'eard \cite{meleard1996asymptotic},  Sznitman \cite{sznitman1991topics} for background on propagation of chaos. We refer the reader to e.g. Faugeras et al. \cite{10.3389/neuro.10.001.2009} for formal derivations of mean field equations for multi type population networks of integrate-and-fire neurons, and respectively to Delarue et al.\cite{delarue2015global} and  Perthame and Salort \cite{Perthame2013841}  for probabilistic and PDE approaches to the global solvability of that equation (which can in principle have  explosive solutions) when the interaction is small. See also Fournier and L\"ocherbach \cite{FournierLocherbach} for further  recent results on propagation  of chaos for integrate and fire models.  The propagation of chaos  for mean field networks of neurons described by  stochastic differential equations, including  stochastic, multi type  HH and FitzHugh-Nagumo networks, has been addressed in Baladron et al. \cite{Baladron:2012aa},  and  then rigorously established  in  Bossy et al.   \cite{Bossy2015}.  We also refer the reader to  Mischler et al. \cite{Mischler2016}  for the mean field description of a  network of  FitzHugh-Nagumo neurons. 

\medskip 

In this present work, we establish that, under  strong enough electrical connectivity of the network (i.e. large enough $\JE$),  the $N$  neurons get synchronized,  up to an error proportional to the  channels' noise level $\sigma^2$,   at an exponential rate which is independent of $N$.   Moreover, we  exhibit a deterministic single-neuron dynamics which  is  ``mimicked'', as time goes to infinity,  by every  neuron of  the system \eqref{eq:HH-model}, over  short enough   moving time-windows, up to an error that vanishes with $\sigma^2$ and $N^{-1}$.  As far as we know, this is a first  mathematical result which establishes the synchronization of large networks of neurons.   We also establish the propagation of chaos for system  \eqref{eq:HH-model}, or its convergence  to  solutions to  a  McKean-Valsov equation,  for arbitrary parameters of the model (and for slightly more general coefficients than in  \cite{Bossy2015} in the single population case). This allows us to transfer our synchronization results to  the limiting PDE, which can be understood as the description of 
an  infinite network of neurons. 
Our theoretical results will be complemented with simulations, which in particular point out that synchronization phenomena might also hold for small electrical  connectivity and even for pure chemical connectivity ($\JE=0$).

The remainder of the paper is organized as follows. In Section \ref{sec:theModel}, we detail the model we consider and  state precisely our main results. Section \ref{sec:numerical-experiments} is devoted  to numerical experiments, both to illustrate  our theoretical statements and  to explore  related phenomena  in mean field settings not covered by them (like multi-type neuron populations or chemical-only synapses). In Section \ref{sec:ConcluOpen} we a discuss  possible improvements and extensions of our results, and some open questions. The mathematical proofs of our results are given in the Appendix sections.

\section{Model and main results }\label{sec:theModel}

We start by briefly recalling  how chemical and electrical synapses in networks of neurons  are modeled (we follow  \cite[Chapter 7]{Ermentrout:2010aa}  which we also refer to for further background on synaptic channels). 

 In chemical synapses, a neurotransmitter is released to the intercellular media (technically the synaptic cleft), from a pre-synaptic neuron to the post-synaptic one through  synaptic  channels, which are  voltage-gated just as ion channels are. With each pre-synaptic  neuron  we can thus associate a new variable $y$  in $[0,1]$  which represents its proportion of open synaptic channels at  each  time.  The dynamics of this  variable   can be modeled in a similar way as those of ion channels, that is,   in terms of  certain rate functions $ \rho_y$ and $\zeta_y$ depending on  the membrane potential  $V$ of that same neuron, and on some parameters (see \eqref{def:rho_y}).    The choice of these parameters determines the characteristic  (inhibitory or excitatory) of the chemical synapse. Hence, in a  fully connected network of $N$  similar neurons, chemical synapses coming from a pre-synaptic neuron $j$ should induce on the voltage  $V^{(i)}$ of the post-synaptic neuron $i$ an instantaneous variation at time $t$  of
 $$-\frac{\JCh}{N} y_t^{(j)}(V^{(i)}_t-\Vrev),$$
where  $y^{(j)}_t$ is the proportion of open synaptic channels of neuron $j$, $\JCh\geq 0$ is a constant  representing the chemical conductance of the network and $\Vrev$ is a reference potential. The factor $\frac{1}{N}$ is introduced in order that the contribution of each incoming synapse to the neuron $i$ has similar weight, which  corresponds  to a  global interaction of  mean field type. 

 On the other hand, the interior of one neuron can be directly connected with  another neuron's one through an intercellular channel called gap junction, which allows the constant flow of ions between them,   as a result of their possibly  different potentials. We thus  may assume that  pre-synaptic neuron $j$ 
contributes to the variation  of the voltage of post-synaptic neuron $i$  by the  amount 
$$-\frac{\JE}{N}(V^{(i)}_s-V^{(j)}_s),$$
where $\JE\geq 0$ is the electrical conductance (that can be thought of as a measure of the connectivity of the network) and the factor $\frac{1}{N}$ appears by similar reasons as before.   Connections of this type are termed electrical synapses and are less frequent than chemical ones; on the other hand,  they transmit information faster. (See also Hormuzdi et al. \cite{Hormuzdi2004113} for a deeper discussion on electrical synapses.)

\medskip 

In all the sequel, for each fixed $N$ we consider a stochastic process $X=(X^{(1)},\dots,X^{(N)})$  valued  in $(\R^5)^N$,  with coordinates  $X^{(i)}_t=(V^{(i)}_t,m^{(i)}_t,n^{(i)}_t,h^{(i)}_t,y^{(i)}_t)$   given  for $i=1,\dots, N$ and $t\geq 0$ by the solution 
 of the system of stochastic differential equations:
  \begin{equation}\label{eq:HH-model}
\begin{aligned}
& V_t^{(i)}=  V^{(i)}_0 + \int_{0}^{t}F(V_s^{(i)}, m_s^{(i)},n_s^{(i)},h_s^{(i)}) ds \\
& \qquad\qquad   - \int_0^t \frac{1}{N}\sum_{j=1}^{N}\JE (V_s^{(i)}-V_s^{(j)})  -\frac{1}{N}\sum_{j=1}^{N}{\JCh y^{(j)}_s (V_s^{(i)}-\Vrev)} ds,\\
& x^{(i)}_t= x^{(i)}_0+\int_{0}^{t}\rho_x(V^{(i)}_s)(1-x^{(i)}_s)  -\zeta_x(V^{(i)}_s)x^{(i)}_sds + \int_{0}^{t}{\sigma_x(V_s^{(i)},x_s^{(i)})dW_s^{x,i}}, 
\end{aligned}
\end{equation}
where $(W^{x,i}\,:   i\in \N, x=m,n,h,y )$  are independent one dimensional Brownian motions independent of $X_0$ and $F$ is defined in \eqref{eq:HHOriginalDynamic}.
Notice that,  for notational simplicity, the dependence of system  \eqref{eq:HH-model} on $N$ is omitted.  Throughout this work, we will additionally make the following   assumptions on  system \eqref{eq:HH-model}: 

\begin{hypothesis}\label{hyp:MainHypotheses} 

\begin{enumerate}
\item \label{hyp:properties-of-rho-zeta} For $x=m,n,h$ and $y$,  $\rho_x$ and $\zeta_x$   are  strictly positive, locally  Lipschitz continuous functions  defined on $\R$.
\item \label{hyp:definition-of-sigma} For $x=m,n,h$ and $y$,  functions $\sigma_x:\R^2\to \R$  are given by
\begin{equation}\label{eq:def-sigma_x}
\sigma_x(v,z) = \sigma\sqrt{|\rho_x(v)(1-z) + \zeta_x(v)z|}\chi(z),
\end{equation}
with $\chi: \R\to [0,1]$ a Lipschitz continuous function with support contained in $[0,1]$ and $\sigma\geq 0$. 
\item   One has  $( m_0^{(i)},n_0^{(i)},h_0^{(i)},y_0^{(i)})\in [0,1]^4$ a.s. 
\end{enumerate}
\end{hypothesis}

 These assumptions cover, for parameters $a_r^x, a_d^x> 0$,   functions  of the form
\begin{align}\label{def:rho_m,n}
\rho_x(V) = \frac{a_r^x (V-V_r^x)}{1-\exp\left(-\lambda_r^x(V-V_r^x)\right)},\quad \zeta_x(V)=a_d^x\exp\left(-\lambda_d^x(V-V_d^x)\right),
\end{align}
for $x=m,n$, and
\begin{align}\label{def:rho_h}
\rho_h(V) = a_r^h\exp\left(-\lambda_r^h(V-V_r^h)\right),\quad\zeta_h(V)=\frac{a_d^h}{1+\exp\left(-\lambda_d^h(V-V_d^h)\right)}, 
\end{align}
considered in the original HH model  \cite{Hodgkin1952}, as well as  functions 
\begin{align}\label{def:rho_y}
\rho_y(V) = \frac{a^y_r T_{\max}}{1+\exp\left(-\lambda(V-V_T)\right)},\;\;\;\zeta_y(V)=a^y_d 
\end{align}
associated with synaptic channels in   \cite[Chapter 7]{Ermentrout:2010aa}.  Diffusion coefficients  $\sigma_x$ defined in terms of the  functions  $\rho_x$ and $\zeta_x$ as in \eqref{eq:def-sigma_x}
 have been considered in 
\cite{Bossy2015}, and arise naturally in diffusive scaling limits of the hybrid models studied in \cite{K.-Pakdaman:2010aa}.

Observe that  for functions  $\rho_h$, $\zeta_m$ and $\zeta_n$ as above, the coefficients   of  system \eqref{eq:HH-model}   do not  satisfy classic   conditions for wellposedness. This well-posedness will be proved in Lemma \ref{lem:well-posedHH} below, relying on results  in \cite{Bossy2015} that ensure that under Hypothesis \ref{hyp:MainHypotheses}  the  processes $( m_t^{(i)},n_t^{(i)},h_t^{(i)},y_t^{(i)})$ remain in $[0,1]^4$. 
Notice that the absolute value in \eqref{eq:def-sigma_x} can then be removed. 
 
\subsubsection*{Synchronization}
\new{  By  synchronization we will understand the dissipation of the empirical variance of the network \eqref{eq:HH-model} as time goes by.  
In Figure \ref{figure:def-synchro} we show two extreme  situations in this regard. On the left, for a small interaction parameter $\JE$ and noise  $\sigma \neq 0$ we observe a chaotic behavior resulting in an empirical variance of constant order in time. On the right, for large $\JE$ and  $\sigma = 0$   we observe the fast  emergence of a coherent evolution implying the dissipation of the empirical variance. Our main result, Theorem \ref{theo:synchro-empiricalVar} will provide a quantitative picture  of this behavior with respect to  noise level $\sigma$ and the size of the network $N$,  for large enough  connectivity $\JE$.  Our results  require the following additional assumption:

\begin{figure}[ht!]
    \centering
    \begin{subfigure}[t]{0.4\textwidth}
        \centering
        \includegraphics[width=\textwidth]{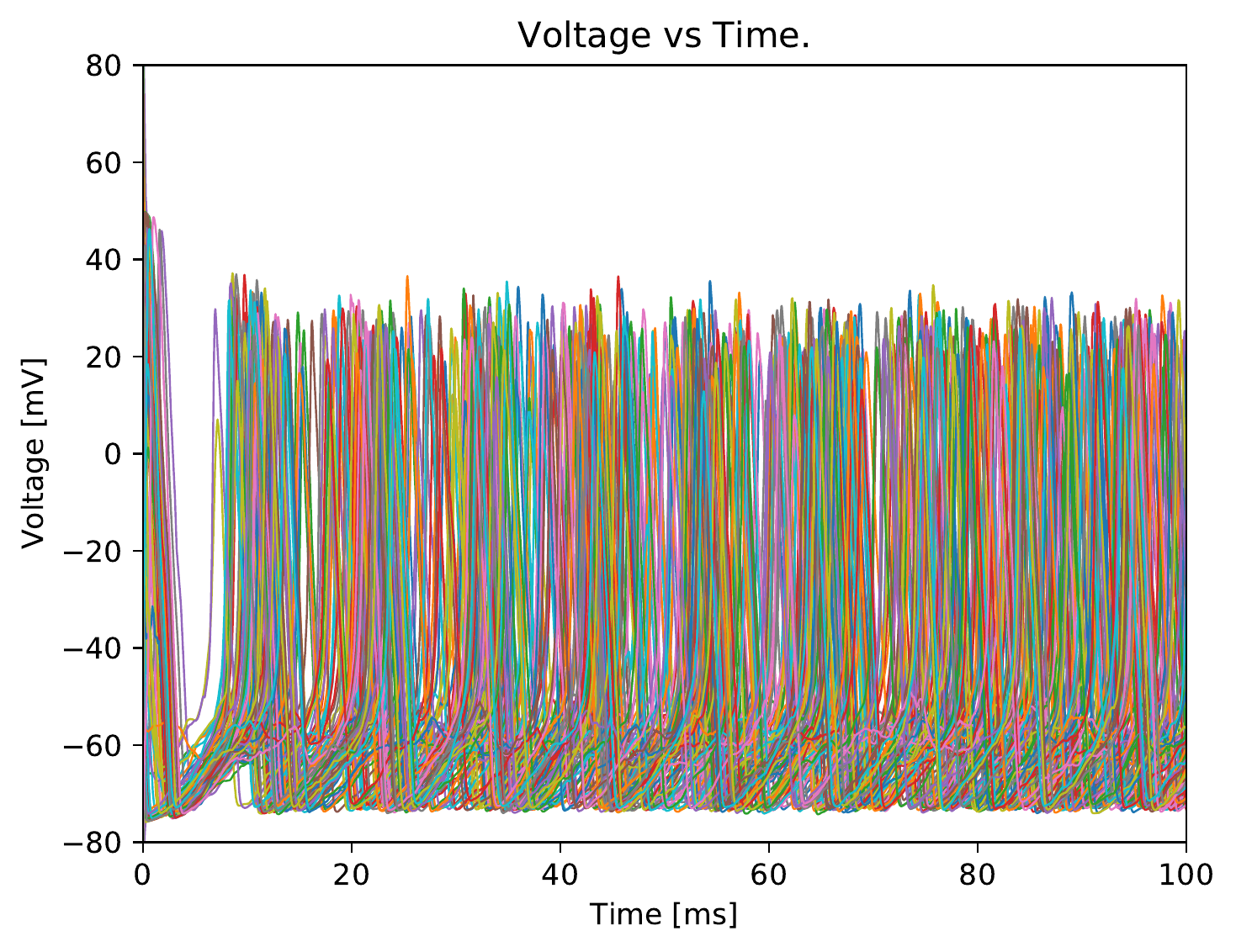}
        \caption{\fontsize{9}{11}\selectfont \new{  Small interaction parameter and noise}. \label{figure:def-synchro-a}}
    \end{subfigure}
    \quad
    \begin{subfigure}[t]{0.4\textwidth}
        \centering
        \includegraphics[width=\textwidth]{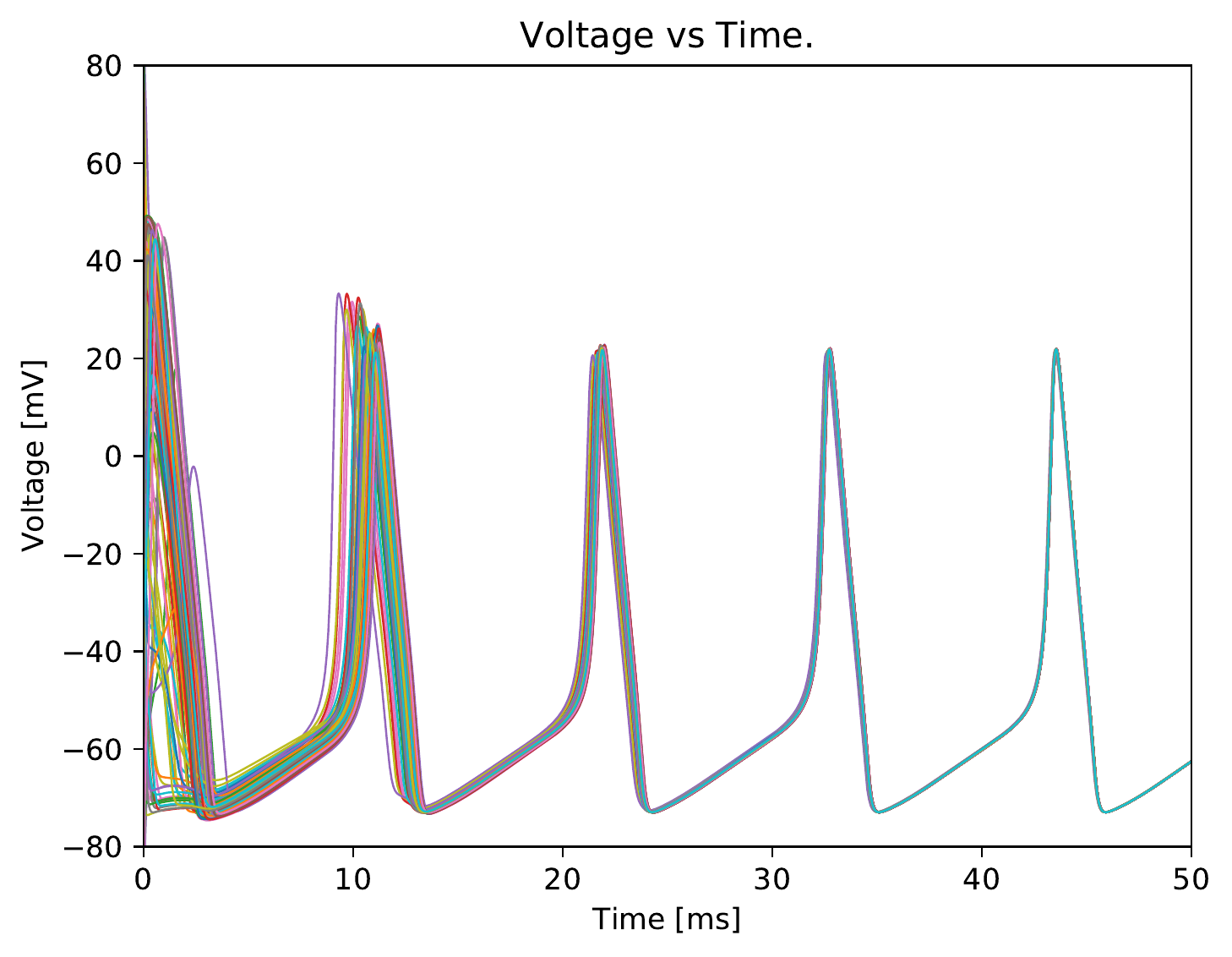}
        \caption{\fontsize{9}{11}\selectfont \new{ Large interaction parameter and no noise}.\label{figure:def-synchro-b}}
    \end{subfigure}
 
 \caption{\fontsize{9}{11}\selectfont Two typical situations of the evolution of the network  \eqref{eq:HH-model}.\label{figure:def-synchro}}
\end{figure}
}

\begin{hypothesis} \label{hyp:V0}
Hypothesis \ref{hyp:MainHypotheses} holds and moreover: 
\begin{enumerate}
\item[4.] The  parameter $g_L$ in \eqref{eq:HHOriginalDynamic} (termed  ``leak conductance'') is strictly positive.
\item[5.]  There exists  a constant $\Vmax_0>0$ not depending on $N$ such that 
$$\sup_{i=1,\ldots,N}|V^{(i)}_0|\leq \Vmax_0\;\;\; a.s.$$  
\end{enumerate}
\end{hypothesis}

We also need to introduce  notation for some  empirical means, namely 
\[\meanV_t = \frac{1}{N}\sum_{i=1}^N V _t^{(i)} \, , \quad  \bar{X}_t^N=  \frac{1}{N}\sum_{i=1}^N X _t^{(i)} \,\,\, \mbox{ and }  \, \,\, \meanX_t = \frac{1}{N}\sum_{i=1}^N  x_t^{(i)} \,  \text{ for $x=m,n,h,y$}.\]
Moreover, for  each  $N\geq 1$ and $t_1\geq 0$ we let   $(\hX^{N,t_1}_t:\,  t \geq t_1)= \left(  (\hV^N_t,\hm_t^N,\hn_t^N,\hh_t^N,\hy_t^N): \,  t\geq t_1\right) $ denote  the solution of the ordinary differential equation
\begin{equation}\label{eq:HatProcess} 
\begin{aligned}
\hV_t^N &= \hV_{t_1}^N+ \int_{t_1}^{t}{F(\hV^{N}_s,\hm_s^N,\hn_s^N,\hh_s^N)-\JCh \hy_s^N(\hV^{N}_s-\Vrev) ds},\\
 \hx^{N}_t& = \hx^{N}_{t_1}+\int_{t_1}^{t}\rho_x(\hV^{N}_s)(1-\hx^{N}_s)  -\zeta_x(\hV^{N}_s)\hx^{N}_sds ,\;\;x=m,n,h,y.
\end{aligned}
\end{equation}
with random initial condition  $$\hX_{t_1}^{N,t_1} = \bar{X}^N_{t_1}.$$

We are now is position to state our  main result about  synchronization of the  system \eqref{eq:HH-model}: 
\begin{theo}\label{theo:synchro-empiricalVar}
Assume Hypothesis \ref{hyp:V0} holds and that $(X_0^{(1)}, \dots, X_0^{(N)})$ is an exchangeable  random vector.  
\begin{itemize}
\item[a)]  {\bf Synchronization.}  There exist  constants $\JE^0>0$,  $C^0_{\zeta,\rho}>0$  and  $\lambda^0>0$ not depending on $N\geq 1$, $\sigma\geq 0$ or $X_0$, and there exists a time  $t_0\geq 0$ not depending on $N\geq 1$ or $\sigma\geq 0$, such that for each  $\JE>\JE^0$  the solution $X$ of \eqref{eq:HH-model}  satisfies, for every $t\geq t_0$ and each $i\in \{1,\dots,N\}$: 
\begin{equation}\label{eq:dissipation_Of_Empirical_Variance}
\E\left(   | X^{(i)}_t-\bar{X}_t^N|^2 \right)\leq\E\left(    | X^{(i)}_{t_0}-\bar{X}_{t_0}|^2 \right)e^{-\lambda^0( t-t_0)} + \sigma^2 \frac{C^0_{\zeta,\rho}}{\lambda^0}.
\end{equation}
In particular, 
$\limsup_{t\to\infty}\E\left(   | X^{(i)}_t-\bar{X}_t^N|^2 \right) \leq \sigma^2 \frac{C^0_{\zeta,\rho}}{\lambda^0}.$
\item[b)]  {\bf Synchronized dynamics. } Assume  $\JE>\JE^0$. Then, there are  constant  $K_0, K_0'>0$ depending only on the parameters of the voltage dynamics in \eqref{eq:HHOriginalDynamic}  and, for each $\delta\geq 0$, constants $K_{\delta} , K'_{\delta}>0$ depending on the coefficients  in \eqref{eq:HH-model}  and on $\delta$  (increasingly) but not on $N$,   such that for every $t_1\geq t_0$  and each $i\in \{1,\dots,N\}$:
\begin{align}\label{eq:synchonized_dynamics}
\begin{aligned}
&\sup_{t_1\leq t\leq t_1+\delta}     \E\left(   | X^{(i)}_t- \hX^{t_1,N}_t|^2\right) \\
& \qquad \leq K_0  \wedge  2 \bigg[     \left( K_0' e^{-\lambda_0(t_1-t_0)} + \sigma^2  \frac{ C^0_{\zeta,\rho}}{\lambda^0}\right) (1+ \delta K_{\delta})+   \delta  K'_{\delta}  \frac{\sigma^2  }{N} C^0_{\zeta,\rho}   \bigg].
\end{aligned}
\end{align} 
\end{itemize}
\end{theo}
  
Some remarks on this result are in order: 
\begin{rem}
~
\begin{itemize}
\item[(i)] The constants    $C^0_{\zeta,\rho}$  and $t_0$ depend explicitly on the coefficients of the system,  with the latter possibly depending also on $\Vmax_0$. On the other hand,  bounds  for  $\lambda_0$  and $\JE^0$ which do not depend on the initial data can also be obtained.  The remaining constants are explicit and do not depend on the initial condition. No constant is claimed to be optimal.
\item[(ii)] The bound $K_0$  in Theorem \ref{theo:synchro-empiricalVar} b)  is deduced from global  bounds  (which we establish) on the voltage processes and their average, and its role is  only  to prevent the r.h.s.\ from growing arbitrarily with $\delta$. The estimate becomes informative  as $t_1\to \infty$ for  small enough $\sigma^2>0$, $\delta>0$ and $N^{-1}$,  and  for any $\delta>0$ and $N$ if $\sigma^2=0$.
\item[(iii)]  Aside from the assumption that $g_L>0$, Theorem \ref{theo:synchro-empiricalVar}  holds regardless of the values  of the  parameters of  voltage dynamics  $F$ in \eqref{eq:HHOriginalDynamic}, and in particular if the input current $I$  is replaced by a uniformly bounded function  or  (suitably measurable) process.  
\item[(iv)]  The exchangeability assumption can be removed at the price of adding inside the expectations in \eqref{eq:dissipation_Of_Empirical_Variance} and \eqref{eq:synchonized_dynamics} the averages over $i$.  
\end{itemize}
\end{rem}

\subsubsection*{Mean field limit}

 We next address the question of  the behavior of  system   \eqref{eq:HH-model} as  $N\to \infty$.   We need to introduce additional notation: 
 
\begin{itemize}
\item  We denote by   ${\cal P}(\R\times [0,1]^4)$ the space of Borel probability measures on $\R\times [0,1]^4$  endowed with the weak topology, and by  ${\cal P}_2(\R\times [0,1]^4)$ its subspace of probability measures with finite second moment, endowed with the  Wasserstein distance  $\mathcal{W}_2$. 
That is, for all $\mu_1,\mu_2\in {\cal P}_2(\R\times [0,1]^4)$, 
$$\mathcal{W}^2(\mu_1,\mu_2) =\inf_{\mu\in\Pi(\mu_1,\mu_2)}\int_{\R^5}  |r_1 -r_2|^2 \mu(dr_1,dr_2)  ,$$
with $\Pi(\mu_1,\mu_2)$ the  set of probability measures on $(\R\times [0,1]^4)^2$  with  first and second  marginals equal to  $\mu_1$ and $\mu_2$ respectively. It is well known that  the  infimum is attained and that $\mathcal{W}_2$ defines a complete metric on ${\cal P}_2(\R\times [0,1]^4)$  inducing  the weak topology,  strengthened with the convergence of second moments (see Villani \cite{villani2009} for the  relevant properties of  Wasserstein metrics). 
\item Elements of  $\R\times [0,1]^4$  describing the state space of a single neuron's dynamics  will be written in the form $(v,u)  =(v,(u_m,u_n,u_h,u_y)) $, with $u=(u_m,u_n,u_h,u_y)\in [0,1]^4$.
\item  We introduce  the function  $\Phi: ( \R\times [0,1])\times (\R\times [0,1]^4)\to \R$ given  by
$$ \Phi(w,z, v, u) = F(v,u_m,u_n,u_h)-\JE(v-w)-\JCh z(v-\Vrev) $$
and, for each  channel  type  $x=m,n,h,y$,  we define  functions $b_x, a_x: \R\times [0,1]^4\to \R$    by 
\begin{equation*}
\begin{split}
b_x(v,u):  = &  \rho_x(v)(1-u_x)  -\zeta_x(v) u_x 
\quad  \mbox{ and } \\
a_x(v,u):= & ( \rho_x(v)(1-u_x) + \zeta_x(v) u_x) \chi(u_x)\, \\
\end{split}
\end{equation*}
(that is,  $ \sigma_x^2(v,u_x)=   \sigma^2a_x(v,u)$).
\item Given $\mu\in {\cal P}_2(\R\times [0,1]^4)$ we write
 \begin{align*}
\langle \mu^V\rangle &=\int_{\R^5}{v\mu(dv,d u_m,du_n,du_h,du_y)}\in \R , \\
\langle \mu^x\rangle  &=\int_{\R^5}{u_x\mu(dv,,d u_m,du_n,du_h,du_y)}\in [0,1] \mbox{ for } x=m,n,h,y  \mbox{ and }\\
\langle \mu\rangle  &= (\langle \mu^V\rangle, (\langle \mu^x\rangle )_{x=m,n,h,y }) \in \R\times[0,1]^4.
\end{align*}
\item Finally, with $\delta_x$  denoting the Dirac mass at $x\in \R\times [0,1]^4$, we write   
\begin{equation}\label{eq:empirical_measure} 
\mu_t^N:=\frac{1}{N}\sum_{i=1}^N \delta_{X_t^{(i)}}\in  {\cal P}_2(\R\times [0,1]^4)
\end{equation}
for the empirical measure of  system \eqref{eq:HH-model}   at time $t\geq 0$. 
\end{itemize}

\begin{theo}\label{theo:propchaos_McKeanVlasov_eq}
Assume Hypothesis \ref{hyp:V0}  and  moreover that  for all $N\geq 1$, $(X_0^{(1)},\dots, X_0^{(N)})$ are i.i.d. random vectors with  (compactly supported) common law $\mu_0\in  {\cal P}(\R\times [0,1]^4)$ not depending on $N$. 
\begin{itemize}
\item[a)]   For each $T>0$, the process $(\mu_t^N:t\in [0,T])$ converges in law on $C([0,T]; {\cal P}_2(\R\times [0,1]^4)) $,  when $N$ tends to $\infty$, to a deterministic and uniquely determined flow of probability measures  $(\mu_t:t\in [0,T])$ having uniformly bounded compact support. Moreover $(\mu_t: \, t \geq 0)$ in  $C(\mathbb{R}^+;  {\cal P}_2(\R\times [0,1]^4))$  is a   global  solution (in the \new{ sense} of distribution)  of the non linear McKean-Vlasov Fokker Planck equation
\begin{equation}\label{eq:Non-Linear-PDE}
\partial_t\mu_t  = \partial_v\left(\Phi (\langle \mu_t^V\rangle ,\langle \mu_t^y\rangle ,  \cdot,\cdot)\mu_t \right)+ \sum_{x=m,n,h,y}\frac{1}{2}\sigma^2 \, \partial_{u_x u_x}^2 \left(a_x\mu_t\right)-\partial_{u_x}\left( b_x\mu_t\right)
\end{equation}
with initial condition $\mu_0$. 
\item[b)] There is a  constant $C(T)>0$ depending on $\Vmax_0$ , $T>0$ and on the coefficients of system \eqref{eq:HH-model}, but not on $N$,  such that
\begin{equation}\label{estimatepropchaos}
\sup_{t\in [0,T]}  \E\left(\mathcal{W}_2^2(\mu^N_t,\mu_t )  \right) \leq C(T) N^{-2/5}.
\end{equation}
\item[c)] If additionally functions $\rho_x$ and $\zeta_x$ are of class $C^2(\R)$, (or  of class $C^1(\R)$ when $\sigma =0$),    $(\mu_t: \, t \geq 0) $ given in part a) is the unique weak solution of \eqref{eq:Non-Linear-PDE}   with  initial condition $\mu_0$ which has   supports  bounded uniformly in time. 

\end{itemize}
\end{theo}
 
\begin{rem}
~
\begin{itemize} 
\item[i)]   We have not  been able to prove  uniqueness   of weak  (measure) solutions to  \eqref{eq:Non-Linear-PDE} in full generality.  However, the global weak solution of \eqref{eq:Non-Linear-PDE} given by Theorem \ref{theo:propchaos_McKeanVlasov_eq}  a) is uniquely determined. 

\item[ii)] Classically (see \cite{meleard1996asymptotic}, \cite{sznitman1991topics}),  convergence  in law of $\mu_t^N$ to $\mu_t$ for fixed $t\geq 0$ implies the asymptotic independence as $N\to \infty$  of any  subfamily  $(X_t^{(1)},\dots, X_t^{(k)})$ of fixed size $k\leq N$ of system  \eqref{eq:HH-model}  (the propagation of chaos property). Somewhat counterintuitively, this is not  incompatible with part a) of  Theorem \ref{theo:synchro-empiricalVar}, even when $\sigma^2=0$.

\item[iii)]  Parts  a) and b) of   Theorem \ref{theo:propchaos_McKeanVlasov_eq}  also hold for general exchangeable $\mu_0$-chaotic initial conditions  $(X_0^{(1)},\dots, X_0^{(N)})$ (that is, such that  $\mu_0^N$ converges in law to $\mu_0$ on $  {\cal P}_2(\R\times [0,1]^4)$),  in which case one must add a term of the form $ C \E\left(\mathcal{W}_2^2(\mu^N_0,\mu_0 )  \right) $ on the right hand side of  \eqref{estimatepropchaos}. 
\item[iv)]  The first assertion in Theorem \ref{theo:propchaos_McKeanVlasov_eq}  would be standard   if the coefficients  in \eqref{eq:HH-model}  were   globally Lipschitz.  Under the key Hypothesis \ref{hyp:V0} we will be able to reduce the proof to the Lipschitz case. Moreover, this assumption  will allow us to  take full advantage of the estimates for empirical measures of i.i.d.\ samples proved in Fournier and Guillin \cite{fournier-guillin2013},  from where the convergence rate  \eqref{estimatepropchaos}  will stem.
\end{itemize}
\end{rem}

\medskip
Equation \eqref{eq:Non-Linear-PDE} can be interpreted as the dynamical description  of  a system of  infinitely many HH neurons  in mean field interaction.  Thanks to Theorem \ref{theo:propchaos_McKeanVlasov_eq}  and  to the  uniformity  in $N$  of the results in Theorem \ref{theo:synchro-empiricalVar}, we can  now  finally  transfer our synchronization results to this infinite dimensional setting. For each $t_1\geq 0$, define    $(\hX^{t_1,\infty}_t:\,  t \geq t_1)= \left(  (\hV^{\infty}_t,\hm_t^{\infty},\hn_t^{\infty},\hh_t^{\infty},\hy_t^{\infty}): \,  t\geq t_1\right) $ as  the solution of the ordinary differential equation
\begin{equation}\label{eq:HatProcess_infty} 
\begin{aligned}
\hV_t^{\infty} &= \hV_{t_1}^{\infty}+ \int_{t_1}^{t}{F(\hV^{{\infty}}_s,\hm_s^{\infty},\hn_s^{\infty},\hh_s^{\infty})-\JCh \hy_s^{\infty}(\hV^{{\infty}}_s-\Vrev) ds},\\
 \hx^{{\infty}}_t& = \hx^{{\infty}}_{t_1}+\int_{t_1}^{t}\rho_x(\hV^{{\infty}}_s)(1-\hx^{{\infty}}_s)  -\zeta_x(\hV^{{\infty}}_s)\hx^{{\infty}}_sds ,\;\;x=m,n,h,y \, ,
\end{aligned}
\end{equation}
with deterministic initial condition  $$\hX_{t_1}^{t_1, \infty} = \langle \mu_{t_1}\rangle,  $$ 
 where $(\mu_t: t\geq 0)\in C([0,\infty), {\cal P}_2(\R\times [0,1]^4)) $ is the global weak solution of  \eqref{eq:Non-Linear-PDE} with initial condition $\mu_0$ given by Theorem \ref{theo:propchaos_McKeanVlasov_eq} a).  We have:

\begin{corollary}\label{coro:synchro-McKeanVlasov}
Under the assumptions of Theorem \ref{theo:propchaos_McKeanVlasov_eq} and for the same  constants  as in Theorem \ref{theo:synchro-empiricalVar}, whenever   $\JE>\JE^0$ we have:
\begin{itemize}
\item[a)] For every $t\geq t_0$,
\begin{equation}\label{eq:dissipation_Variance_McKeanVlasov}
\mathcal{W}_2^2(\mu_t,\delta_{\langle \mu_t \rangle }) \leq \mathcal{W}_2^2(\mu_{t_0},\delta_{\langle \mu_{t_0} \rangle })  e^{-\lambda^0( t-t_0)} + \sigma^2 \frac{C^0_{\zeta,\rho}}{\lambda^0}.
\end{equation}
In particular, 
$\limsup_{t\to\infty}\mathcal{W}_2^2(\mu_t,\delta_{\langle \mu_t \rangle })   \leq \sigma^2 \frac{C^0_{\zeta,\rho}}{\lambda^0}.$
\item[b)] For every $t_1\geq t_0$  and $\delta \geq 0$ we have:
\begin{equation*}
\sup_{t_1\leq t\leq t_1+\delta}   \mathcal{W}_2^2(\mu_t,\delta_{ \hX^{t_1, \infty}_t  })   \leq   K_0  \wedge  2 \bigg[     \big( K_0' e^{-\lambda_0(t_1-t_0)} + \sigma^2  \frac{ C^0_{\zeta,\rho}}{\lambda^0}\big) (1+ \delta K_{\delta}) \bigg]  .
 \end{equation*} 
\end{itemize}
  \end{corollary}

\medskip 

We next present some numerical simulations which illustrate   the validity of our theoretical results (at least from a qualitative point of view)  and moreover we  explore the behavior of system \eqref{eq:HH-model} when  several neurons subpopulations are considered and when only chemical interaction is present. Furthermore, in view of the numerical experiments, we discuss some of the  limitations and possible extensions of our theoretical results.

\section{Numerical Experiments}\label{sec:numerical-experiments}
\new{
Inspired in Bossy et al.  \cite{Bossy:2016aa}, we have implemented numerical simulations of system  \eqref{eq:HH-model}  by means of  an Exponential Projective  Euler Scheme (EPES) which we next describe.
}

For a given time horizon $T>0$ and a natural number $M$, we consider the time grid $\{t_0=0,t_1=T/M,t_2 =2T/M,\ldots,t_k = kT/M,\ldots,t_M= T\}$. As initial condition for each neuron in the system we  consider independent random variables, uniformly distributed on $[-100,100]\times[0,1]^4$. Given the value of the system at $t_k$, the value for $\widehat{V}_{t_{k+1}}^{(i)}$ is computed as the exact solution to the ODE
\begin{equation*}
\begin{aligned}
 \widehat{V}_t^{(i)}=   \widehat{V}^{(i)}_{t_k} & + \int_{t_k}^{t}F( \widehat{V}_s^{(i)},  \widehat{m}_{t_k}^{(i)}, \widehat{n}_{t_k}^{(i)}, \widehat{h}_{t_k}^{(i)}) ds \\
 & - \int_0^t \frac{1}{N}\sum_{j=1}^{N}\JE ( \widehat{V}_s^{(i)}- \widehat{V}_{t_k}^{(j)}) -\frac{1}{N}\sum_{j=1}^{N}{\JCh  \widehat{y}^{(j)}_{t_k} ( \widehat{V}_s^{(i)}-\Vrev)} ds. 
\end{aligned}
\end{equation*}
which is indeed  a linear ODE since $F$ is linear in $V$. To compute  $\widehat{x}^{(i)}_{t_{k+1}}$ we first solve the SDE
\begin{equation*}
\begin{aligned}
  \cxi_t  =  \widehat{x}^{(i)}_{t_k} & +\int_{{t_k}}^{t}\rho_x( \widehat{V}^{(i)}_{t_k})(1- \cxi_s)  -\zeta_x( \widehat{V}^{(i)}_{t_k}) \cxi_sds\\
&+ \int_{t_k}^{t}{\sigma_x( \widehat{V}_{t_k}^{(i)}, \widehat{x}_{t_k}^{(i)})dW_s^{x,i}},\;\;x=m,n,h,y,
\end{aligned}
\end{equation*}
which corresponds to an Ornstein-Uhlenbeck process, so that $ \cxi_{t_{k+1}}$ can be exactly simulated. However, since  conditionally on $ \widehat{x}_{t_j}^{(i)}, j\leq k $ the law of $ \cxi_{t_{k+1}}$ is Gaussian,  $\{ \cxi_{t_{k+1}} \notin [0,1]\}$ happens with positive probability, so we are led to define $ \widehat{x}^{(i)}_{t_{k+1}}$ by projecting $\cxi_{t_{k+1}}$ onto $[0,1]$, that is:
\begin{equation*}
\widehat{x}^{(i)}_{t_{k+1}} =\left\{
\begin{array}{cc}
0, &  \cxi_{t_{k+1}} \in(-\infty,0)\\
\cxi_{t_{k+1}}, &  \cxi_{t_{k+1}}\in[0,1]\\
 1, & \cxi_{t_{k+1}} \in(1,+\infty).
\end{array}
\right.
\end{equation*}

\new{
In Appendix \ref{sec:proof-numerical-scheme} we prove  the convergence in $L^2$-norm of the EPES applied to \eqref{eq:HH-model}. We also provide the rate of convergence which is $1/2$ as for the classical Euler scheme.  }

In our simulations we have used as cut-off function (see Hypothesis \ref{hyp:MainHypotheses}-2)
$$
\chi(u) = \left\{ 
\begin{array}{cc}
0.1\exp\left(\frac{-0.5}{ 1-(2u-1)^2}\right )& u\in(0,1)\\
0 & u\notin (0,1),
\end{array}
\right.
$$
whereas, the specific values of the constants we have used are given in Table \ref{tab:ParametersVoltage}, and the   rate functions $\rho_x$ and $\zeta_x$ are given in Table \ref{tab:RateFunctionsSingleNeuronModel} and  shown in Figure \ref{figure:RateFunctions}. Although our results hold  irrespectively  of the  value of input current $I$, we have taken in all simulations  $I=25$, in which case a noiseless single neuron with the chosen parameters has  a limiting  regime of sustained oscillations, see Figure \ref{figure:SingleNeuronDifferentImputCurrents}. 

\begin{table}[ht!]
\fontsize{9}{11}\selectfont
\centering
\begin{tabular}{|c|c|c|c|c|c|}
\hline
$\gNa$&$120\; [\text{mS}/\text{cm}^3]$&$\gK$&$36\; [\text{mS}/\text{cm}^3]$&$\gL$&$0.3\; [\text{mS}/\text{cm}^3]$\\ \hline
$\VNa$&$50\;[\text{mV}]$&$\VK$&$-77\;[\text{mV}]$&$\VL$&$-54.4\;[\text{mV}]$\\ \hline
\end{tabular}
\caption{\fontsize{9}{11}\selectfont Values for the constants in the function for $F$. Taken from \cite[p.23]{Ermentrout:2010aa} \label{tab:ParametersVoltage} }
\end{table}%

\begin{table}[ht!]
\fontsize{9}{11}\selectfont
\begin{center}
\begin{tabular}{|c|c|c|}	\hline
Channel type&$\rho_x(V)$&$\zeta_x(V)$\\ \hline
Sodium (Na) Activation Channels $m$&$\dfrac{0.1(V+40)}{1 - \exp\left( -\dfrac{V+40}{10}\right)} $& $4\exp\left( -\dfrac{V+65}{18}\right) $ \\ \hline 
Sodium (Na) Deactivation Channels $h$&$0.07\exp\left( -\dfrac{V+65}{20}\right) $& $ \dfrac{1}{1+\exp\left( -\dfrac{V+35}{10}\right)}$ \\ \hline 
Potassium (K) Activation Channels $n$ &$\dfrac{0.01(V+55)}{1-\exp\left( -\dfrac{V+55}{10}\right)} $& $0.125\exp\left( -\dfrac{V+65}{80}\right) $ \\ \hline 
\new{ Neurotransmitter Channels $y$}  & \new{ $ \dfrac{5}{1+\exp\left(-0.2(V-2.0)\right)} $}&\new{  $0.18 $}\\ \hline 
\end{tabular}
\end{center}
\caption{\fontsize{9}{11}\selectfont Rate functions for the dynamics of the channels. Taken from \cite[p.23]{Ermentrout:2010aa} for the Sodium and Potassium channels, and from page  \cite[pp.160,163]{Ermentrout:2010aa} for the neurotransmitter channel.
\label{tab:RateFunctionsSingleNeuronModel}}
\end{table}%

\begin{figure}[ht!]
    \centering
    \begin{subfigure}[t]{0.4\textwidth}
        \centering
        \includegraphics[width=\textwidth]{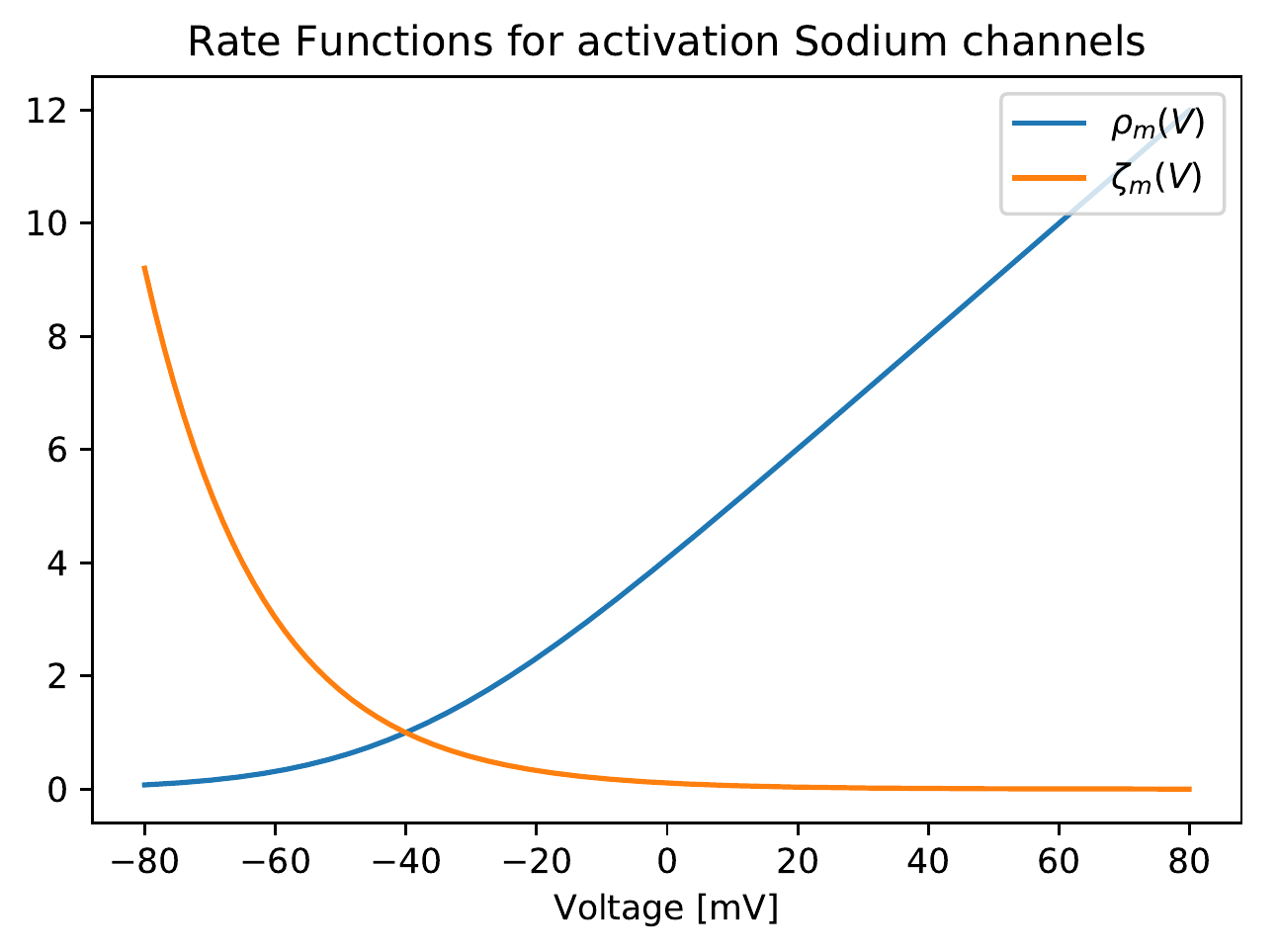}
        \caption{\fontsize{9}{11}\selectfont Activation Sodium Channels. \label{figure:RateFunctions-a}}
    \end{subfigure}
    \quad
    \begin{subfigure}[t]{0.4\textwidth}
        \centering
        \includegraphics[width=\textwidth]{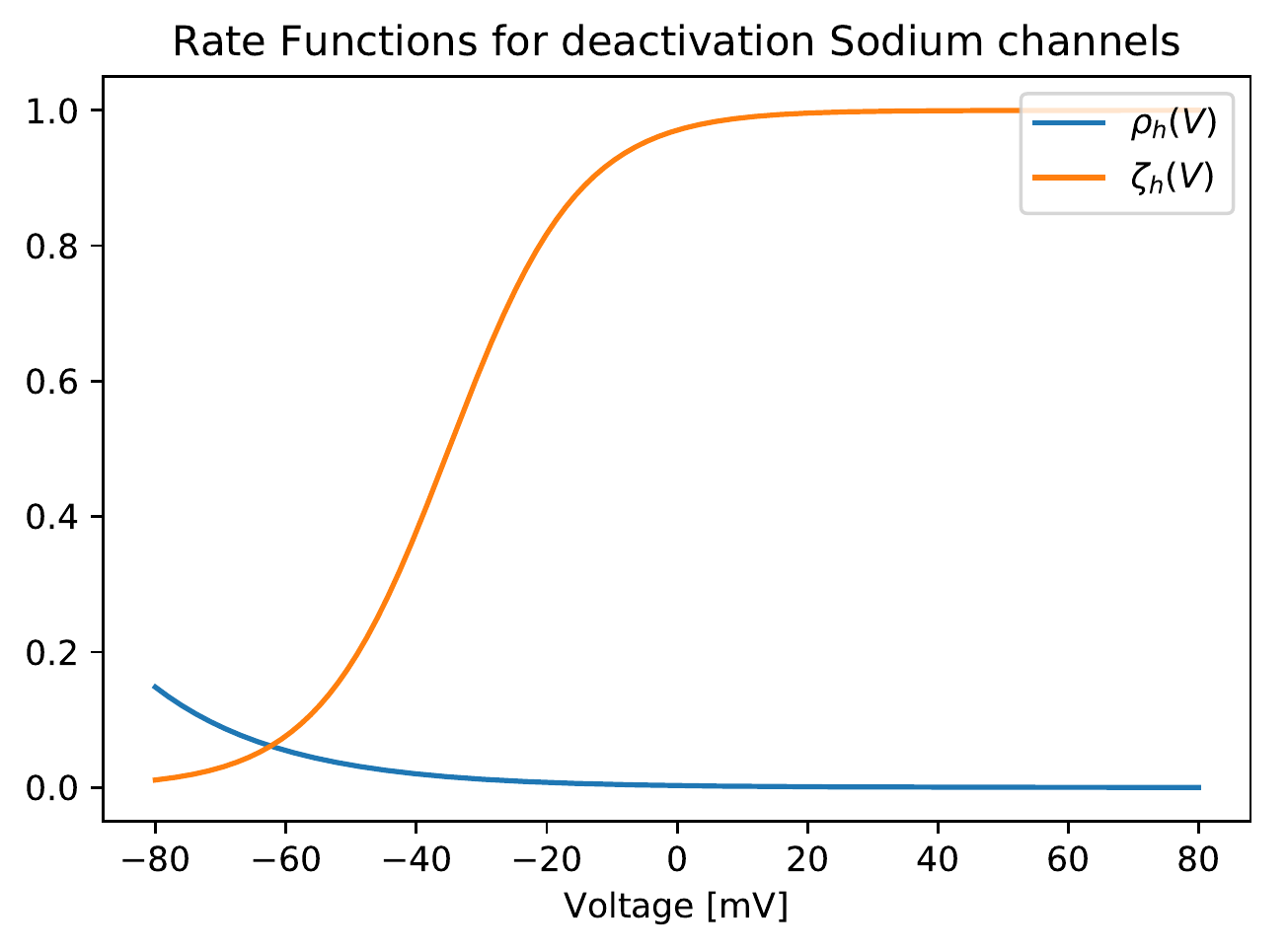}
        \caption{\fontsize{9}{11}\selectfont Deactivation Sodium Channels.\label{figure:RateFunctions-b}}
    \end{subfigure}
    \\
    \begin{subfigure}[t]{0.4\textwidth}
        \centering
        \includegraphics[width=\textwidth]{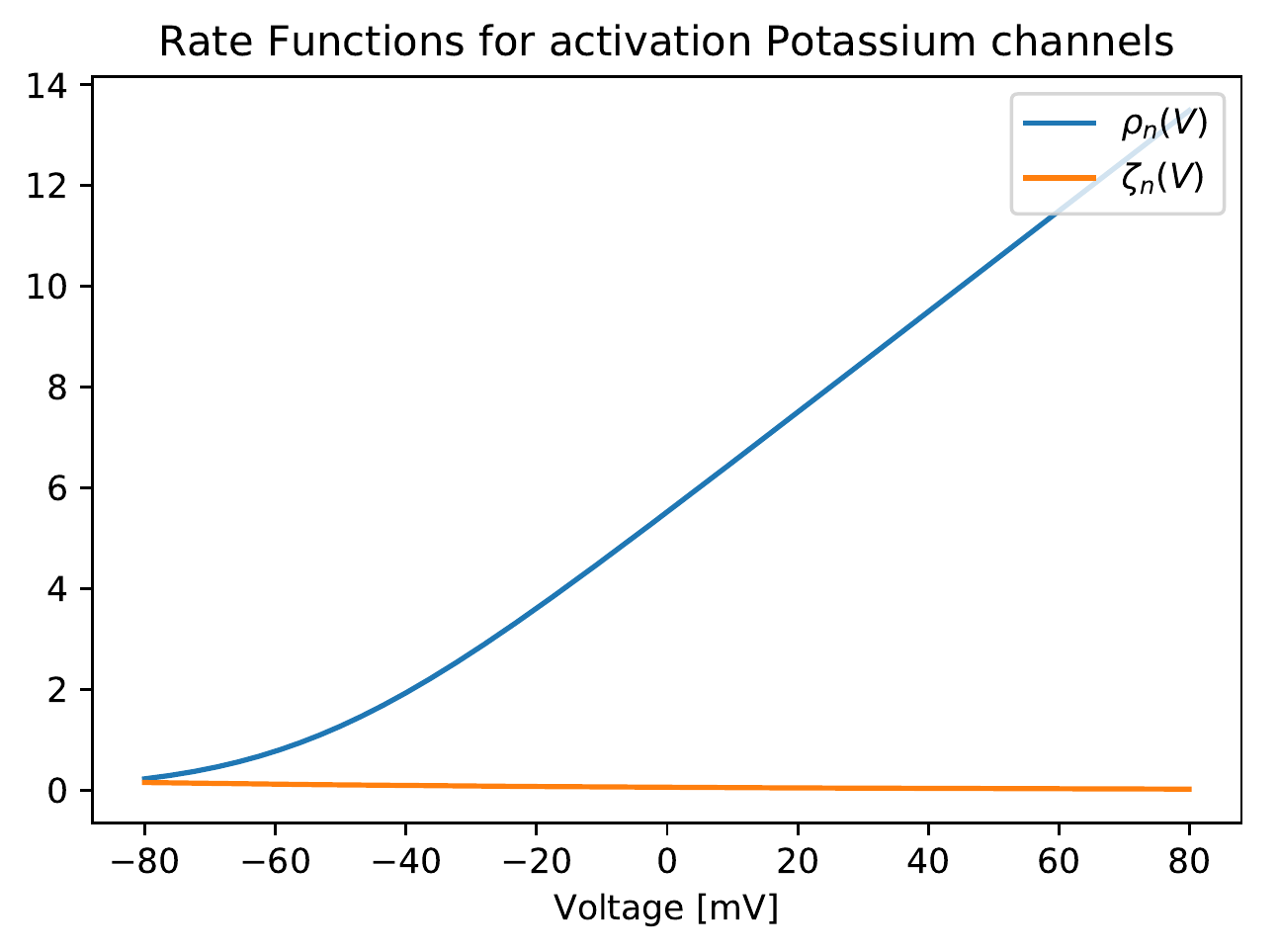}
        \caption{\fontsize{9}{11}\selectfont Activation Potassium Channels.\label{figure:RateFunctions-c}}
    \end{subfigure}
    \quad
    \begin{subfigure}[t]{0.4\textwidth}
        \centering
        \includegraphics[width=\textwidth]{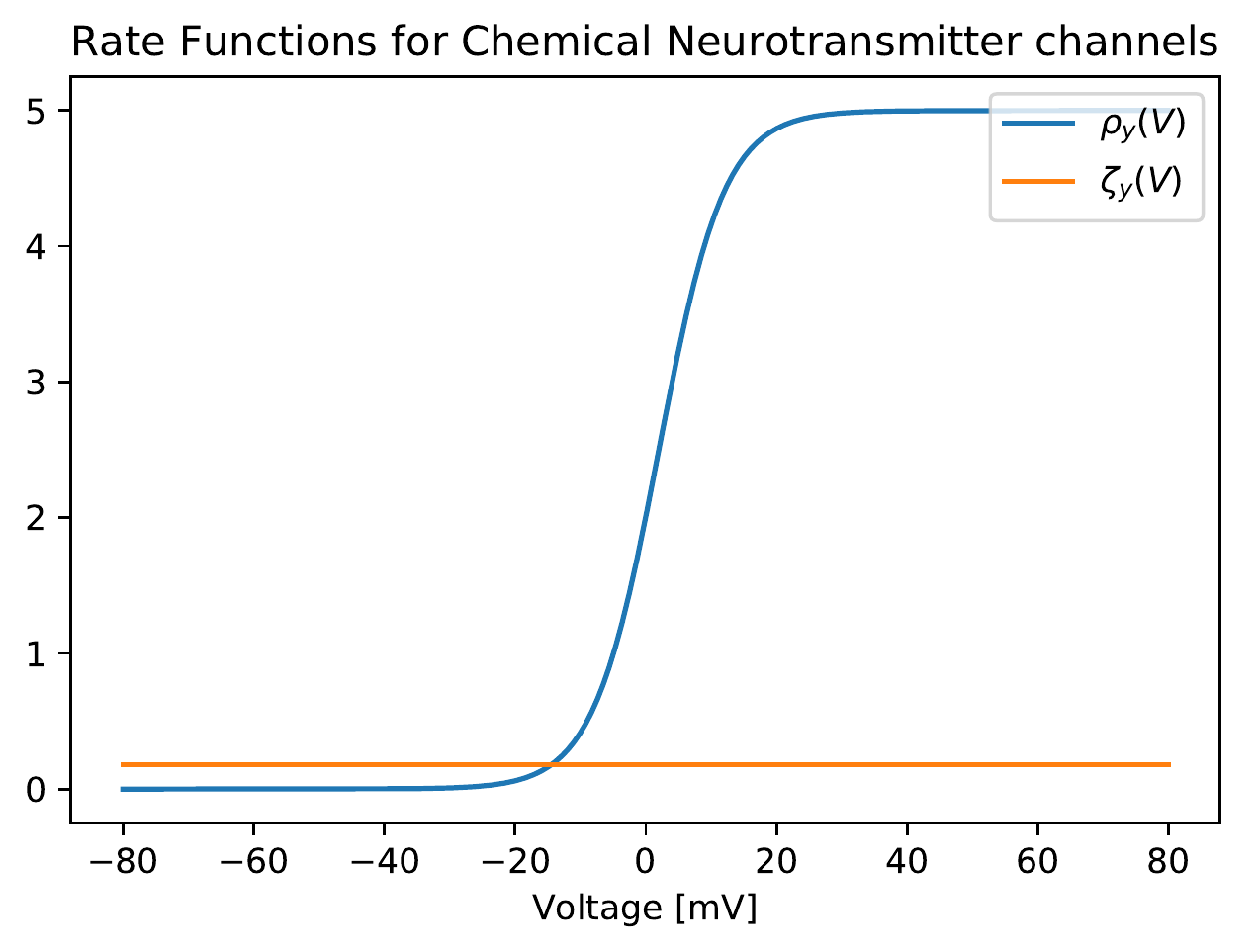}
        \caption{\fontsize{9}{11}\selectfont Synaptic Channels.\label{figure:RateFunctions-d}}
    \end{subfigure}
    \caption{\fontsize{9}{11}\selectfont Charateristic plot of rate functions $\rho_x$ and $\zeta_x$.   \label{figure:RateFunctions}}
\end{figure}

\begin{figure}[ht!]
    \centering
    \begin{subfigure}[t]{0.4\textwidth}
        \centering
        \includegraphics[width=\textwidth]{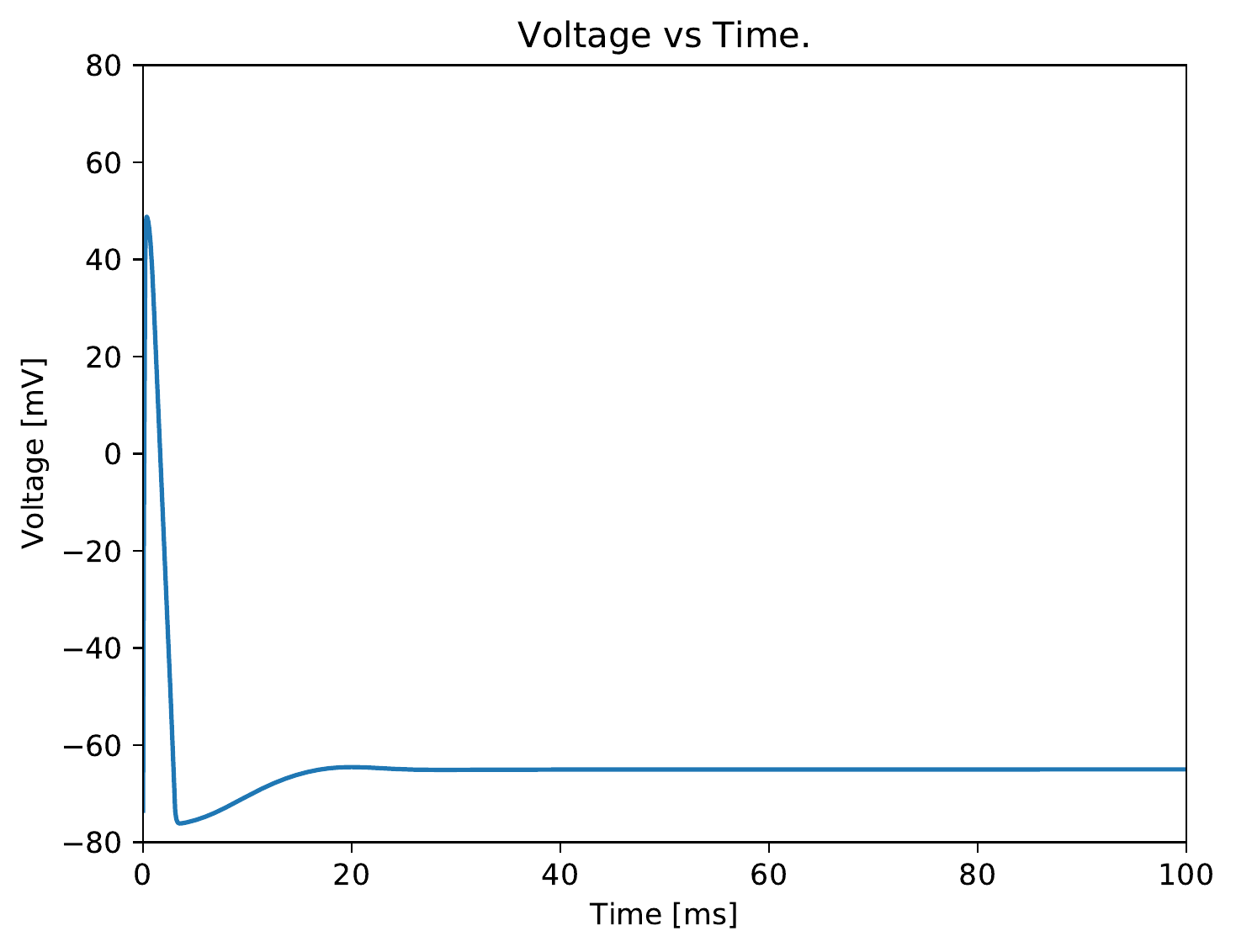}
        \caption{\fontsize{9}{11}\selectfont $I=0$. \label{figure:SingleNeuronDifferentImputCurrents-a}}
    \end{subfigure}
    \quad
    \begin{subfigure}[t]{0.4\textwidth}
        \centering
        \includegraphics[width=\textwidth]{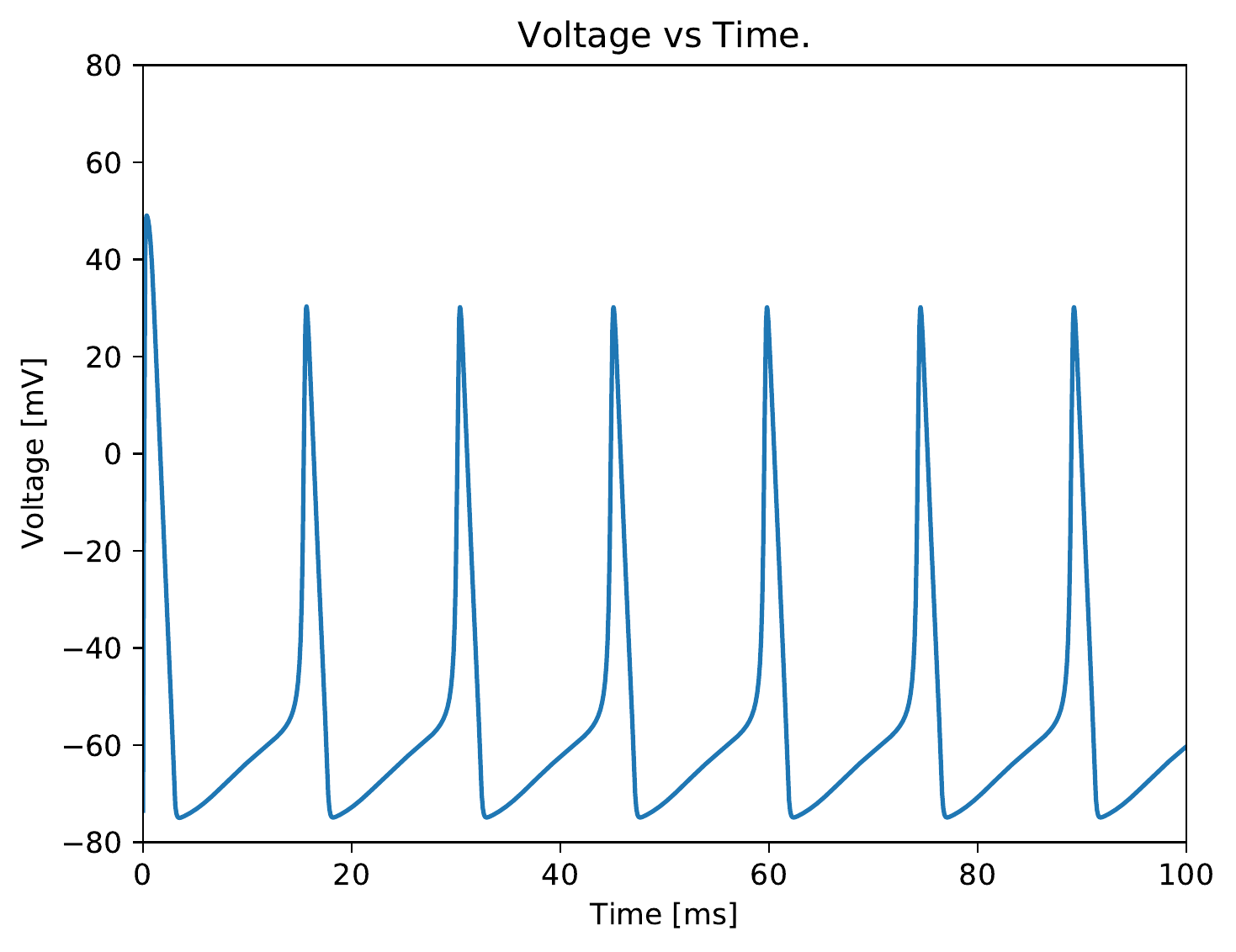}
        \caption{\fontsize{9}{11}\selectfont $I=10$.\label{figure:SingleNeuronDifferentImputCurrents-b}}
    \end{subfigure}
    \\
    \begin{subfigure}[t]{0.4\textwidth}
        \centering
        \includegraphics[width=\textwidth]{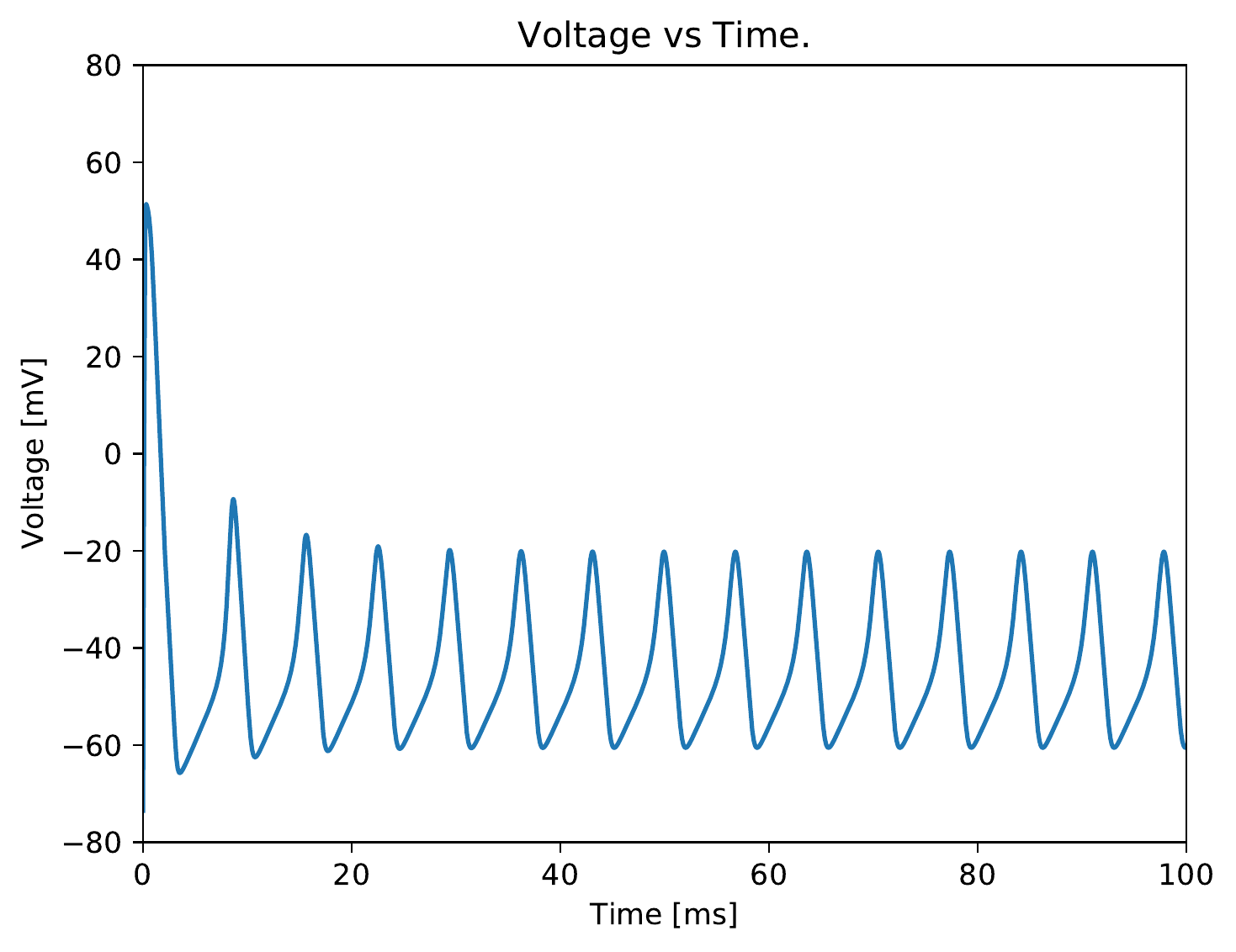}
        \caption{\fontsize{9}{11}\selectfont $I=100$.\label{figure:SingleNeuronDifferentImputCurrents-c}}
    \end{subfigure}
    \quad
    \begin{subfigure}[t]{0.4\textwidth}
        \centering
        \includegraphics[width=\textwidth]{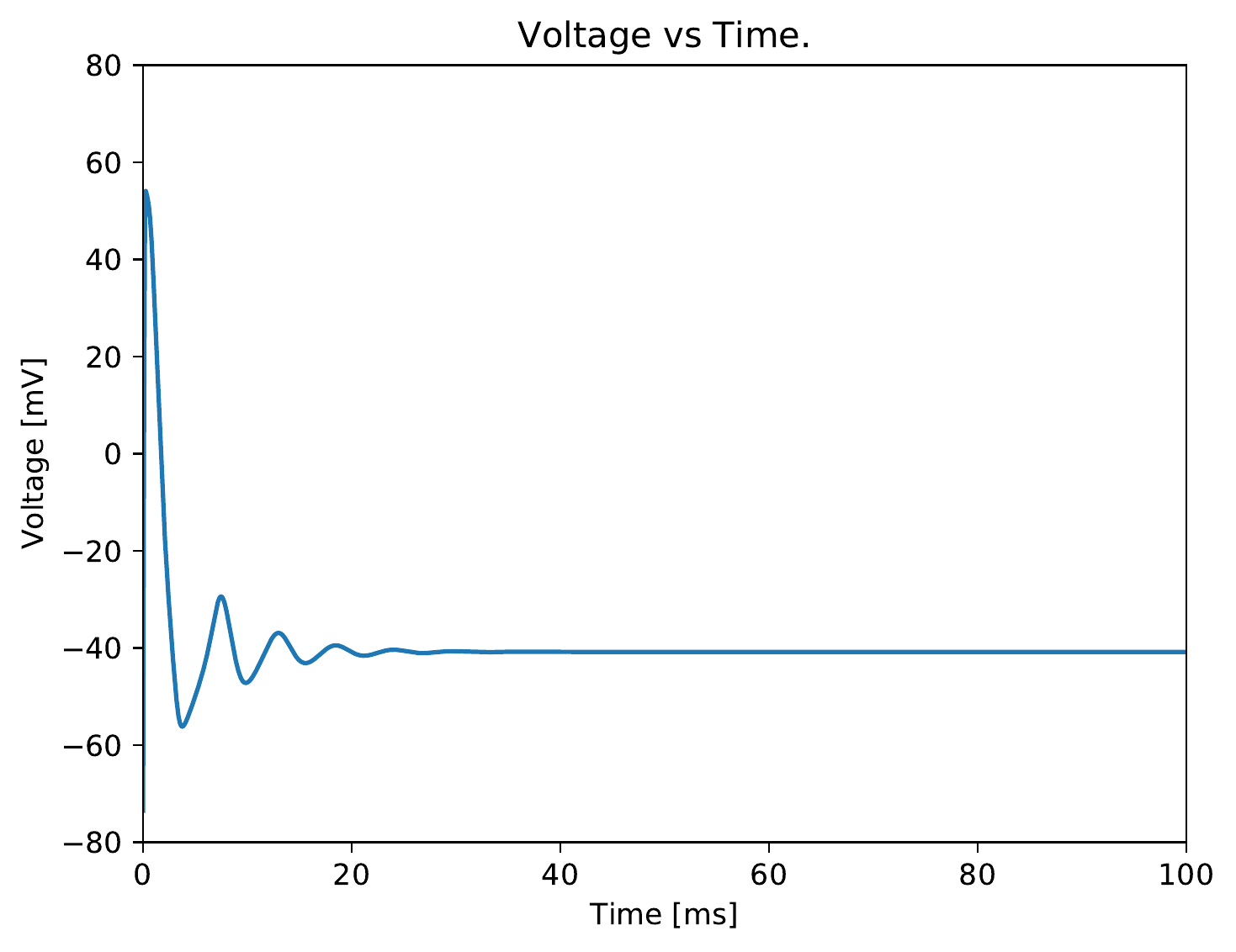}
        \caption{\fontsize{9}{11}\selectfont $I=200$.\label{figure:SingleNeuronDifferentImputCurrents-d}}
    \end{subfigure}
\caption{\fontsize{9}{11}\selectfont  Reponses of the model \eqref{eq:DeterministicHHmodel} depending on the input current $I$:  no oscillations if $I=0$ (a);  large amplitude and low frequency oscillations if $I=10$  (b); small \new{ amplitude} and  high frequency oscillations if $I=100$  (c);   damped oscillations if $I=200$  (d). \label{figure:SingleNeuronDifferentImputCurrents}}
\end{figure}

\subsection{Numerical experiments illustrating our theoretical  results }\label{numexpillustheo}
Our first numerical experiments illustrate the results of part a) in Theorem \ref{theo:synchro-empiricalVar}. In Figure \ref{figure:Samples-IE} we show one trajectory of the system \eqref{eq:HH-model} under purely electrical interaction, for different sizes of the network and levels of noise. The first row shows the trajectories of a network of $10$ neurons for $\sigma=0$, $\sigma=0.5$ and $\sigma=1$.  From the second to the fourth row, the trajectories of networks of $100$, $1000$ and $10000$ neurons, respectively, are shown. The scale of all plots is the same. 
We observe that the qualitative behavior in terms of $\sigma$ is the same for all rows: as expected, the noiseless network ultimately  reaches  perfect synchronization, whereas for $\sigma>0$ the trajectories of the neurons lie  in a band whose width increases with $\sigma$, as  predicted by Theorem   \ref{theo:synchro-empiricalVar}, a).   Moreover,  the speed at which synchronizations takes place does not depend on the size of the network nor on the level of noise. 

In our second experiment, we estimate the expected value of the empirical variance of a network of various sizes  and for different levels of noise. More precisely we estimate the mean  of
$$\varV_t 
= \frac{1}{N}\sum_{i=1}^N{\left( V_t^{(i)}-\meanV_t\right)^2},\;\;\;\varX_t = \frac{1}{N}\sum_{i=1}^N{\left( x_t^{(i)}-\meanX_t\right)^2}, \;\;x=m,n,h.$$
over   $50000$  Monte Carlo replica  for each value of  $\sigma\in\{0.1, 0.5, 1\}$ and $N\in\{10,100,1000,10000\}$ (we now use $\sigma=0.1$ instead of $\sigma=0$ since in the latter case  the obtained plot quickly becomes flat).  The computation of the empirical variance for each time step and  replica was  done using the corrected two-phase algorithm to avoid {\it catastrophic cancellations} (see \cite{Chan:1983aa}). The results of this experiment appear in Figure \ref{figure:DisipationEmpiricalVariance-IE}, where a different variable is presented in each row, from top to bottom: voltage ($V$), Sodium activation channels ($m$), Potassium channels ($n$) and Sodium deactivation channels ($h$). Each column corresponds to a different level of noise, increasing from $\sigma=0.1$ on the left,  to $\sigma=0.5$ in the middle and  to $\sigma=1$ on the right. In each subfigure we show the dissipation of the expected value of the empirical variance for networks of $10,\,100,\,1000$ and $10000$ neurons.  Just as for one trajectory of the system, we observe again a quick synchronization,  now  measured in terms of the average dispersion over many trajectories, at speed  which  does not depend on the noise or the size of the network, with the heights of  the peaks increasing with $\sigma$. Notice that double peaks are expected from  Figure \ref{figure:DisipationEmpiricalVariance-IE} already: even a small dispersion of the phase among different neurons can induce a high dispersion of their voltages and channels right before and after a potential spike is emitted. This  dispersion increases with $N$, but tends to stabilize as $N$ becomes large (notice that the red and green lines in Figure \ref{figure:DisipationEmpiricalVariance-IE} are indistinguishable),  consistently with Theorem \ref{theo:propchaos_McKeanVlasov_eq}. 

\smallskip 
From this last observation,  it is also interesting to point out  that the maximum variance over  time-windows of fixed length $\delta>0$  which drift to infinity cannot decrease for every possible  value of $\delta$, unless $\sigma^2=0$. Indeed, in the noisy case the voltage of significantly many neurons can in principle differ from the voltage of the underlying one-neuron dynamics, over time-windows  larger than its period, by as much as the whole asymptotic  range of the   voltages  dynamics. Thus,  albeit  not sharp, the estimates in part b) of Theorem  \ref{theo:synchro-empiricalVar} and   Corollary \ref{coro:synchro-McKeanVlasov} are qualitatively correct.

\medskip

 \begin{figure}[ht!]
    \centering
    \begin{subfigure}[t]{0.31\textwidth}
        \centering
        \includegraphics[width=\textwidth,height=4.8cm]{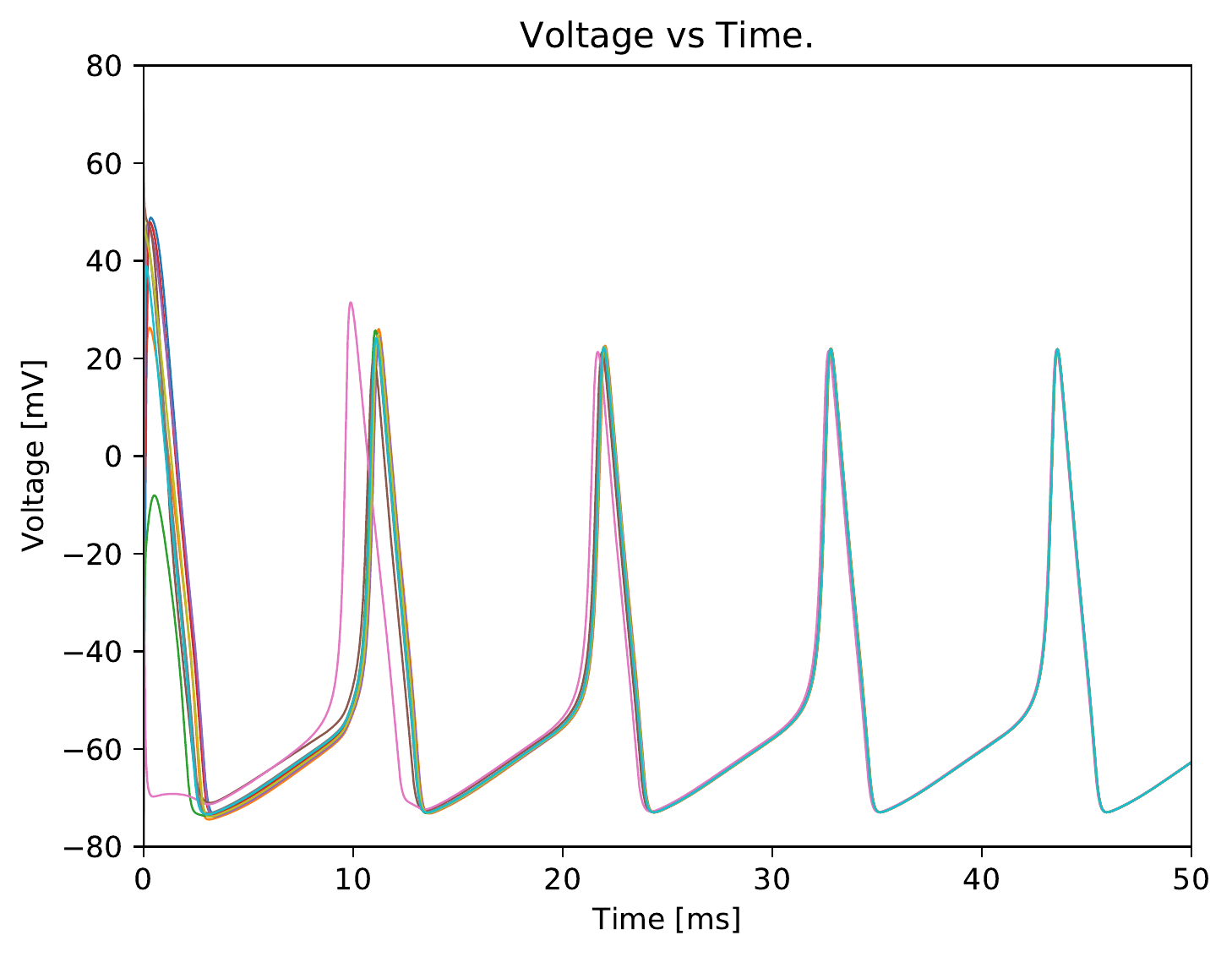}
        \caption{\fontsize{9}{11}\selectfont $N=10$, $\sigma=0$.}
    \end{subfigure}
    \quad
    \begin{subfigure}[t]{0.31\textwidth}
        \centering
        \includegraphics[width=\textwidth,height=4.8cm]{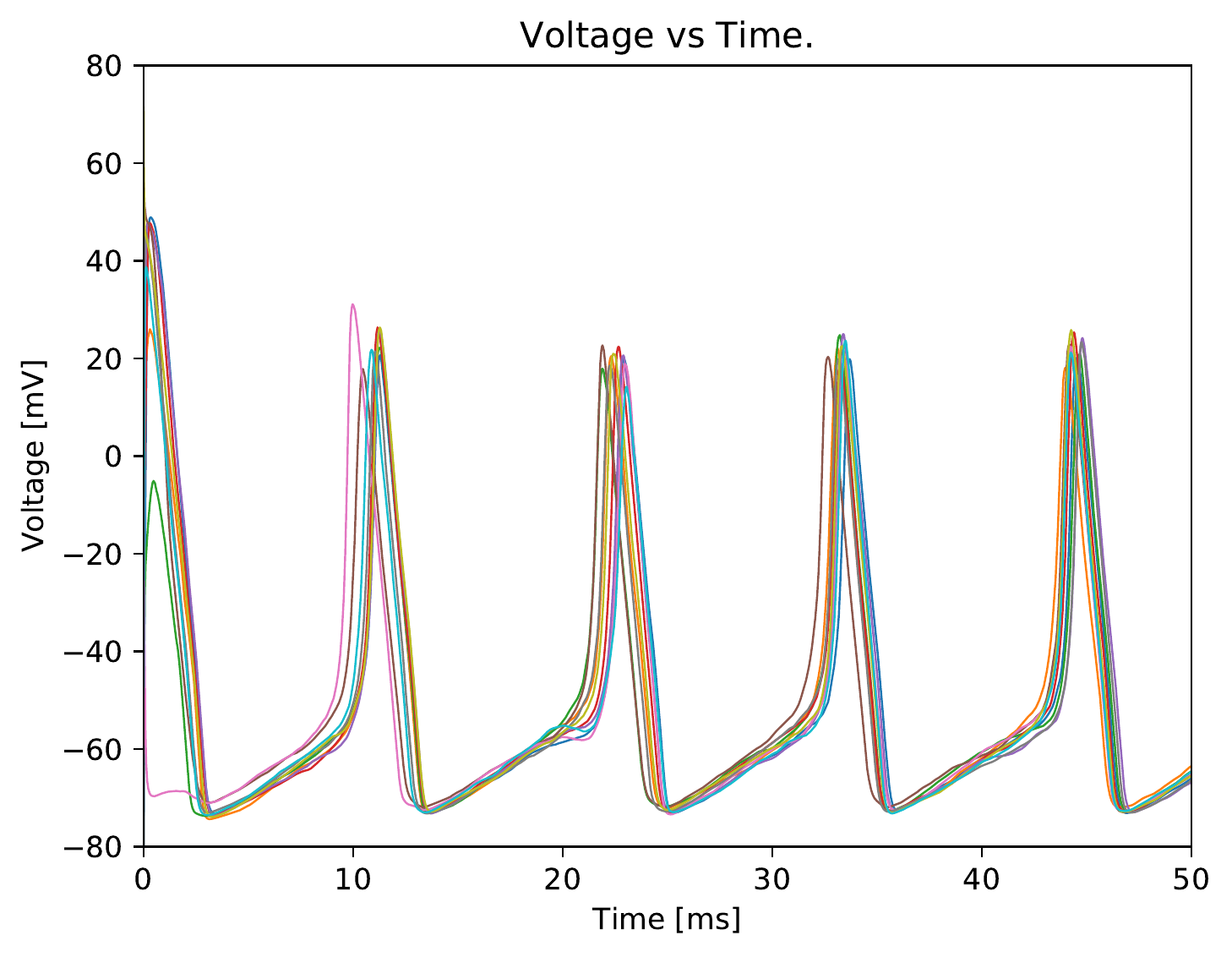}
        \caption{\fontsize{9}{11}\selectfont $N=10$, $\sigma=0.5$.}
    \end{subfigure}
    \quad
    \begin{subfigure}[t]{0.31\textwidth}
        \centering
        \includegraphics[width=\textwidth,height=4.8cm]{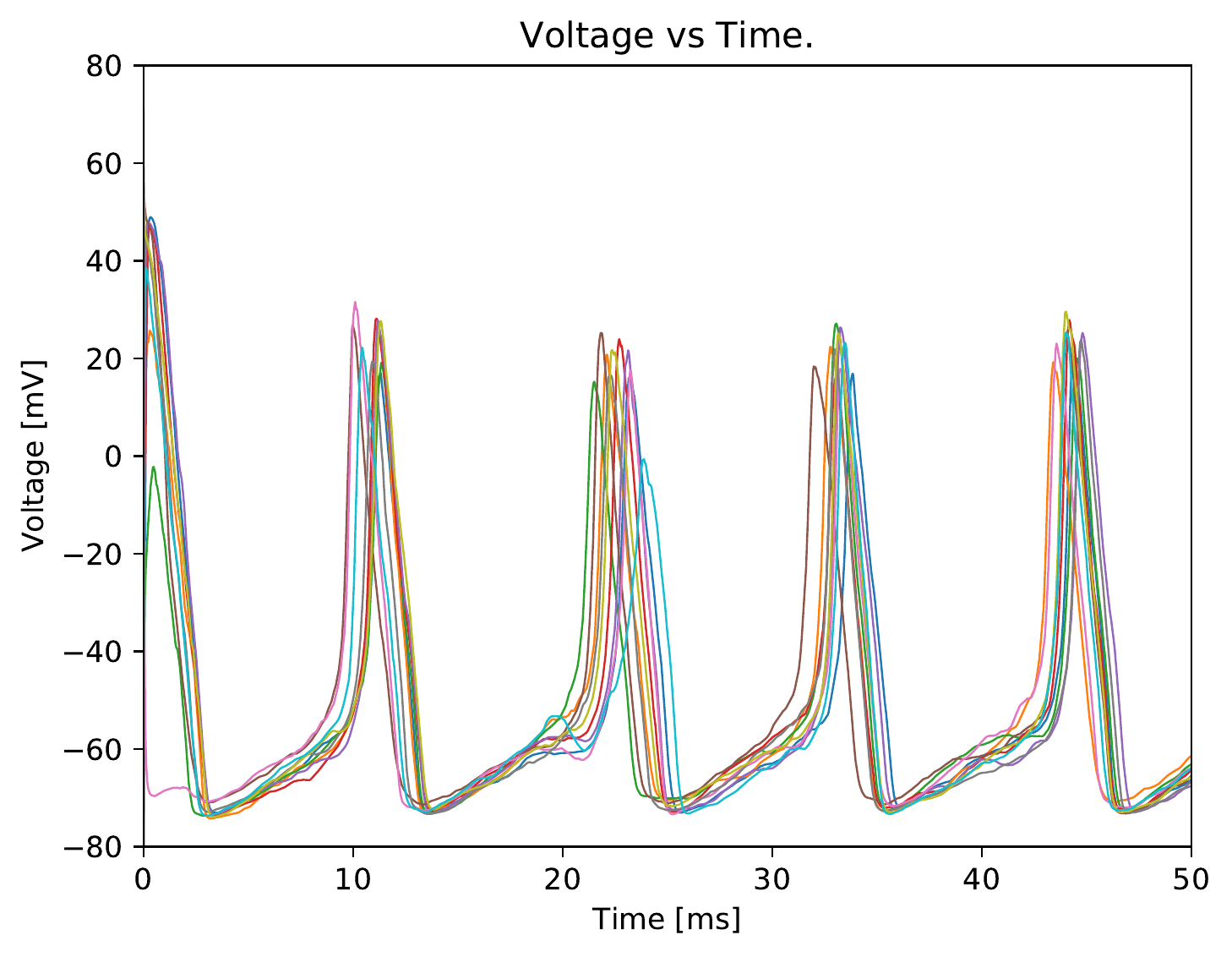}
       \caption{\fontsize{9}{11}\selectfont $N=10$, $\sigma=1$.}
    \end{subfigure}
    \\
       \begin{subfigure}[t]{0.31\textwidth}
        \centering
        \includegraphics[width=\textwidth,height=4.8cm]{figures/EI-100-sigma-0.pdf}
        \caption{\fontsize{9}{11}\selectfont $N=100$, $\sigma=0$.}
    \end{subfigure}
    \quad
    \begin{subfigure}[t]{0.31\textwidth}
        \centering
        \includegraphics[width=\textwidth,height=4.8cm]{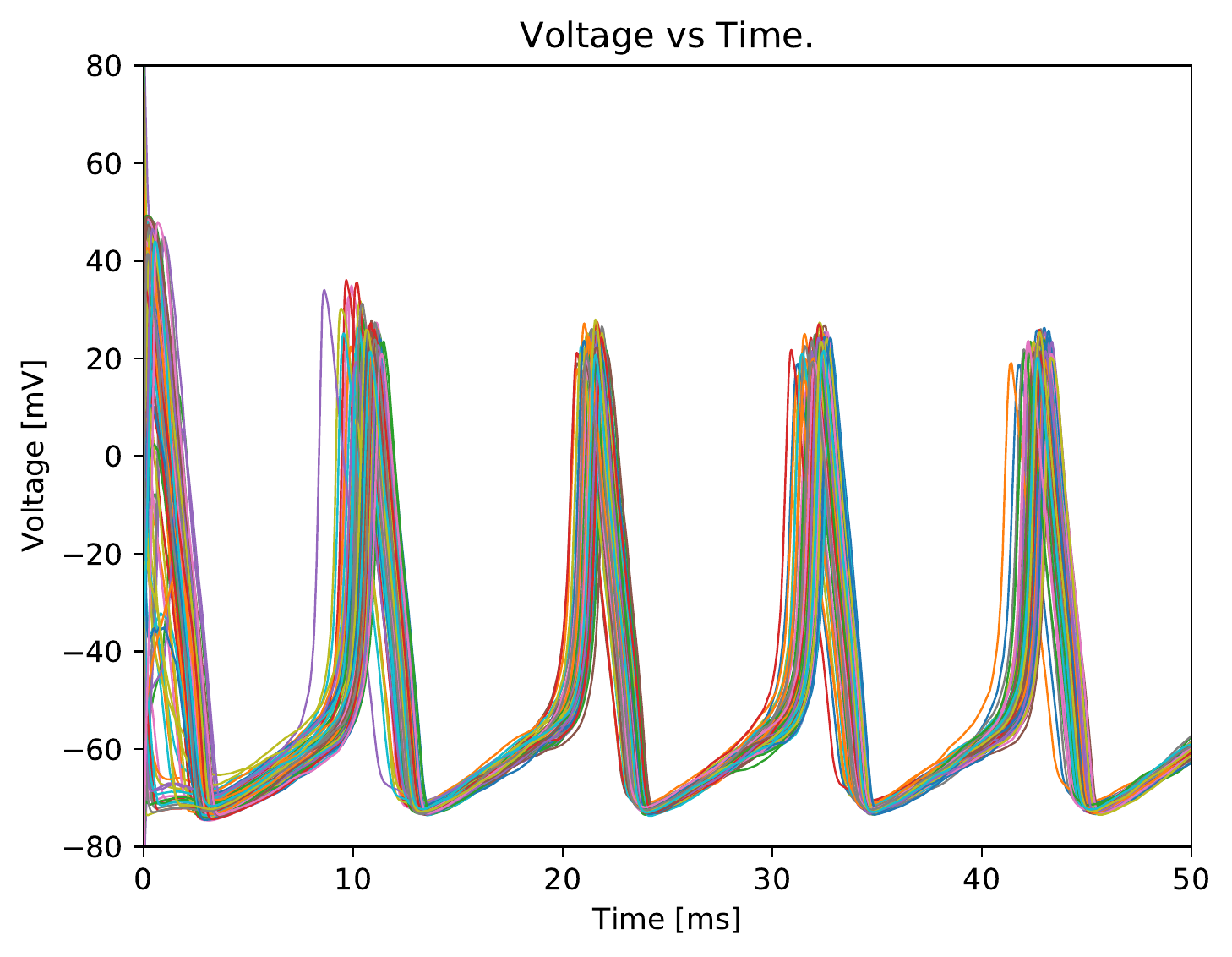}
        \caption{\fontsize{9}{11}\selectfont $N=100$, $\sigma=0.5$.}
    \end{subfigure}
    \quad
    \begin{subfigure}[t]{0.31\textwidth}
        \centering
        \includegraphics[width=\textwidth,height=4.8cm]{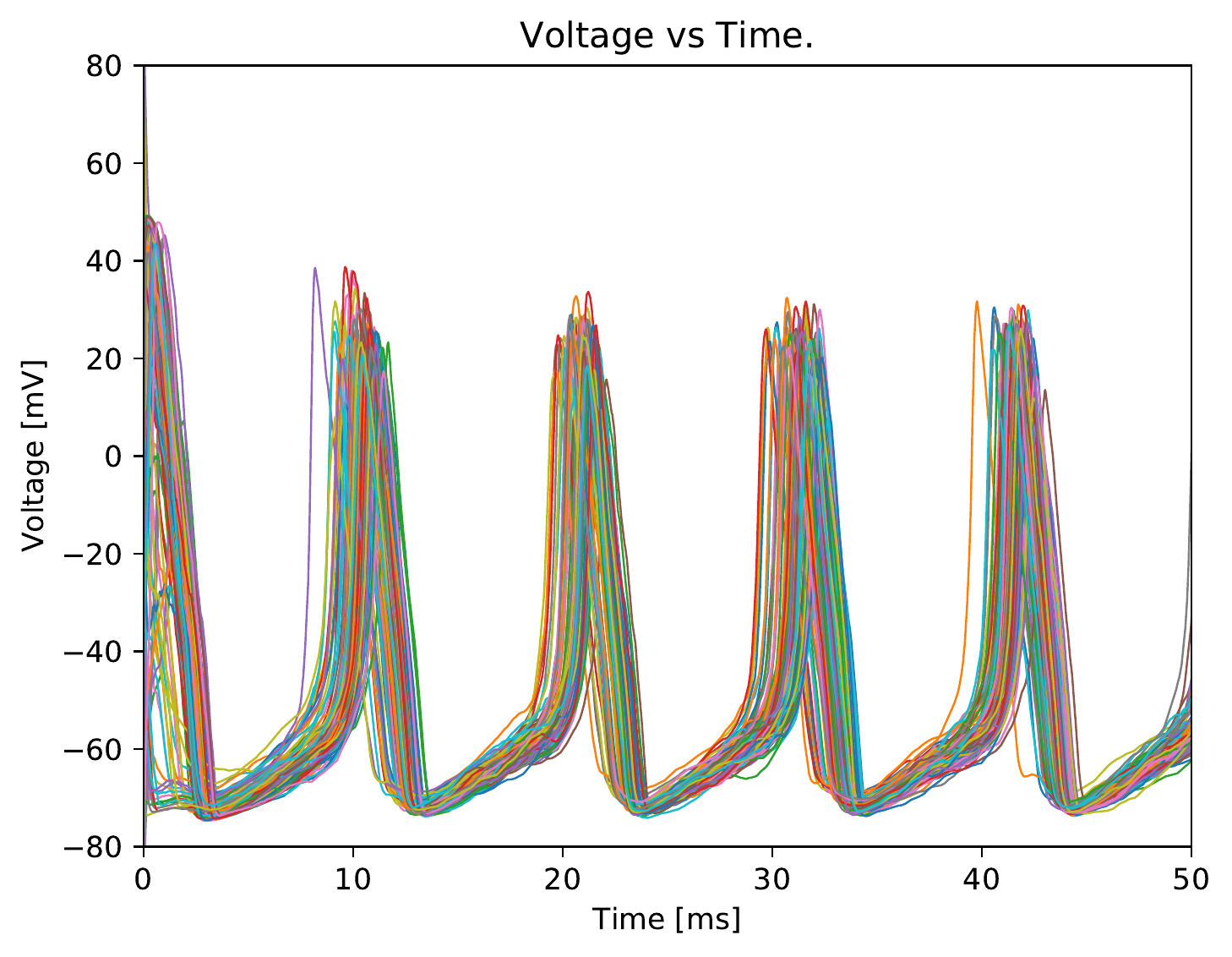}
        \caption{\fontsize{9}{11}\selectfont $N=100$, $\sigma=1$.}
    \end{subfigure}
    \\
       \begin{subfigure}[t]{0.31\textwidth}
        \centering
        \includegraphics[width=\textwidth,height=4.8cm]{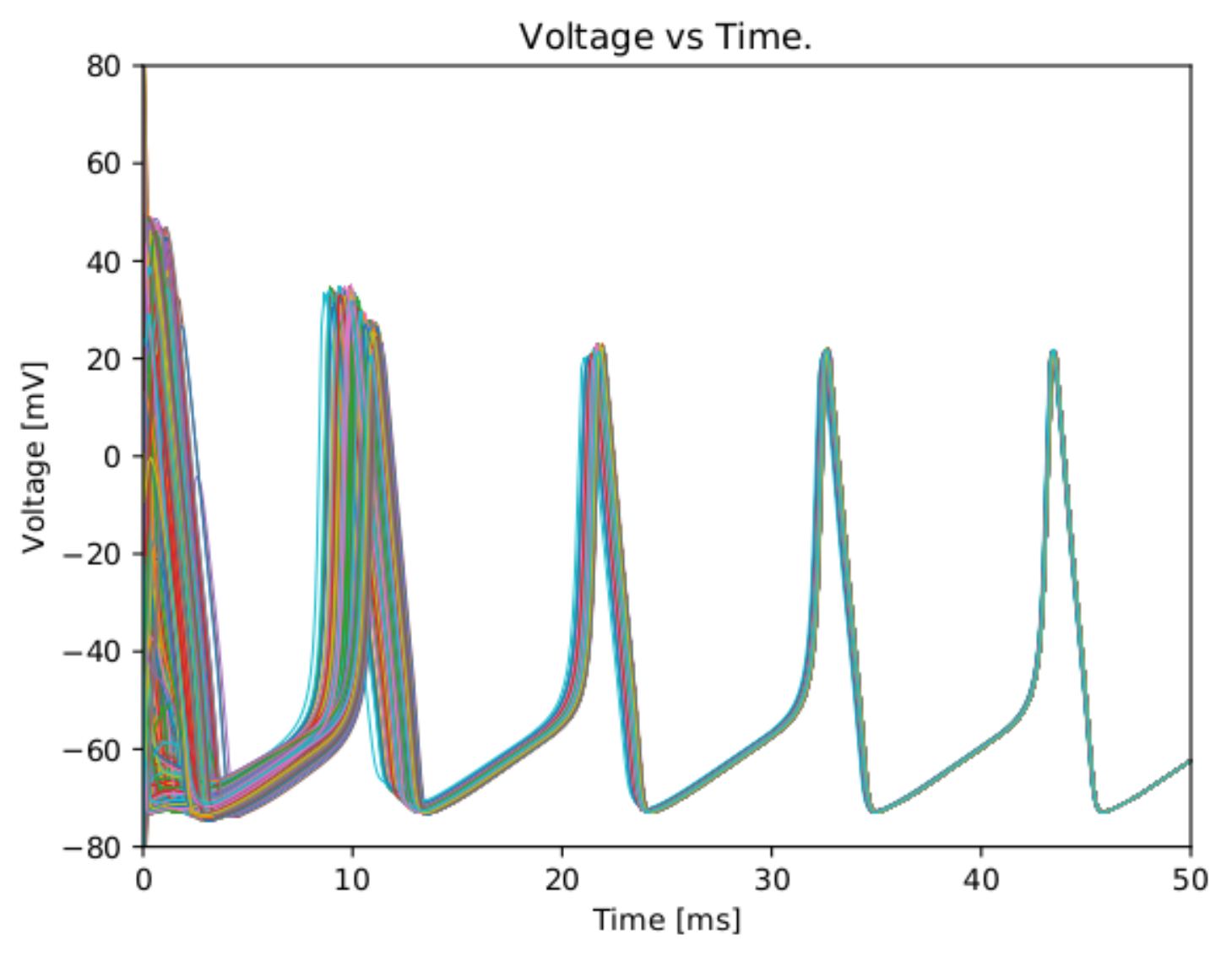}
        \caption{\fontsize{9}{11}\selectfont $N=1000$, $\sigma=0$.}
    \end{subfigure}
    \quad
    \begin{subfigure}[t]{0.31\textwidth}
        \centering
        \includegraphics[width=\textwidth,height=4.8cm]{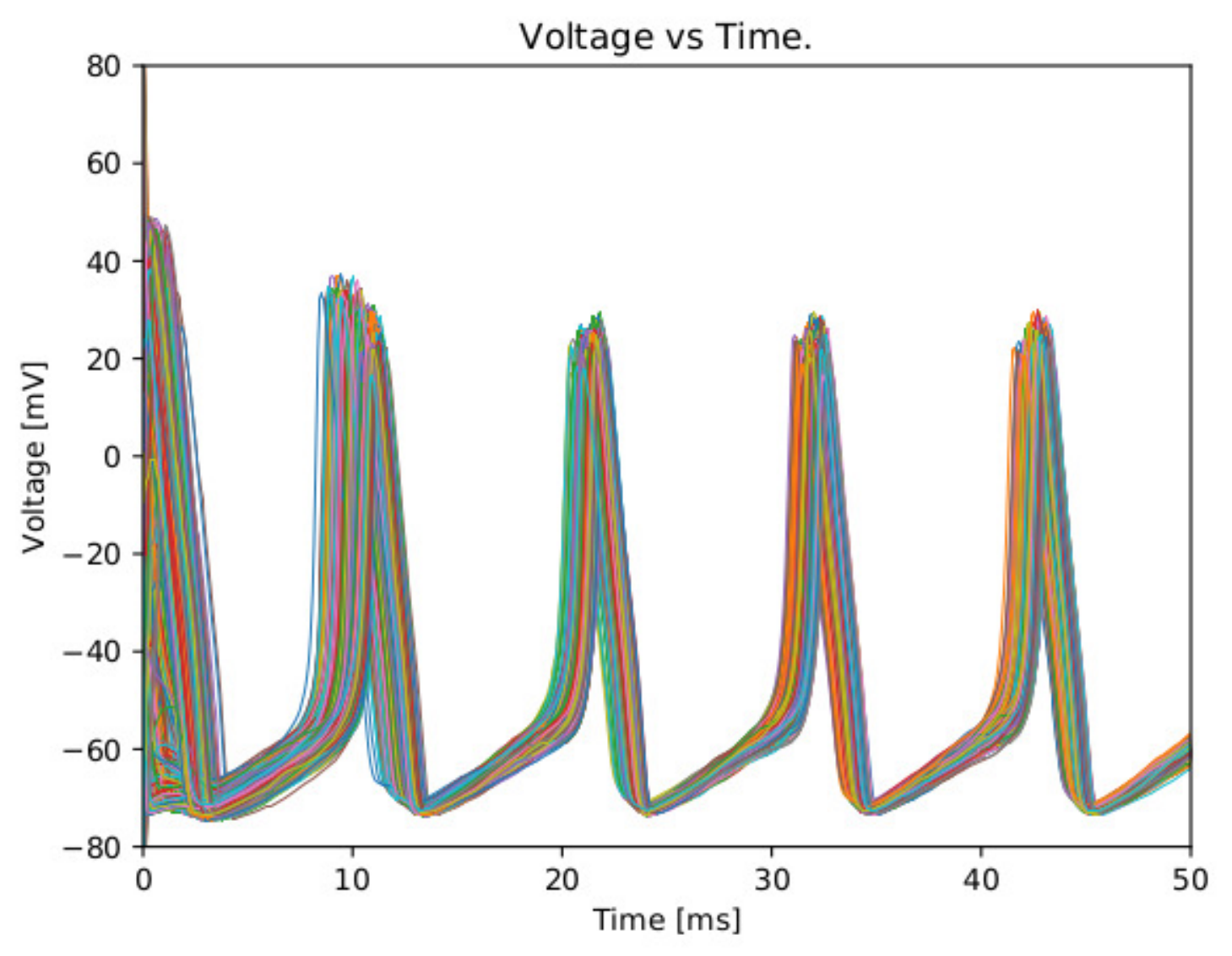}
        \caption{\fontsize{9}{11}\selectfont $N=1000$, $\sigma=0.5$.}
    \end{subfigure}
    \quad
    \begin{subfigure}[t]{0.31\textwidth}
        \centering
        \includegraphics[width=\textwidth,height=4.8cm]{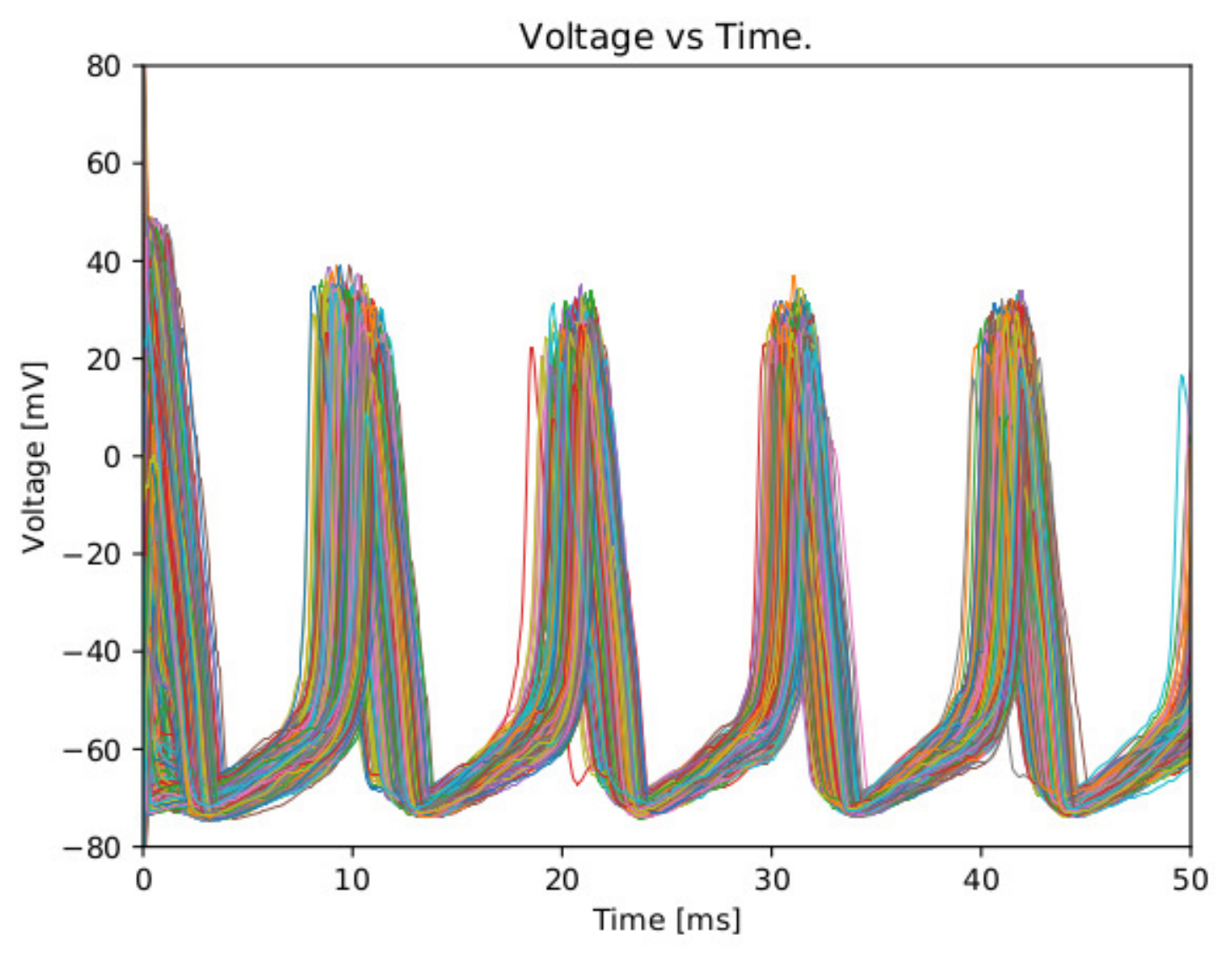}
        \caption{\fontsize{9}{11}\selectfont $N=1000$, $\sigma=1$.}
    \end{subfigure}
    \\
       \begin{subfigure}[t]{0.31\textwidth}
        \centering
        \includegraphics[width=\textwidth,height=4.8cm]{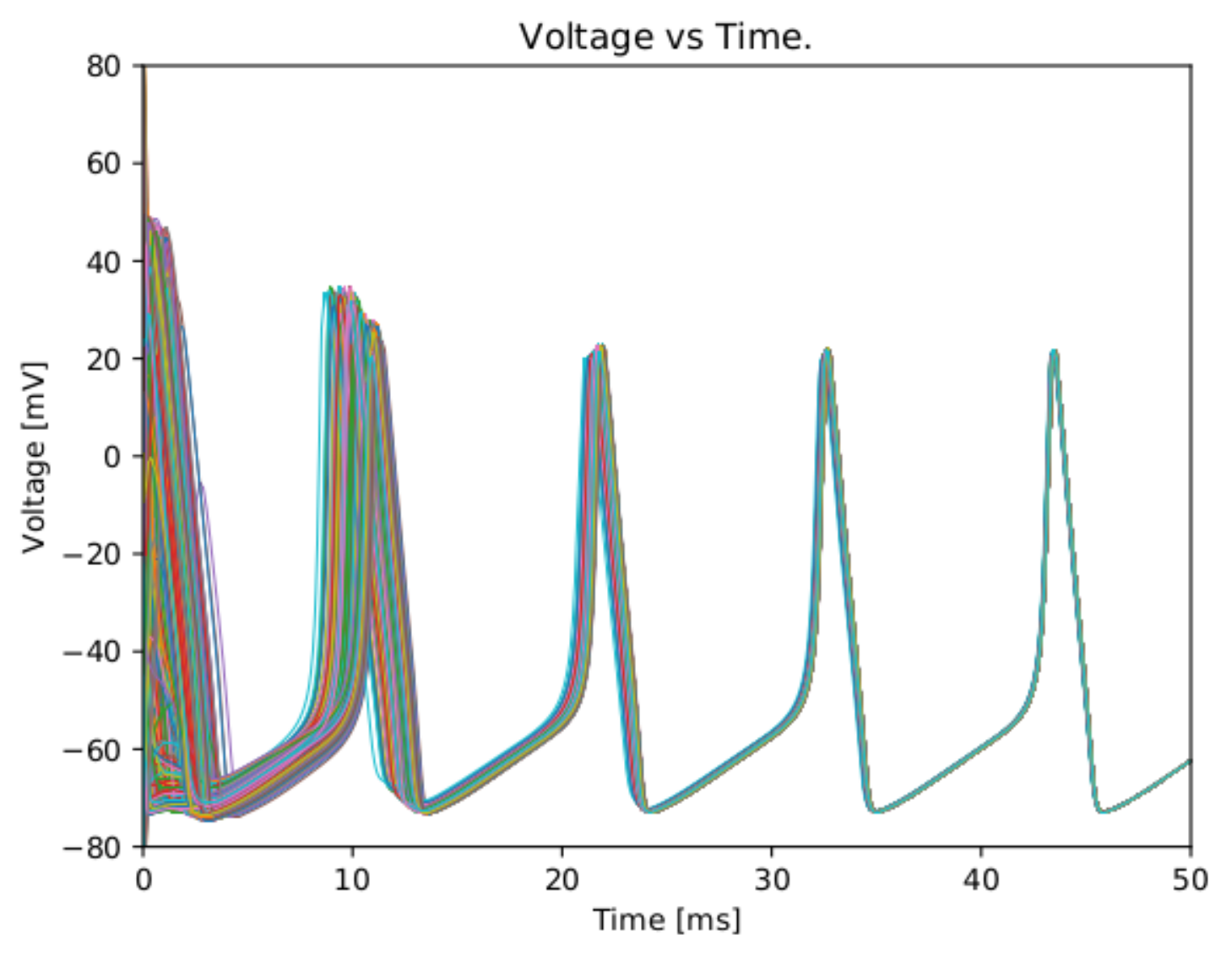}
       \caption{\fontsize{9}{11}\selectfont $N=10000$, $\sigma=0$.}
    \end{subfigure}
    \quad
    \begin{subfigure}[t]{0.31\textwidth}
        \centering
        \includegraphics[width=\textwidth,height=4.8cm]{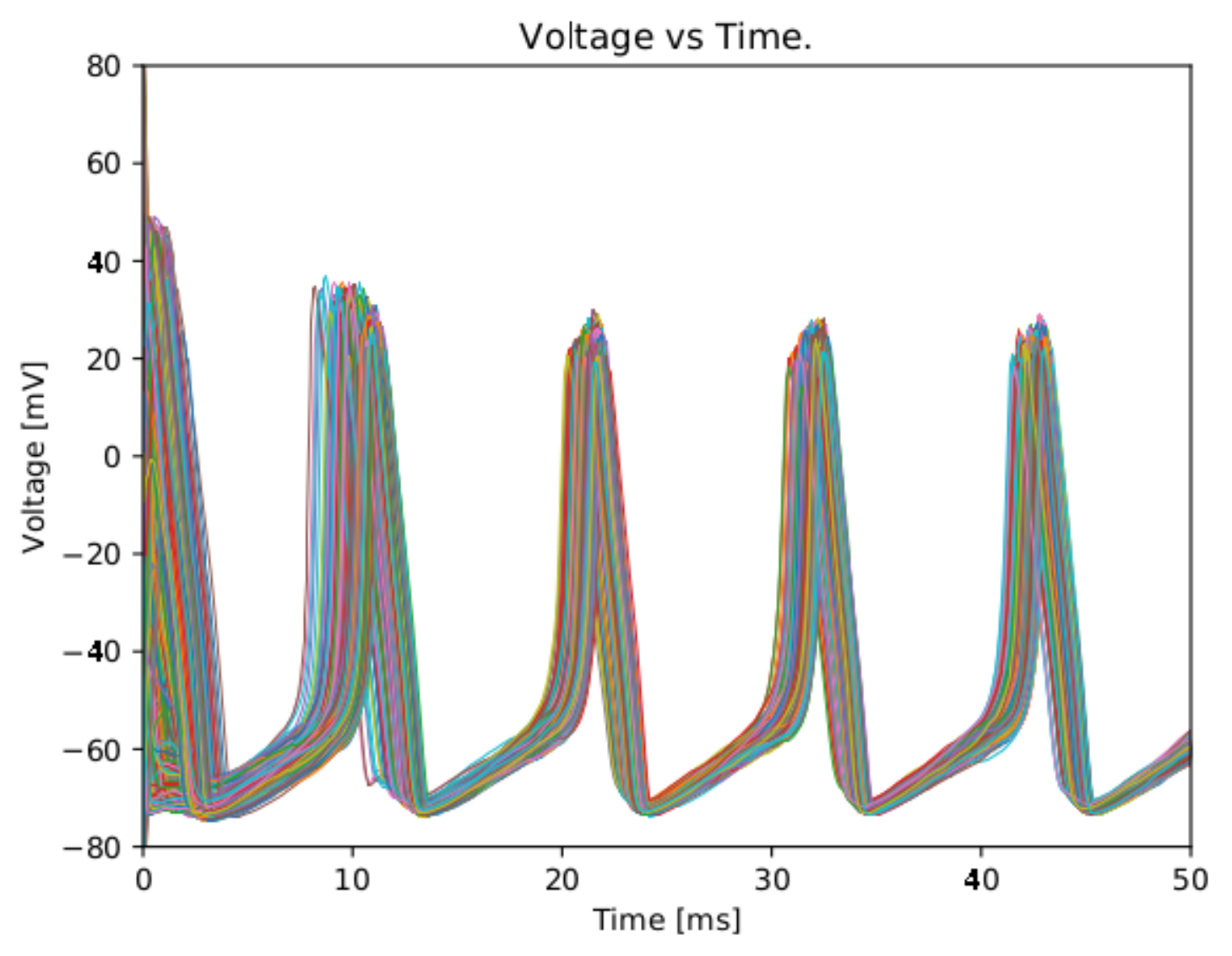}
        \caption{\fontsize{9}{11}\selectfont $N=10000$, $\sigma=0.5$.}
    \end{subfigure}
    \quad
    \begin{subfigure}[t]{0.31\textwidth}
        \centering
        \includegraphics[width=\textwidth,height=4.8cm]{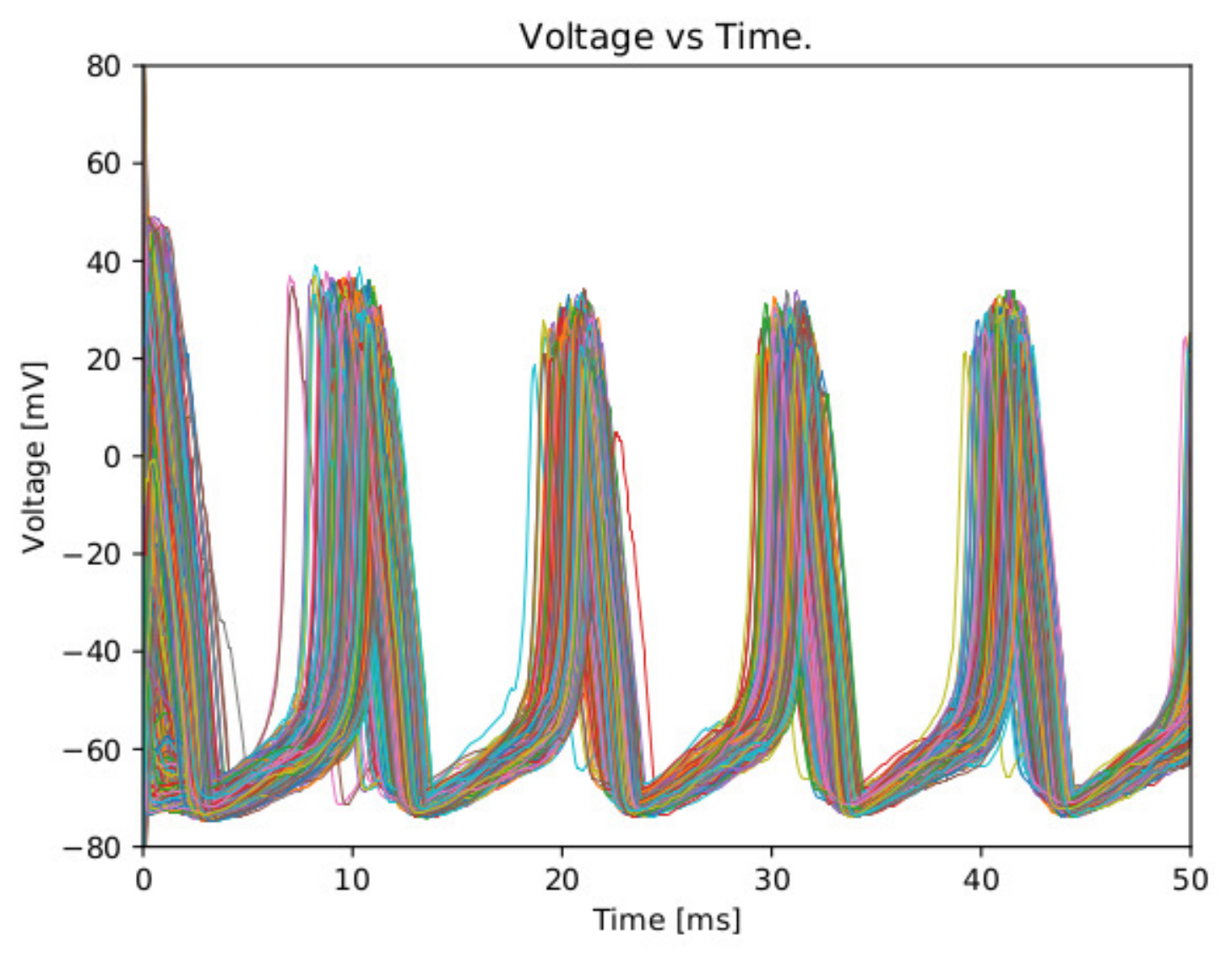}
       \caption{\fontsize{9}{11}\selectfont $N=10000$, $\sigma=1$.}
    \end{subfigure}
 \caption{\fontsize{9}{11}\selectfont  Synchronization of a network under pure electrical interaction for different network sizes  $N$ and noise levels $\sigma$. \label{figure:Samples-IE}}
\end{figure}

\begin{figure}[ht!]
    \centering
    \begin{subfigure}[t]{0.31\textwidth}
        \centering
        \includegraphics[width=\textwidth,height=4.8cm]{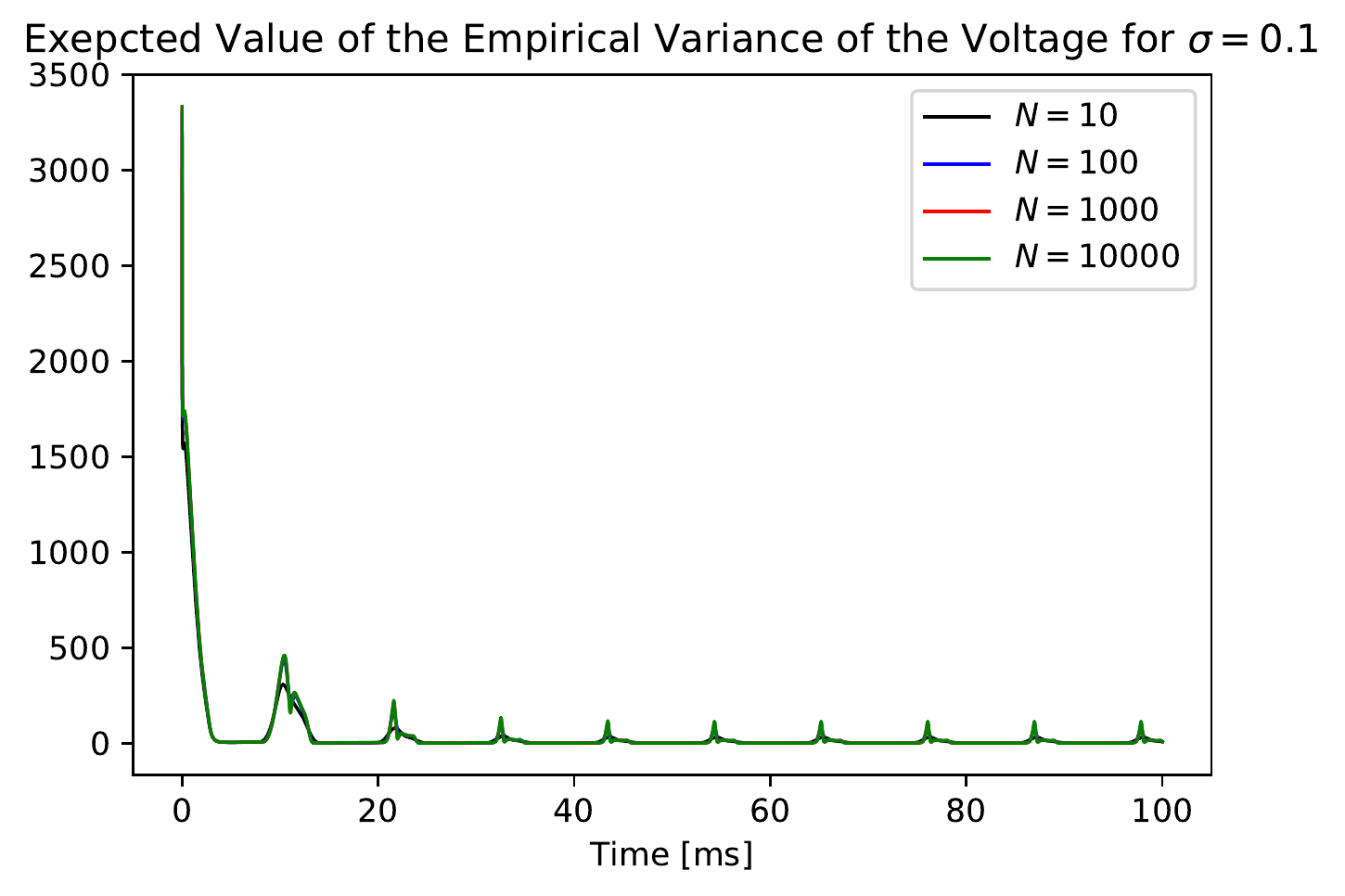}
        \caption{\fontsize{9}{11}\selectfont $V$, $\sigma=0.1$.}
    \end{subfigure}
    \quad
   \begin{subfigure}[t]{0.31\textwidth}
        \centering
        \includegraphics[width=\textwidth,height=4.8cm]{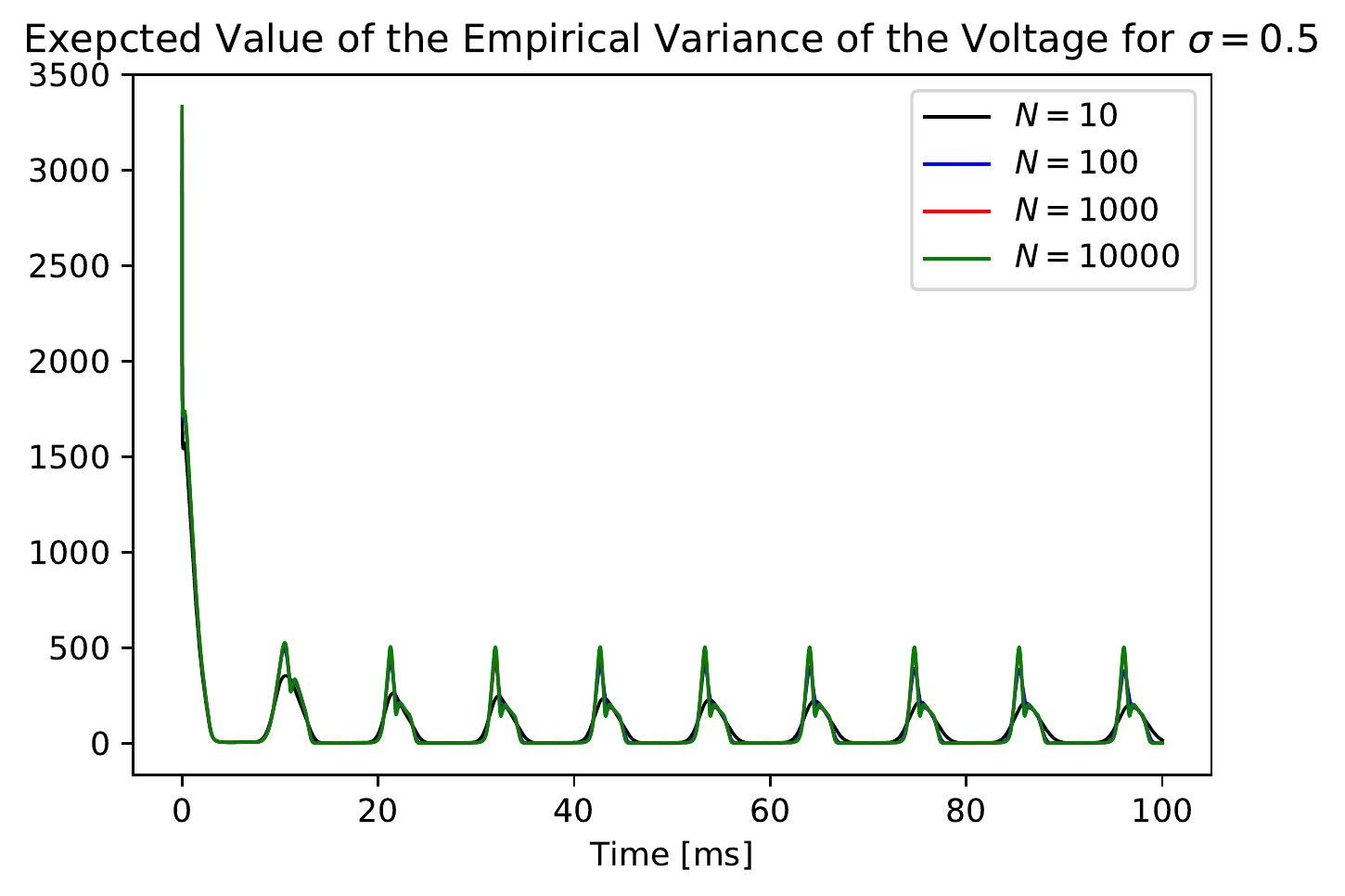}
        \caption{\fontsize{9}{11}\selectfont $V$, $\sigma=0.5$.}
    \end{subfigure}
    \quad
       \begin{subfigure}[t]{0.31\textwidth}
        \centering
        \includegraphics[width=\textwidth,height=4.8cm]{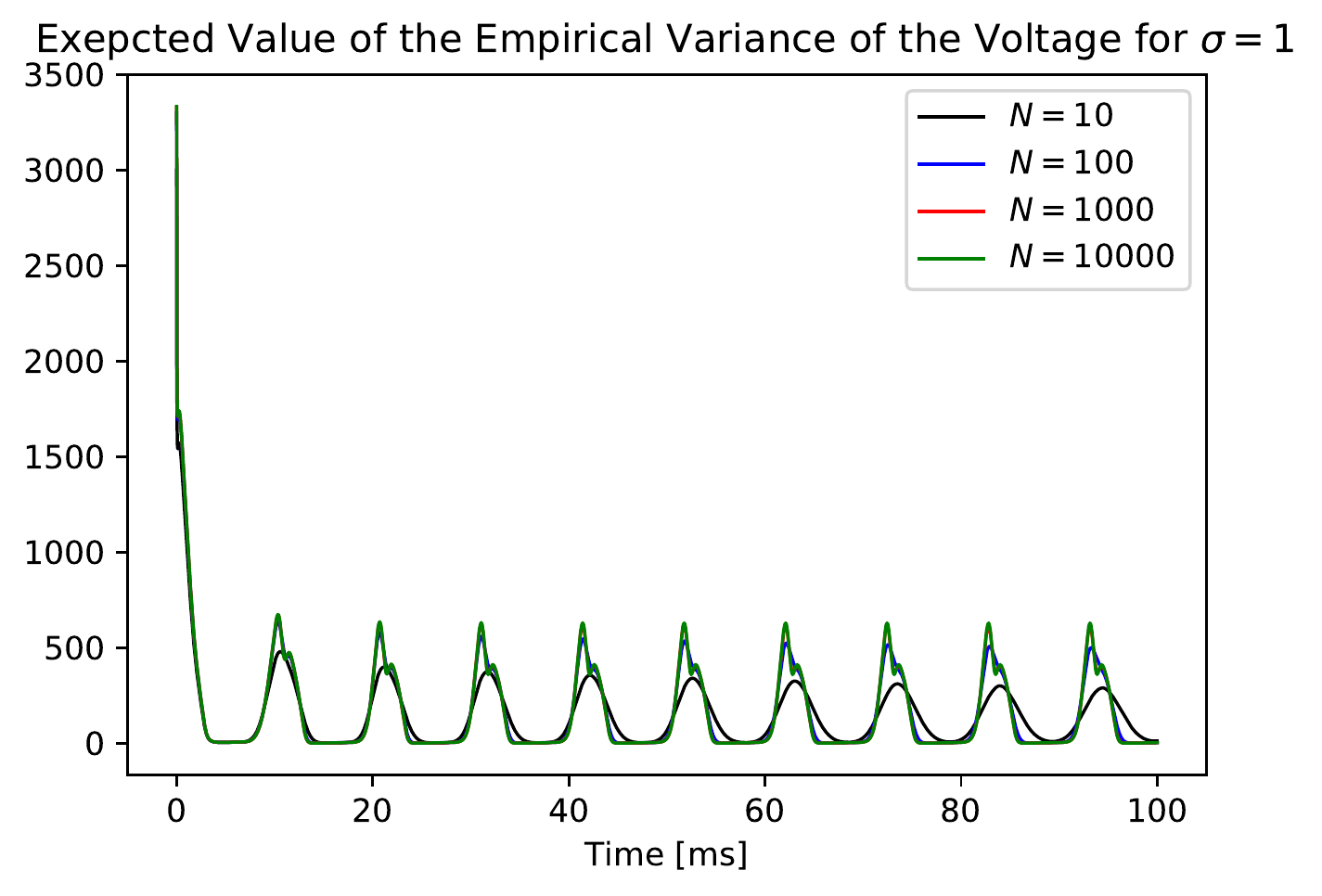}
       \caption{\fontsize{9}{11}\selectfont $V$, $\sigma=1$.}
    \end{subfigure}
    \\
       \begin{subfigure}[t]{0.31\textwidth}
        \centering
        \includegraphics[width=\textwidth,height=4.8cm]{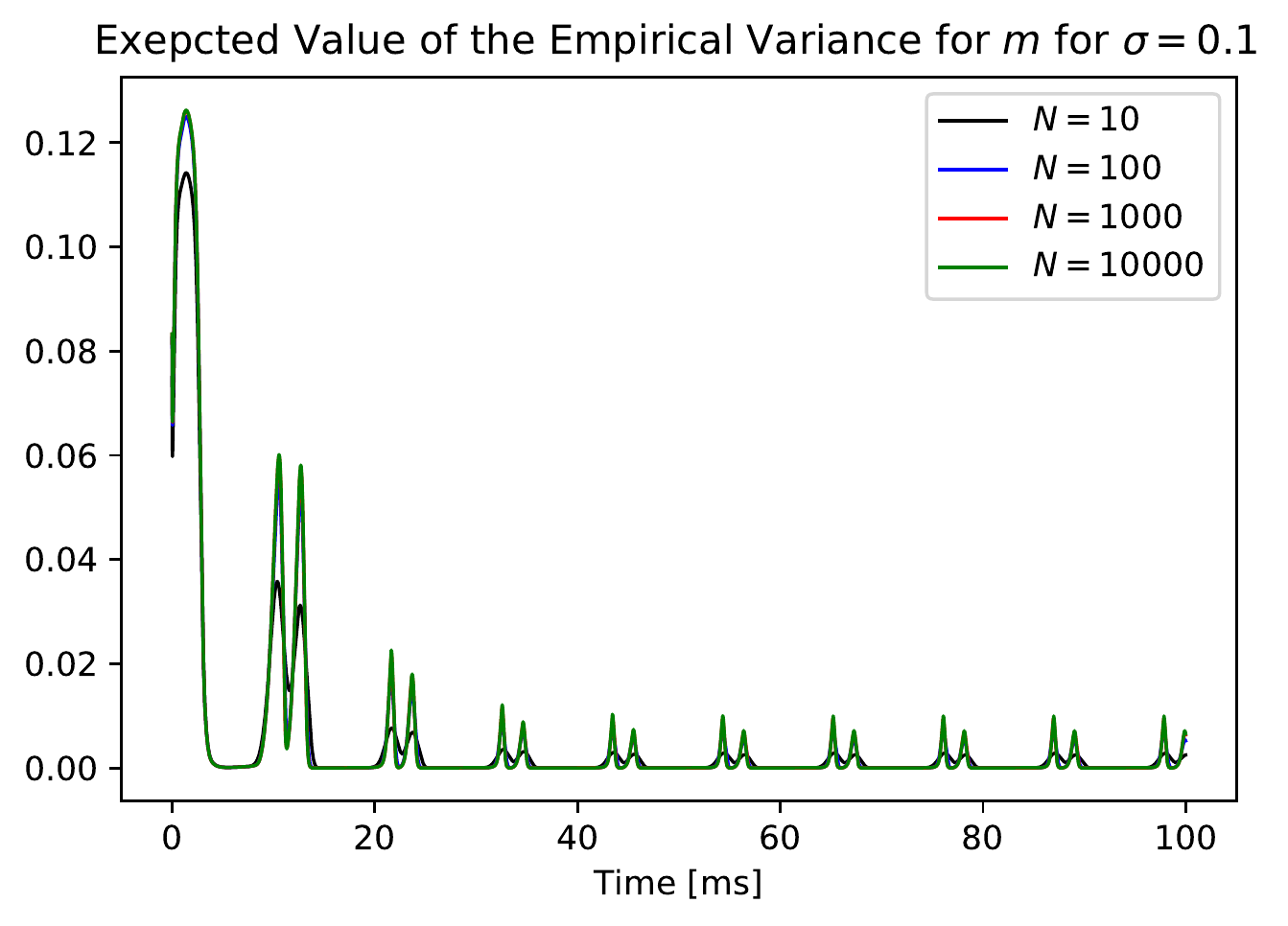}
        \caption{\fontsize{9}{11}\selectfont $m$, $\sigma=0.1$.}
    \end{subfigure}
    \quad
    \begin{subfigure}[t]{0.31\textwidth}
       \centering
        \includegraphics[width=\textwidth,height=4.8cm]{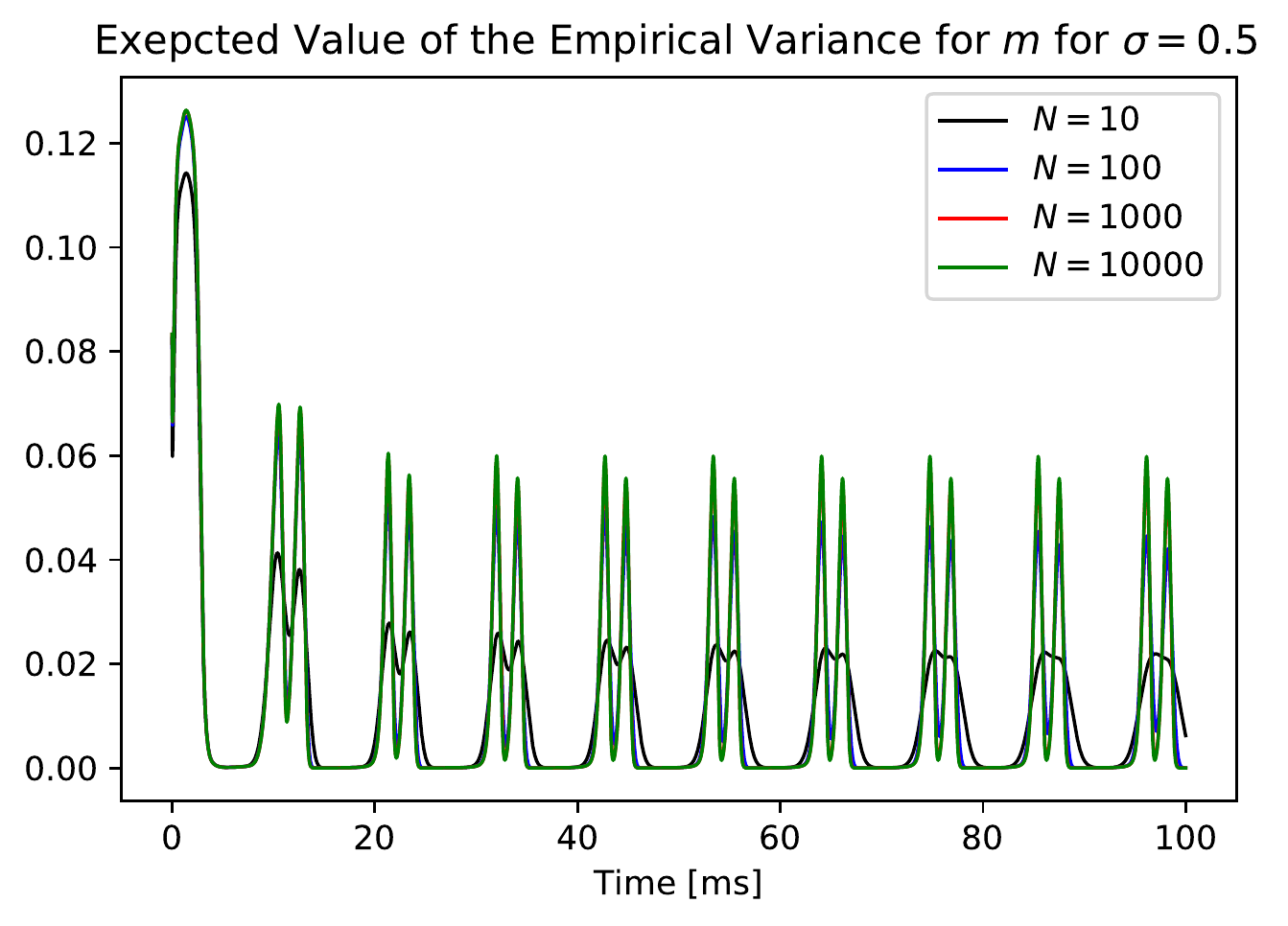}
         \caption{\fontsize{9}{11}\selectfont $m$, $\sigma=0.5$.}
    \end{subfigure}
    \quad
    \begin{subfigure}[t]{0.31\textwidth}
        \centering
        \includegraphics[width=\textwidth,height=4.8cm]{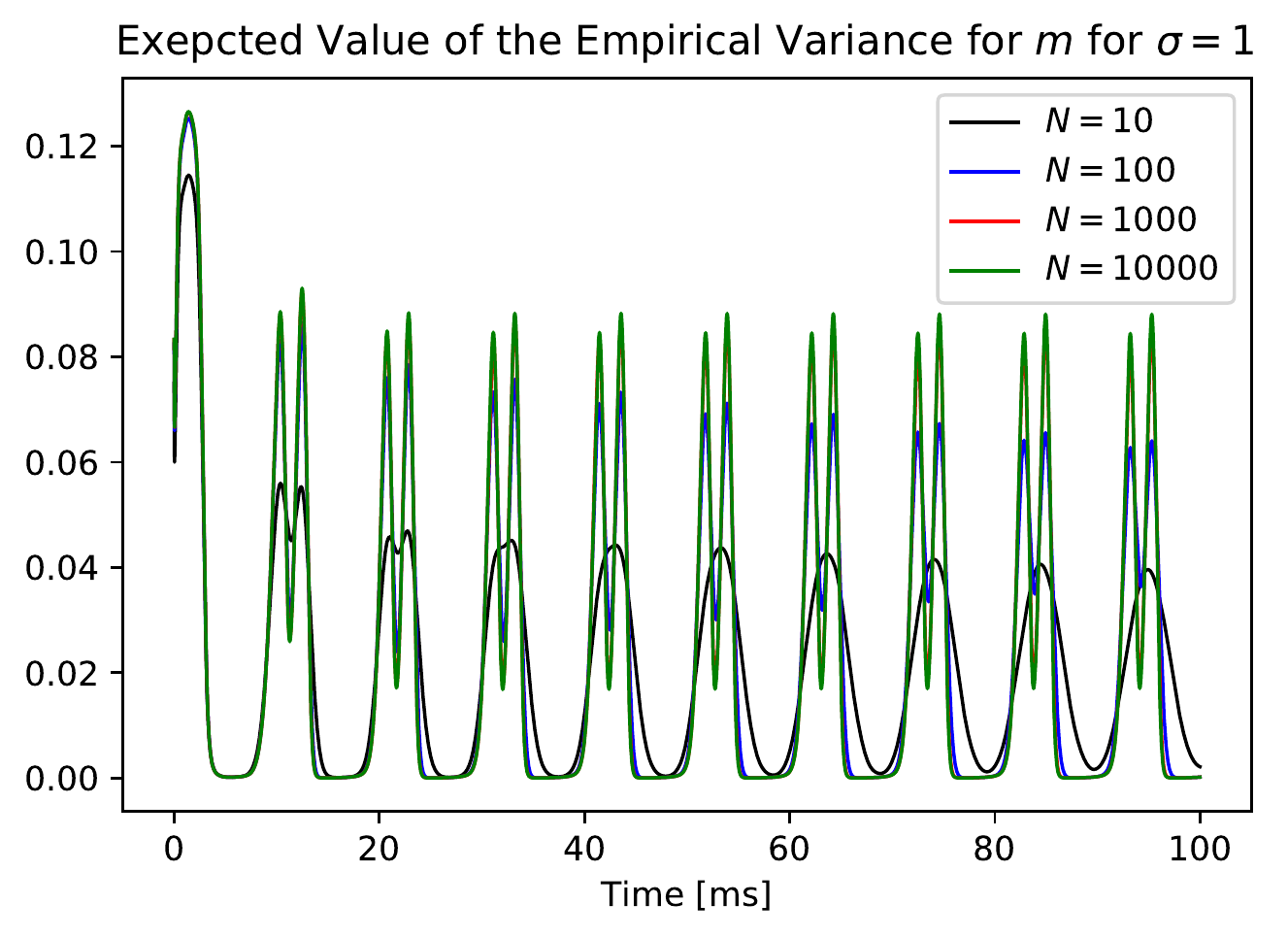}
        \caption{\fontsize{9}{11}\selectfont $m$, $\sigma=1$.}
    \end{subfigure}
    \\
       \begin{subfigure}[t]{0.31\textwidth}
        \centering
        \includegraphics[width=\textwidth,height=4.8cm]{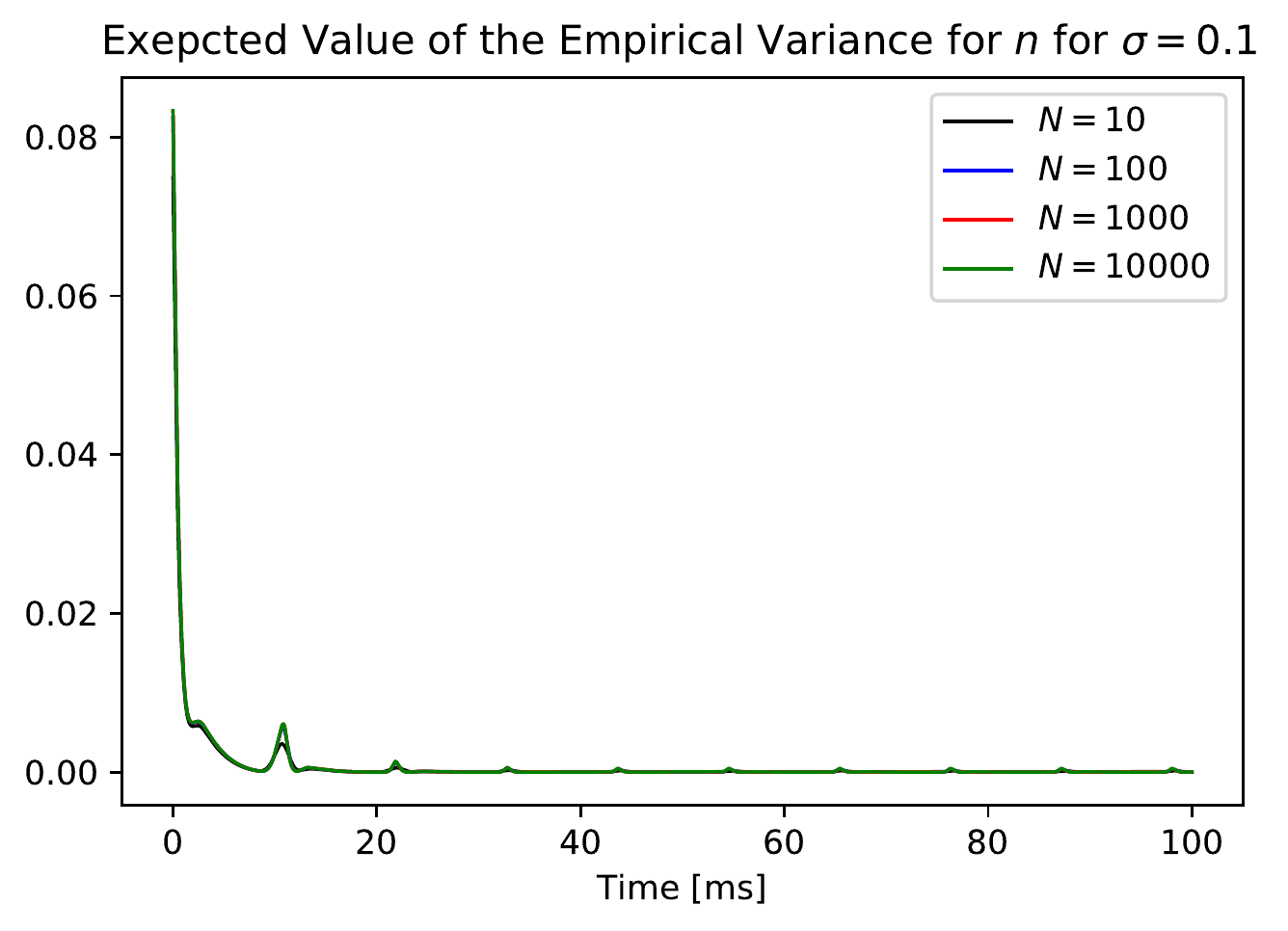}
         \caption{\fontsize{9}{11}\selectfont $n$, $\sigma=0.1$.}
    \end{subfigure}
    \quad
    \begin{subfigure}[t]{0.31\textwidth}
        \centering
        \includegraphics[width=\textwidth,height=4.8cm]{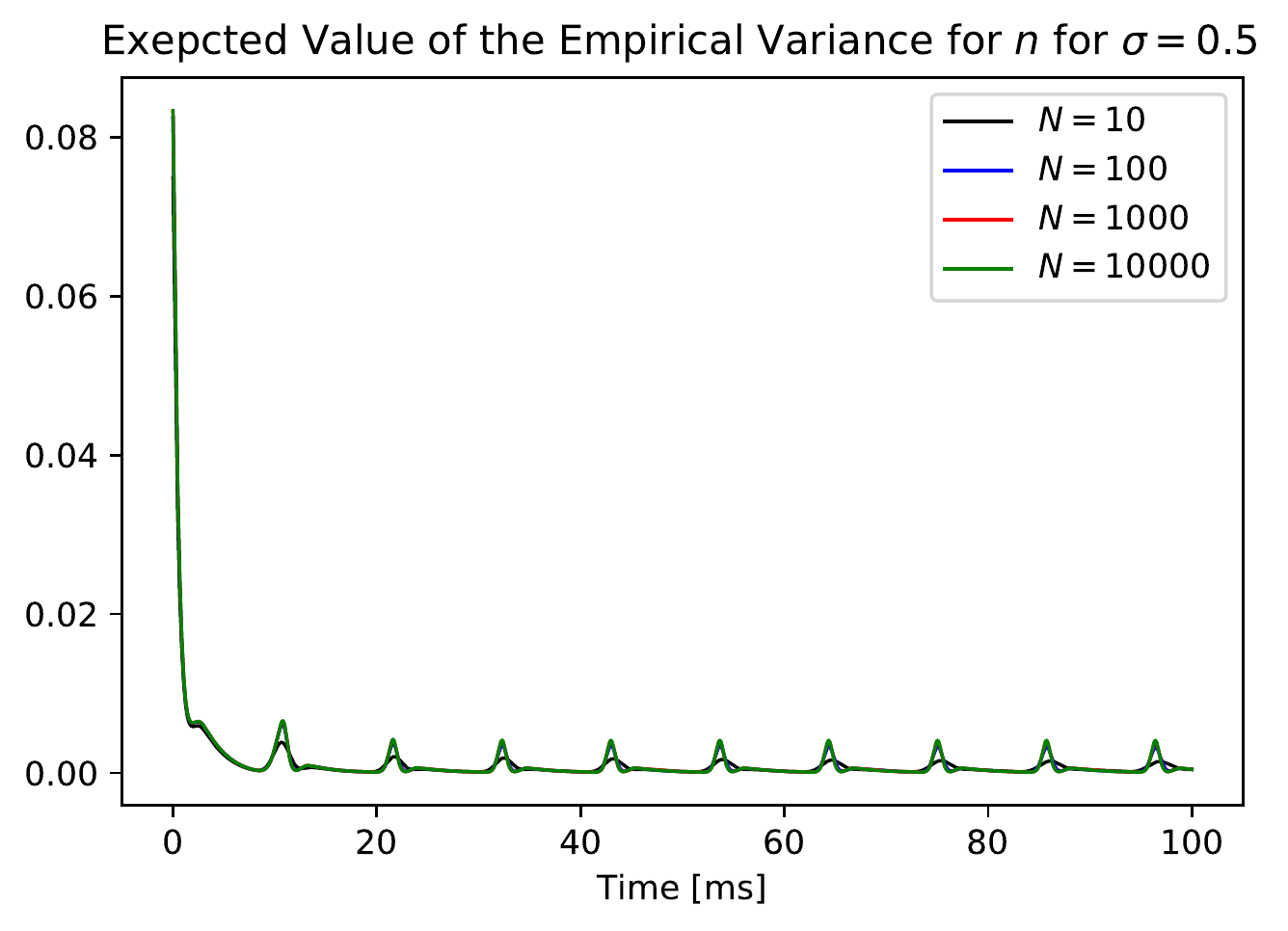}
        \caption{\fontsize{9}{11}\selectfont $n$, $\sigma=0.5$.}
    \end{subfigure}
    \quad
    \begin{subfigure}[t]{0.31\textwidth}
        \centering
        \includegraphics[width=\textwidth,height=4.8cm]{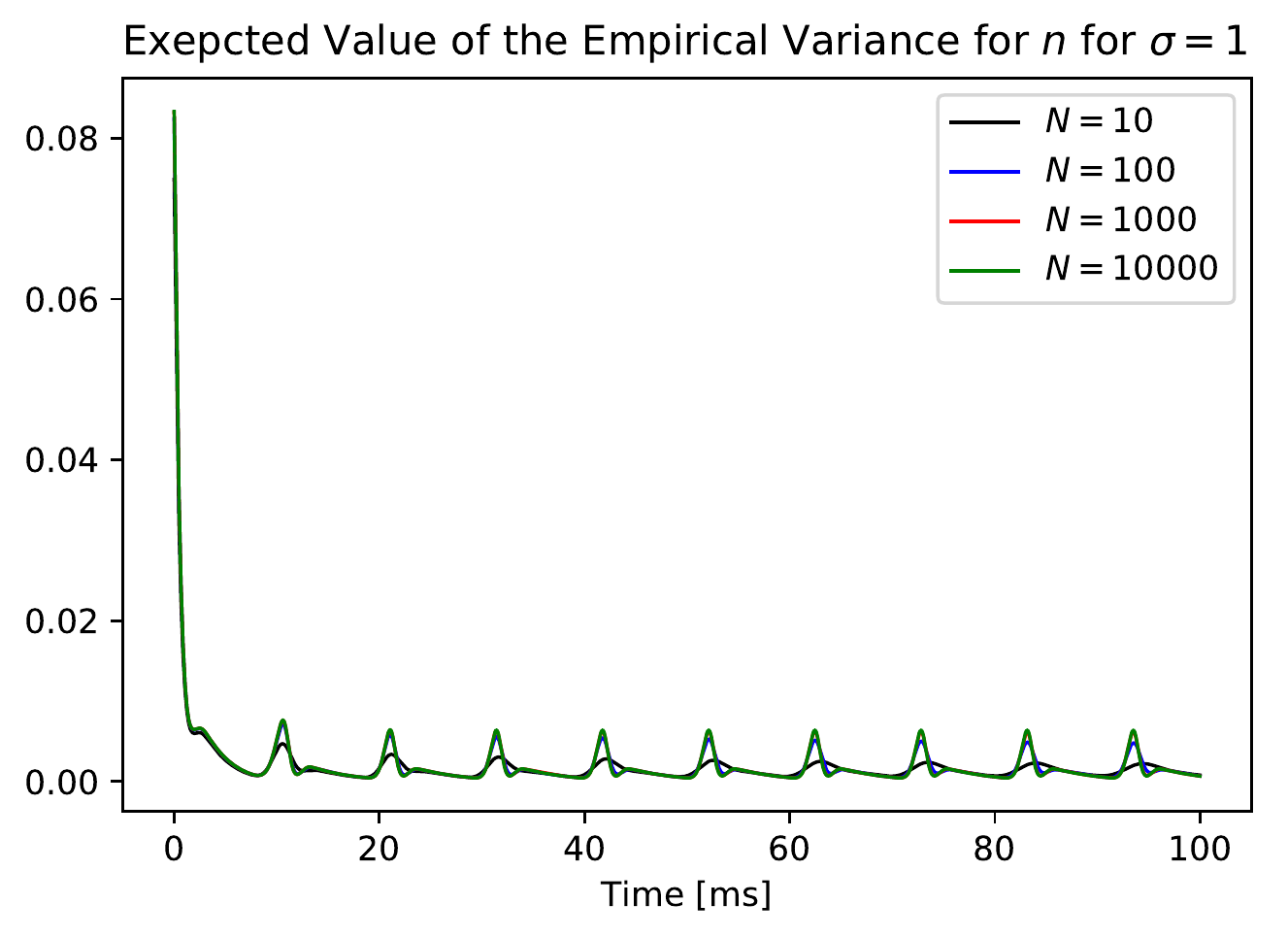}
        \caption{\fontsize{9}{11}\selectfont $n$, $\sigma=1$.}
   \end{subfigure}
     \\
       \begin{subfigure}[t]{0.31\textwidth}
        \centering
        \includegraphics[width=\textwidth,height=4.8cm]{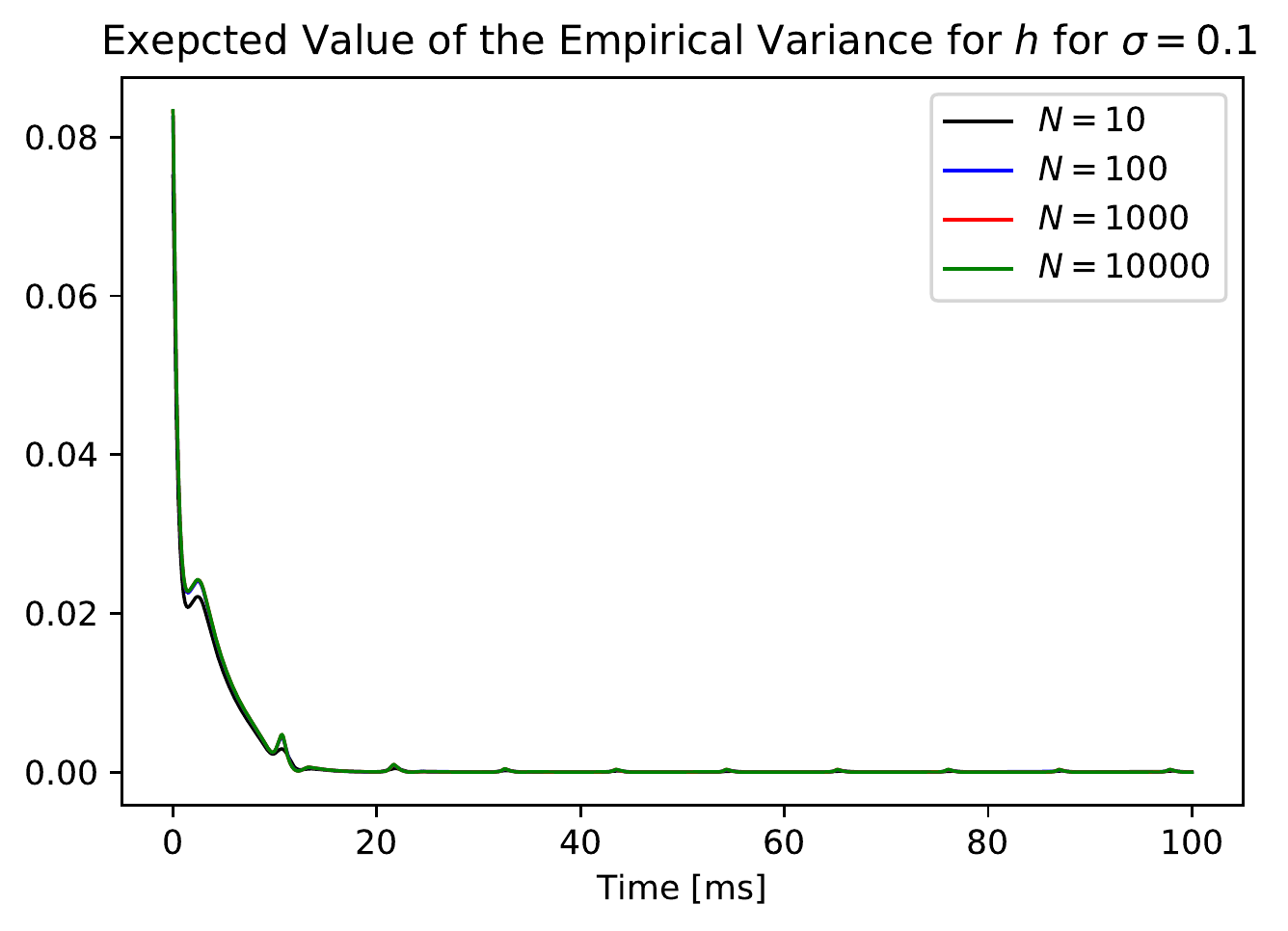}
         \caption{\fontsize{9}{11}\selectfont $h$, $\sigma=0.1$.}
    \end{subfigure}
    \quad
    \begin{subfigure}[t]{0.31\textwidth}
        \centering
        \includegraphics[width=\textwidth,height=4.8cm]{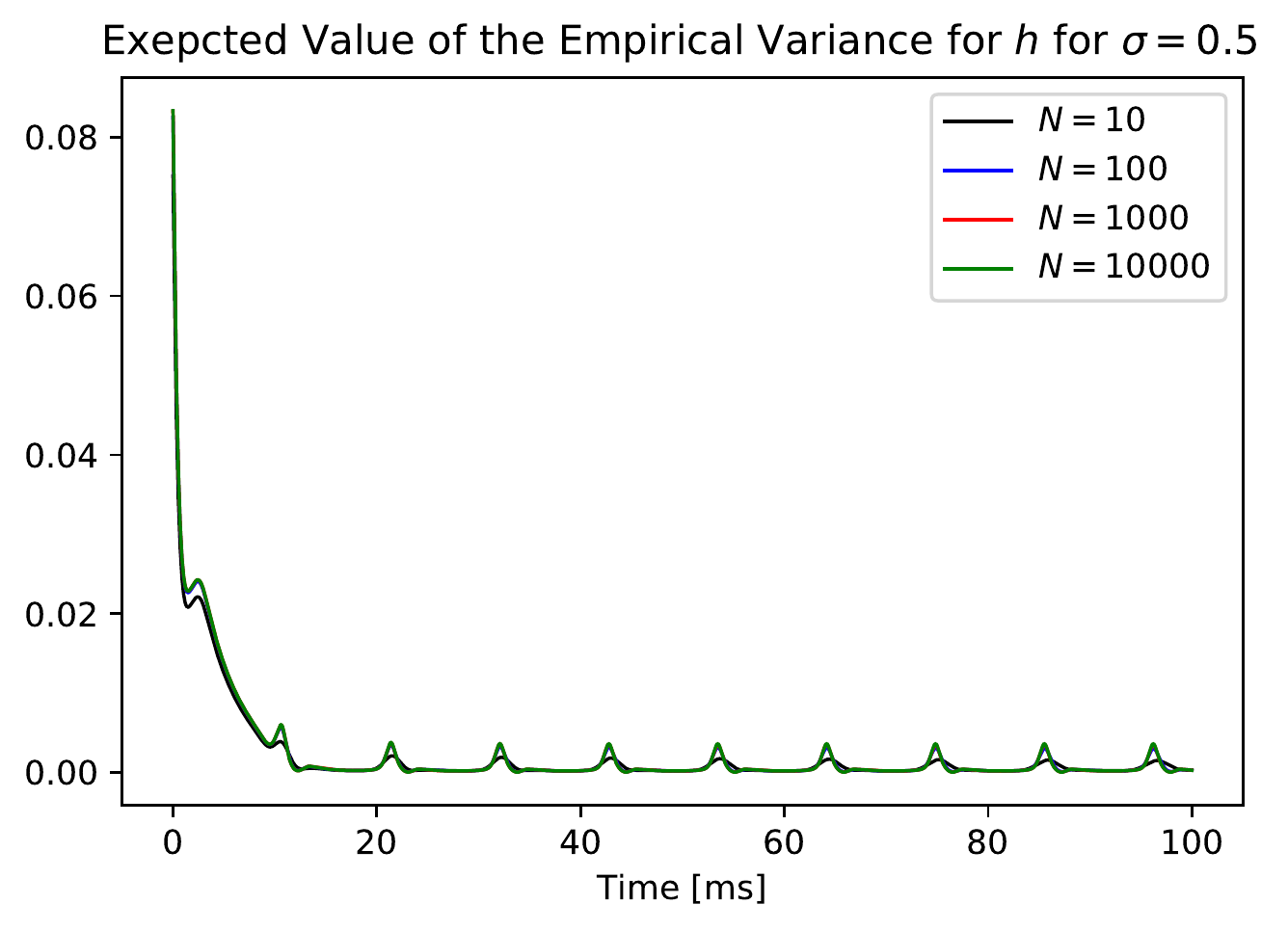}
  \caption{\fontsize{9}{11}\selectfont $h$, $\sigma=0.5$.}
      \end{subfigure}
    \quad
    \begin{subfigure}[t]{0.31\textwidth}
        \centering
        \includegraphics[width=\textwidth,height=4.8cm]{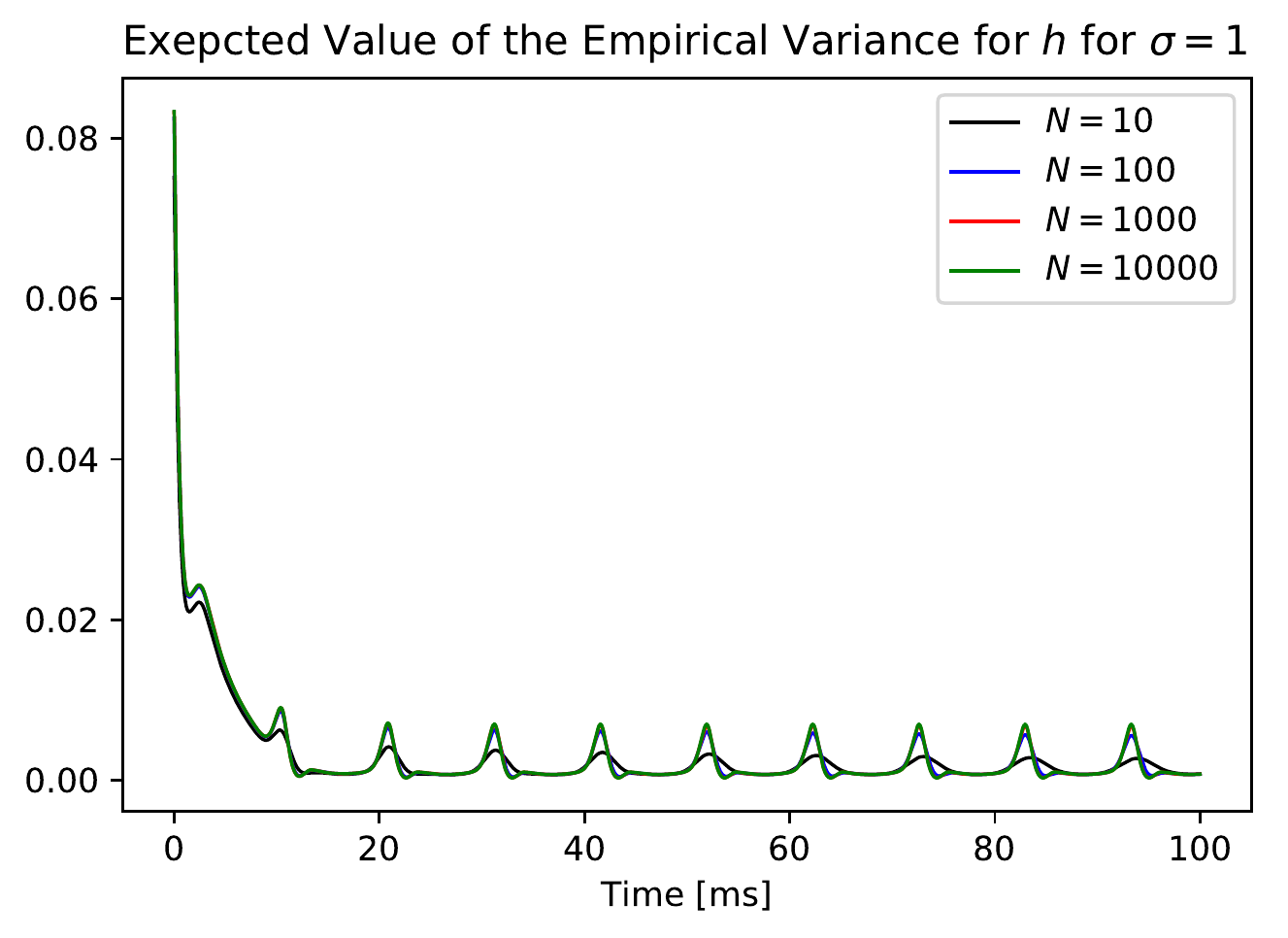}
         \caption{\fontsize{9}{11}\selectfont $h$, $\sigma=1$.}
   \end{subfigure}
   
    \caption{\fontsize{9}{11}\selectfont  Dissipation of the empirical variance for different level of noise. First row:  expected  empirical variance of the voltage; second to fourth rows: expected empirical variance of the Sodium activation channels ($m$), the Potassium channels ($n$) and the Sodium deactivation channels ($h$) respectively. $\JCh=0$ and $\JE=1$ in all simulations.   \label{figure:DisipationEmpiricalVariance-IE}}
\end{figure}

\clearpage

\subsection{Beyond Theorem  \ref{theo:synchro-empiricalVar}}\label{numexpbeyond}

We next carry out  two type of experiments in situations not covered by our theoretical  results on synchronization.  In the first one we study the behavior of a more realistic network with several subpopulations of neurons. The dynamics of the $i$-th neuron in the subpopulation of type $\alpha\in P$, with $P$  denoting the set of subpopulations,  is given by
\begin{equation*}
\begin{aligned}
 V_t^{(i)}& =  V^{(i)}_0 + \int_{0}^{t}F_{\alpha}(V_s^{(i)}, m_s^{(i)},n_s^{(i)},h_s^{(i)}) ds\\
 & \quad - \int_0^t \sum_{\gamma\in P}\frac{1}{N_\gamma}\sum_{j=1}^{N_\gamma}\JE^{\alpha\gamma} (V_s^{(i)}-V_s^{(j)})   -\sum_{\gamma\in P}\frac{1}{N_\gamma}\sum_{j=1}^{N_\gamma}{\JCh^{\alpha\gamma} y^{(j)}_s (V_s^{(i)}-\Vrev^{\alpha\gamma})} ds,\\
 x^{(i)}_t& = x^{(i)}_0+\int_{0}^{t}\rho_x^\alpha(V^{(i)}_s)(1-x^{(i)}_s)  -\zeta_x^\alpha(V^{(i)}_s)x^{(i)}_sds\\
&\quad\quad\quad\quad\quad + \int_{0}^{t}{\sigma_x^\alpha(V_s^{(i)},x_s^{(i)})dW_s^{x,i}},\;\;x=m,n,h,y \, , \qquad t\geq 0 \, ,
\end{aligned}
\end{equation*}
where $N_\gamma$ is the number of neurons in  subpopulation $\gamma$. We notice that in this case the electric and chemical conductivity parameters are  $|P|\times |P|$ matrices. Propagation of chaos for such systems as $N\to \infty$ was proved in  \cite{Bossy2015} (though under slightly more stringent assumptions on the coefficients). 
In Figure \ref{figure:TwoPopElecInt} we show one trajectory of a network of $100$ neurons with two subpopulations, each of them with $50$ neurons. On the left  (plot (a)), we consider the two subpopulations with different levels of noise and  different input current for each of them (different $F$), meanwhile  the electrical conductance matrix $\JE^{\alpha\gamma}$ is homogeneous,  with all the components equal to $1$. We observe that the whole network gets synchronized. We believe that  Theorem \ref{theo:synchro-empiricalVar}  can be easily extended to this case (or, more generally,  when $\inf_{\alpha,\gamma\in P}\JE^{\alpha\gamma}$ is big enough). 
 In the middle  (plot (b)), we observe  that, if in addition to considering different subpopulations,  the matrix $\JE$ is not homogeneous (taking in some entries strictly smaller values  than the largest  value $ 1$),   then synchronization can be observed in each subpopulation but not globally. More precisely, in these examples the two populations synchronize out-of-phase.  Finally, on the right (plot (c)),   we observe no evidence of synchronization when neurons in one population with same $F$ are  electrically connected with two small but different values for $\JE$. This is in line with our theoretical result that thresholds  the synchronization of the dynamics for a big enough $\inf_{\alpha,\gamma\in P}\JE^{\alpha\gamma}$, even if Theorem \ref{theo:synchro-empiricalVar}  gives only a sufficient condition on $\JE$. 
\smallskip

\new{ According to  Kopell and Ermentrout \cite{kopell2004chemical}, ``a small amount of electrical conductance can increase the degree of synchronization far more that a much larger increase in inhibitory conductance''. This is consistent with what we have observed in our numerical experiments. The effect of the electrical interaction is stronger and faster than the effect of the chemical interaction. Therefore to appreciate the effect of the chemical synapses, the second type of experiment we have performed concerns only  chemical  synapses, }that is, $\JE=0$.  In Figure \ref{figure:Ch Inhibitory} we show one trajectory of a network of $100$ neurons interacting through inhibitory chemical synapses (in this case,  we choose $\Vrev= -75 $  according to \cite[p.163]{Ermentrout:2010aa}).  
We observe an anti-phase synchronization that persists in time (see   in plot  (a)  the transition  to  the stationary regime in plot (b)), in which two clusters of simultaneously  firing neurons  emerge.  We note that the relative sizes of the two clusters is random and might change  from one simulation to another one. 

 On the other hand, Figure \ref{figure:Ch Excitatory} shows one trajectory  of a network of $100$ neurons interacting through excitatory chemical synapses (with $\Vrev=0$, see \cite[p.161]{Ermentrout:2010aa}). Some   kind of synchronization,  similar to the case of  purely  electrical synapses  (see e.g. Figure \ref{figure:TwoPopElecInt}(a)),  emerges also here,  although the shape of the oscillations is different and the frequency of the spikes is smaller.

\begin{figure}[ht!]
    \centering
    \begin{subfigure}[t]{0.31\textwidth}
        \centering
        \includegraphics[width=\textwidth,height=4.8cm]{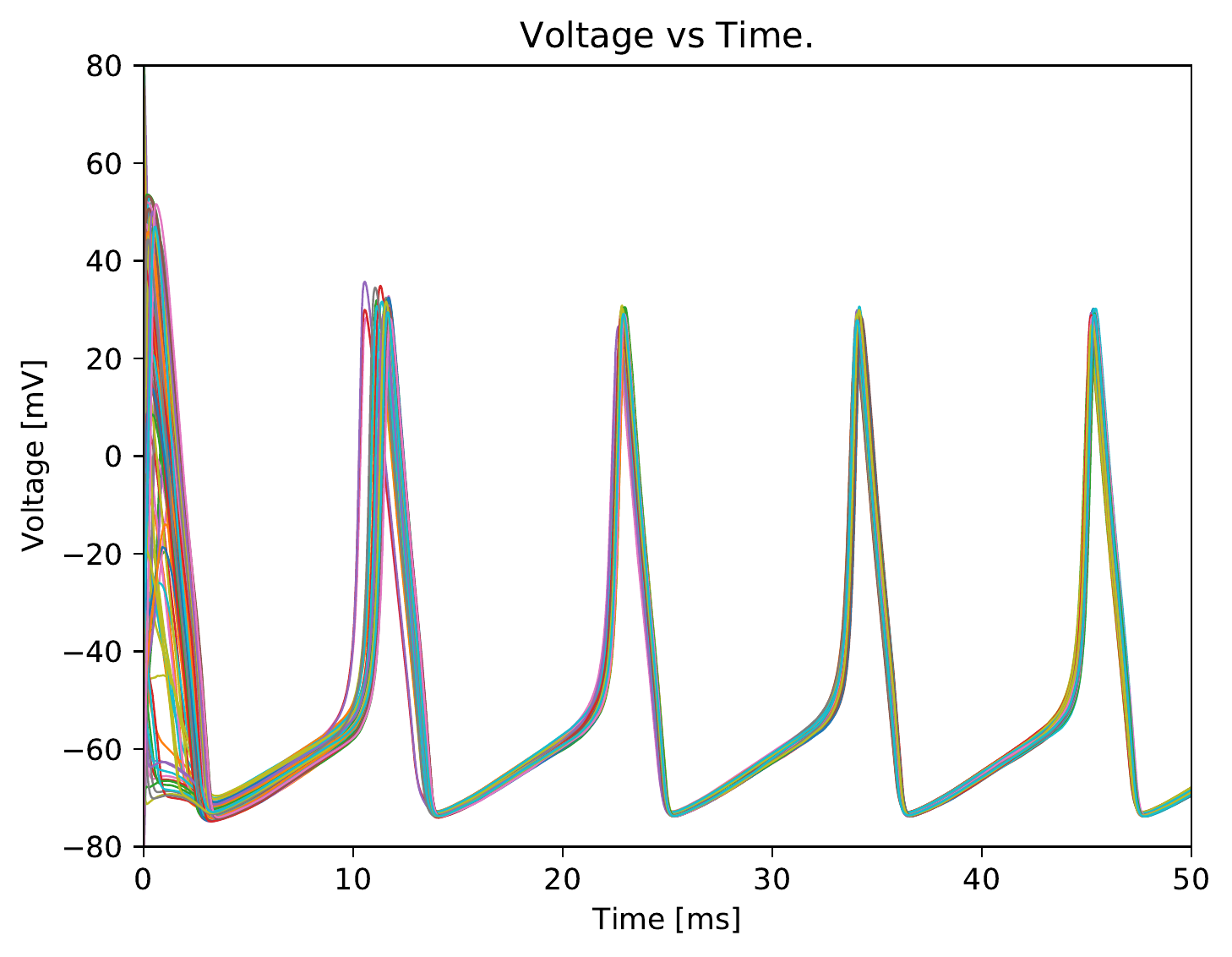}
        \caption{\fontsize{9}{11}\selectfont Populations are different, the interaction is homogeneous.}
    \end{subfigure}
    \quad
   \begin{subfigure}[t]{0.31\textwidth}
        \centering
        \includegraphics[width=\textwidth,height=4.8cm]{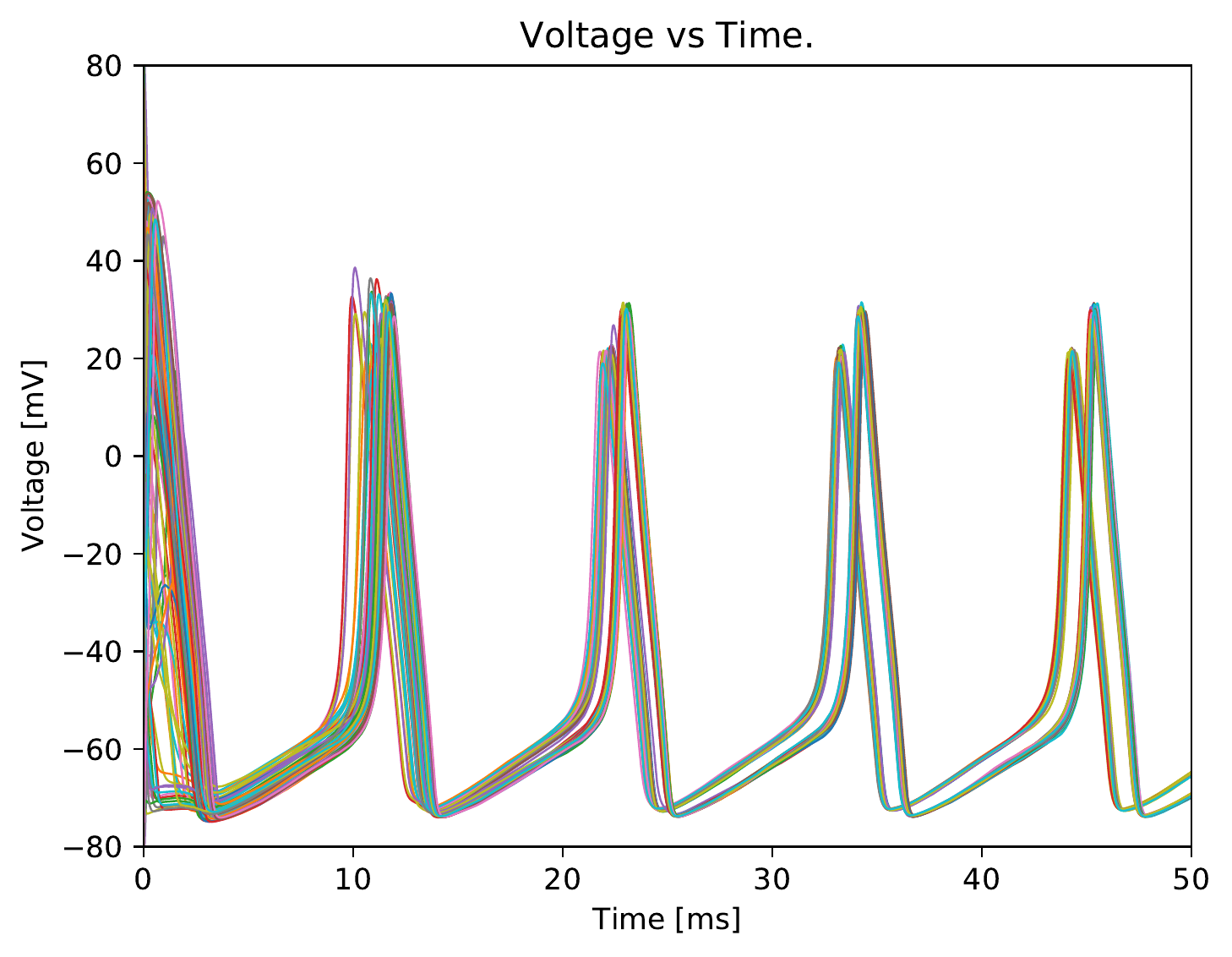}
        \caption{\fontsize{9}{11}\selectfont Populations are different, the interaction is heterogeneous.}
    \end{subfigure}
    \quad
       \begin{subfigure}[t]{0.31\textwidth}
        \centering
        \includegraphics[width=\textwidth,height=4.8cm]{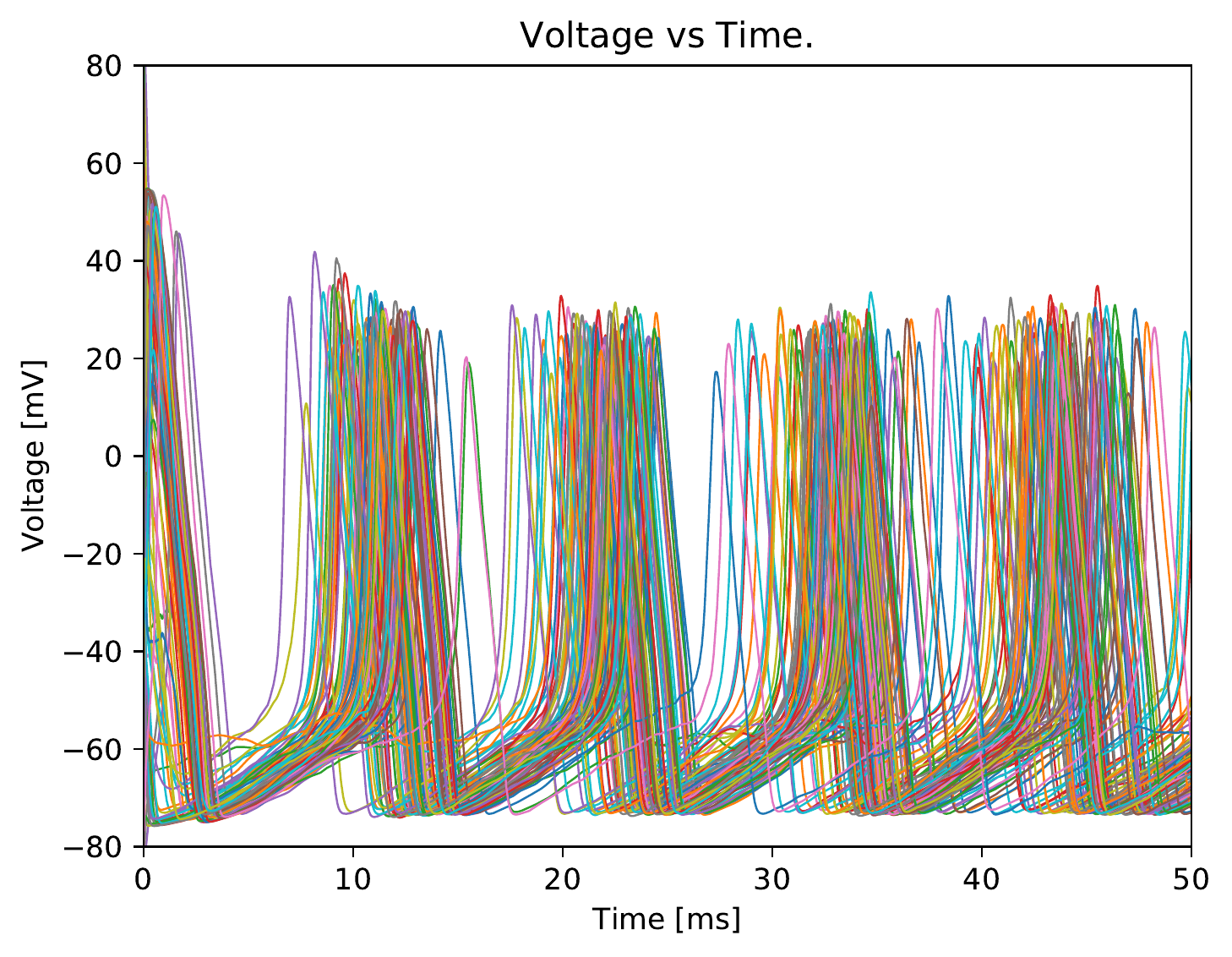}
       \caption{\fontsize{9}{11}\selectfont Populations are equal, the interaction is heterogeneous.}
    \end{subfigure}
   
    \caption{\fontsize{9}{11}\selectfont  Trajectories for network of $100$ neuron with two subpopulations. \label{figure:TwoPopElecInt}}
\end{figure}

\begin{figure}[ht!]
    \centering
    \begin{subfigure}[t]{0.48\textwidth}
        \centering
        \includegraphics[width=\textwidth,height=4.8cm]{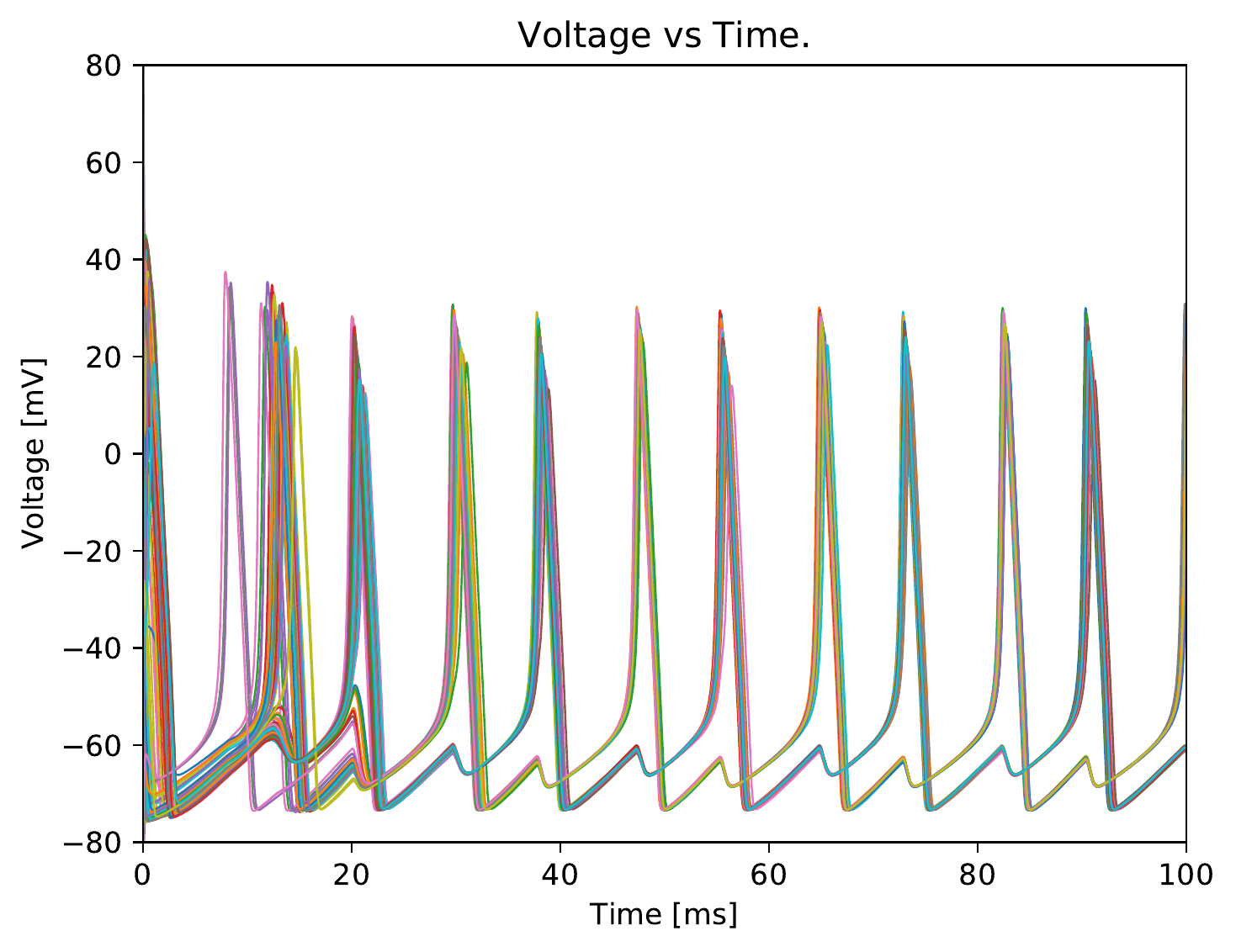}
        \caption{\fontsize{9}{11}\selectfont Voltage of $100$ neurons at the beginning of the simulation.}
    \end{subfigure}
    \quad
   \begin{subfigure}[t]{0.48\textwidth}
        \centering
        \includegraphics[width=\textwidth,height=4.8cm]{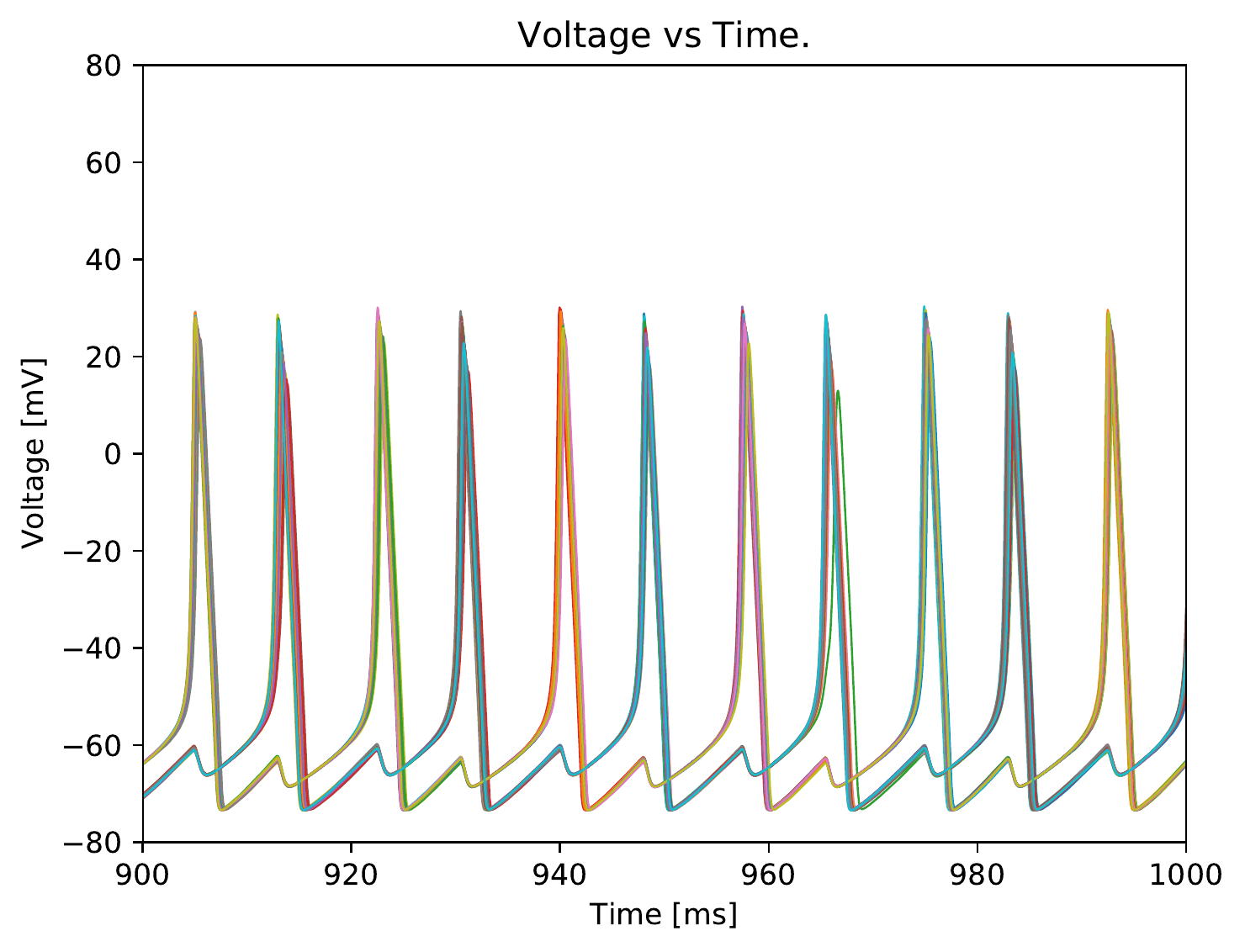}
        \caption{\fontsize{9}{11}\selectfont Voltage of $100$ neurons after $900$ [ms].}
    \end{subfigure}
   
    \caption{\fontsize{9}{11}\selectfont  Trajectories for network of $100$ neuron with inhibitory chemical synapses \label{figure:Ch Inhibitory}}
\end{figure}   

\begin{figure}[ht!]
    \centering
    \begin{subfigure}[t]{0.48\textwidth}
        \centering
        \includegraphics[width=\textwidth,height=4.8cm]{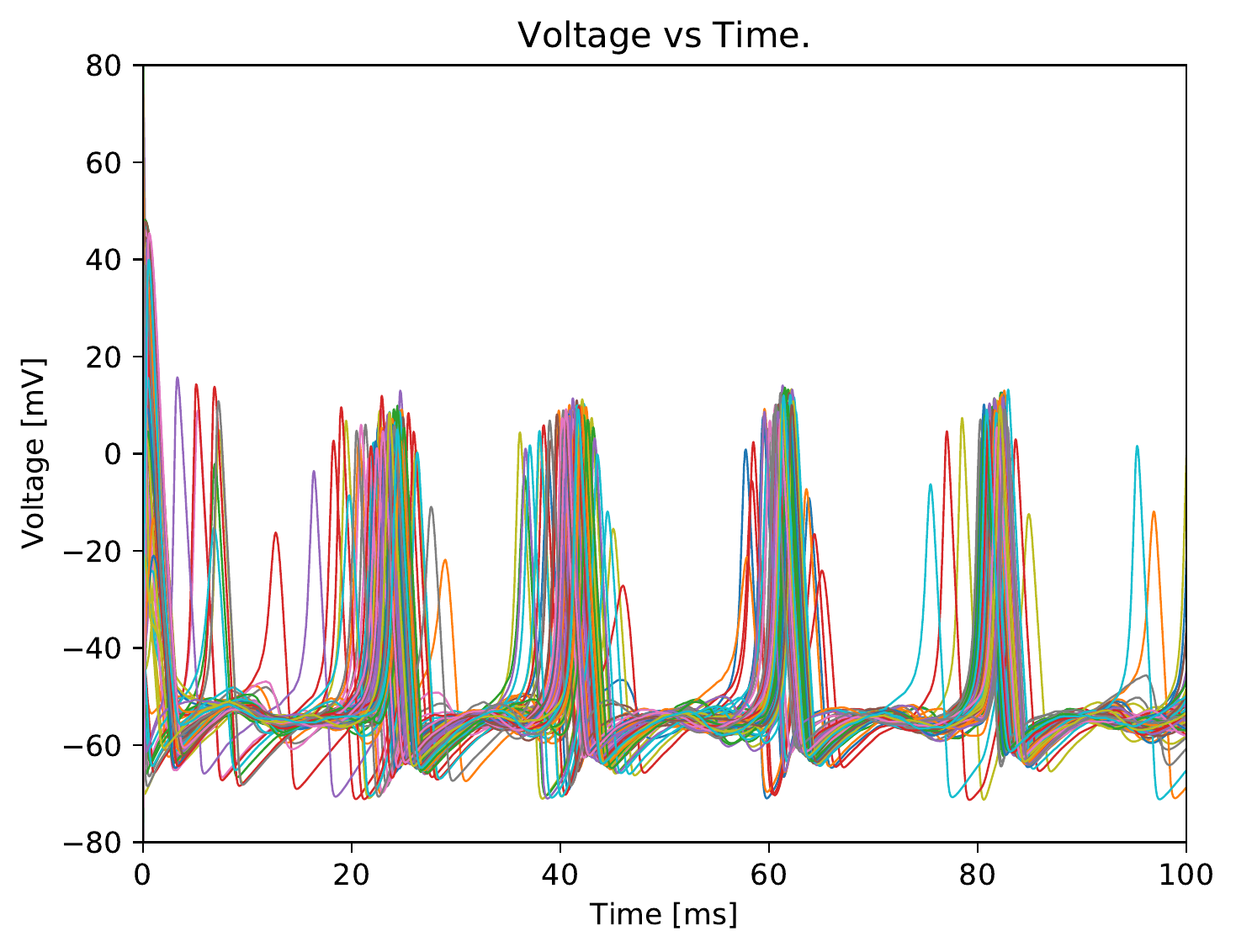}
        \caption{\fontsize{9}{11}\selectfont Voltage of $100$ neurons at the beginning of the simulation.}
    \end{subfigure}
    \quad
   \begin{subfigure}[t]{0.48\textwidth}
        \centering
        \includegraphics[width=\textwidth,height=4.8cm]{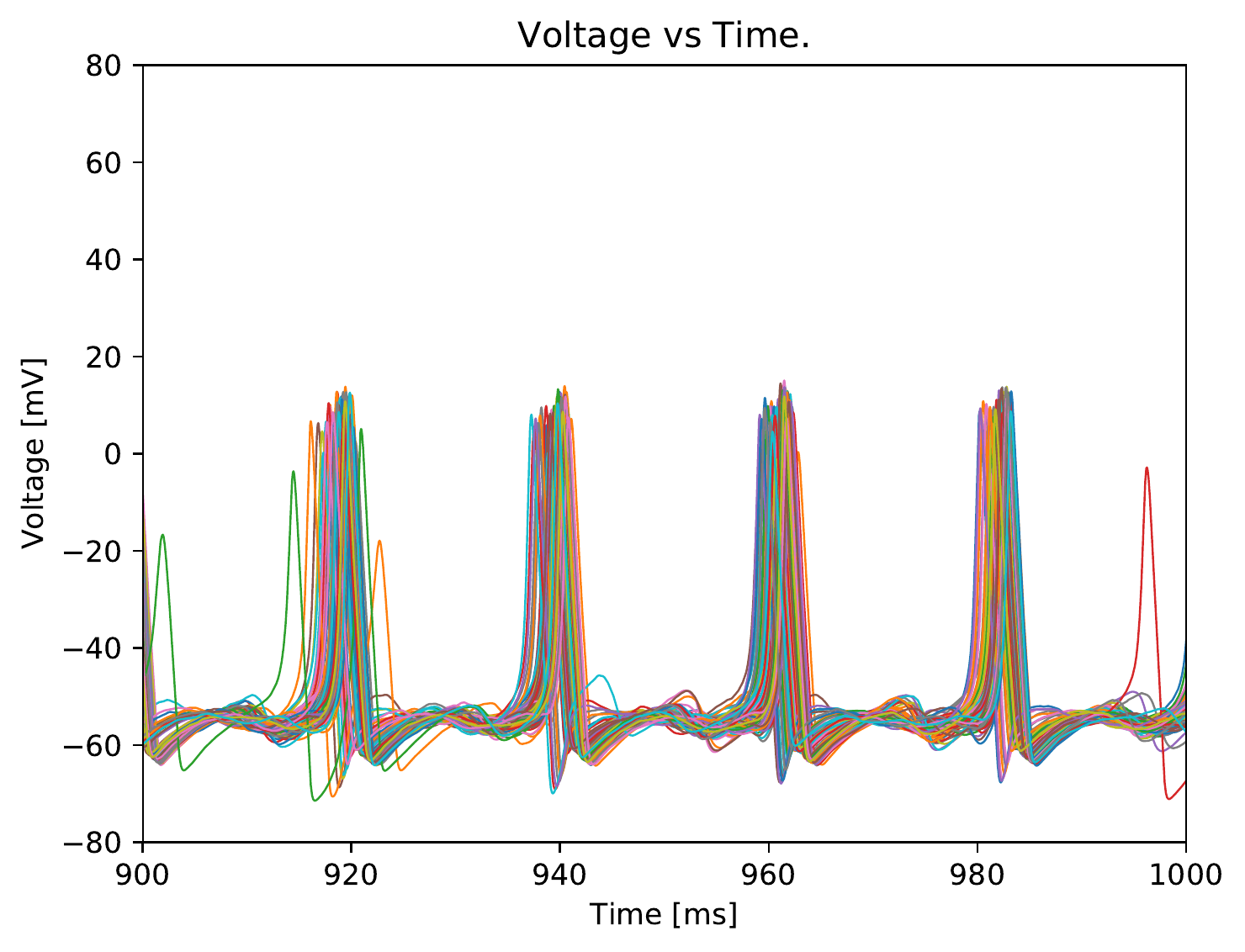}
        \caption{\fontsize{9}{11}\selectfont Voltage of $100$ neurons after $900$ [ms].}
    \end{subfigure}
   
    \caption{\fontsize{9}{11}\selectfont  Trajectory of  a network of $100$ neurons with excitatory chemical synapses \label{figure:Ch Excitatory}}
\end{figure}   

\clearpage
\section{Concluding remarks and open problems}\label{sec:ConcluOpen}

The numerical experiments  presented in Section \ref{numexpillustheo} show that,  from a quantitative point of view,  our theoretical results should still  allow for considerable improvements.  Indeed, our simulations indicate that the actual global bounds on the voltage,  the critical interaction strength above which synchronization  happens and the  asymptotic discrepancy from synchronization  are much smaller than suggested from rough estimates  of the bounds in our theoretical results and their proofs. In turn,  the  actual exponential synchronization rate  seems to be much higher.   Also, our theoretical results  treat  the values of the voltage and channels variables jointly, although they are of considerably different orders of magnitude (i.e. channel variables and variances are negligible compared to the voltages ones). The numerical experiments show in turn that the theoretically described  behavior happens at  each variable's scale.

We  must also emphasize that  the anti- or out-of-phase-  synchronization responses put in evidence in the numerical experiments  presented in Section \ref{numexpbeyond}   are not well captured by the empirical variance criteria proposed in our Theorem \ref{theo:synchro-empiricalVar}. 
In those cases,  a natural, though challenging strategy would be to extend the phase reduction approach and results developed  e.g. in  \cite{hansel1993phase} and \cite{hansel1993patterns}  for finite deterministic networks of HH neurons
 in order to obtain  rigorous synchronization results, regardless of the network size, and then in the mean field limit.  Another interesting but also challenging  question is  studying the existence of stationary measures for the McKean-Vlasov limit equation  in relation with  some characteristic  behavior  of the system and its possible synchronization regimes, in the vein of recent works of Bertini et al.  \cite{bertini-giacomin-poquet-2014},  Giacomin et al. \cite{giacomin-lucon-poquet-2014} and  Lu\c{c}on  and  Poquet \cite{lucon-poquet-2017}.  These questions are left for future works.

\appendix

\section{Basic properties of  the model \eqref{eq:HH-model}}\label{sec:theDiffusiveModel}
We start establishing three basic facts about the system of stochastic differential equations \eqref{eq:HH-model}: its (strong) global well-posedness, the fact that the open channels proportion processes stay (as required)  in $[0,1]$ and, finally, and explicit  global bound for the voltage processes in terms of a bound for the initial values.

 \begin{lem}\label{lem:well-posedHH}
 Assume Hypothesis \ref{hyp:MainHypotheses}.
Then, strong existence and pathwise uniqueness holds for system  \eqref{eq:HH-model}. Moreover,  a.s. for all $t\geq 0$ and every $i=1,\ldots,N$
 we have $( m_t^{(i)},n_t^{(i)},h_t^{(i)},y_t^{(i)})\in [0,1]^4$. In particular, the absolute value in  \eqref{eq:def-sigma_x}
can be removed.
  \end{lem}
\begin{proof} It is enough to prove the result for deterministic initial data so we assume this is the case. Take $M>0$ fixed, and for $j=1,3, 4$ define  truncation functions $p^j_M$  on $\R$  by 
\begin{equation*}
p^j_M(x) = \left\{ \begin{array}{cl}
x^j&x\in[-M,M]\\
M^j&x\in(M,\infty)\\
(-M)^j&x\in(-\infty,-M).
\end{array}\right.
\end{equation*}
Let $X^{(M)}:=(X^{(1,M)},\dots, X^{(N,M)}) $ with $X_t^{(i,M)}=(V_t^{(i,M)}, m_t^{(i,M)},n_t^{(i,M)},h_t^{(i,M)},y_t^{(i,M)})$, $i=1,\dots,N$  be defined by
\begin{equation}\label{eq:truncated-HH-model}
\begin{aligned}
 V_t^{(i,M)}& =  V^{(i,M)}_0 + \int_{0}^{t}F_M(V_s^{(i,M)}, m_s^{(i,M)},n_s^{(i,M)},h_s^{(i,M)}) -\frac{1}{N}\sum_{j=1}^{N}\JE (V_s^{(i,M)}-V_s^{(j,M)})  \\
 &\quad\quad\quad\quad\quad -\frac{1}{N}\sum_{i=1}^{N}{\JCh p_M^1(y^{(j,M)}_s) (p_M^1(V_s^{(i,M)})-\Vrev)}ds,\\
 x^{(i,M)}_t& = x^{(i,M)}_0+\int_{0}^{t}\rho_x(p_M^1(V^{(i,M)}_s))(1-p_M^1(x^{(i,M)}_s))  -\zeta_x(p_M^1(V^{(i,M)}_s))p_M^1(x^{(i,M)}_s)ds\\
&\quad\quad\quad\quad\quad + \int_{0}^{t}{\sigma_x(p_M^1(V_s^{(i,M)}),x_s^{(i,M)})dW_s^{x,i}},\;\;x=m,n,h,y , \,
\end{aligned}
\end{equation}
where 
\begin{equation}\label{eq:truncated-HH-drift}
F_M(v,m,n,h)=I - \gK p_M^4(n)(p_M^1(v)-\VK) - \gNa p_M^3(m)p_M^1(h)(p_M^1(v)-\VNa) -\gL (v-\VL). 
\end{equation}
Is is immediate  that the drift coefficients in system \eqref{eq:truncated-HH-model} are Lipschitz continuous. This is less clear in the case of the diffusion coefficients, so we  check this point next.  Notice that 
$$ \max_{(v,u)\in \R \times [0,1]} \rho_x(p_M^1(v))(1-u) + \zeta_x(p_M^1(v))u\leq  S_M:= \max_{v\in [-M,M]}  \rho_x(v) + \zeta_x(v)<\infty$$
whereas,  thanks to point 2) in  Hypothesis  \ref{hyp:MainHypotheses}, 
$$ \min_{(v,u)\in \R \times [0,1]} \rho_x(p_M^1(v))(1-u) + \zeta_x(p_M^1(v))u\geq \delta_M:= \min_{v\in[-M,M]}\{\rho_x(v), \zeta_x(v)\}>0. $$
Therefore,  one can find a bounded Lipschitz continuous function $g_x: \R\to \R_+$ such that $g_x(s)=\sqrt{s}$ on $(\delta_M/2,2S_M)$  
and rewrite
 the diffusion coefficients in  \eqref{eq:truncated-HH-model} as
 $$\sigma_x(p^1_M(v),u)=  \sigma g_x( |\rho_x(p^1_M(v))(1-u) + \zeta_x(p^1_M(v))u|)\chi(u).$$
It is  then easily seen that $|\sigma_x(p^1_M(v),u)-\sigma_x(p^1_M(v'),u')|\leq C_M (|u-u'|+|v-v'|) $ for some $C_M>0$   in each of the three cases $(u,u')\in [0,1]^2$, $(u,u')\in( [0,1]^2)^c$ and $(u,u')\in [0,1]\times [0,1]^c$ for any $v,v'\in \R$. 
Thus,  global pathwise well-posedness for system  \eqref{eq:truncated-HH-model}  holds.  

Thanks to the  second assumption  in point 4) of Hypothesis \ref{hyp:MainHypotheses} and the fact that $\sigma_x(v,u)=0$ for $(v,u)\in \R\times (0,1)^c$ and  $[\rho_x(v)(1-u)- \zeta_x(v)u][\mathbf{1}_{(-\infty,0]}(u)- \mathbf{1}_{[1,-\infty)}(u)]\geq 0$ for $(v,u)\in \R^2$,  we  can more  apply  Proposition 3.3 in \cite{Bossy2015} to get that   $x^{(1,M)},\dots, x^{(N,M)}$   are confined in $[0,1]$ for all time  (notice that the  proof of that result  still works if  Hypothesis 2.1 i)  therein that  $\chi$  be  compactly supported in $(0,1)$ is replaced by $\chi$ being supported in  $[0,1]$). 

 We can now use standard arguments  to deduce global existence and pathwise uniqueness of a solution to  system  \eqref{eq:HH-model}. Indeed,  setting   $\theta_M = \inf\{t\geq0: |X_t^{(M)}|\geq M\}$, using  the global Lipschitz character of its coefficients together with It\^o calculus  and Gronwall's lemma we get for every $M'>M$ that    a.s. for all  $t\geq 0$,
$X^{(M)}_{t\wedge \theta_M} = X_{t\wedge \theta_M}^{(M')}$.  This implies that $\theta_{M'}>\theta_M$ a.s.\ and allows us to unambiguously define a process $X$ solving \eqref{eq:HH-model}
on  the random interval $[0,\theta)$, with $\theta:= \sup_{M>0}\theta_M$, by   $X_t= X^{(M)}_t$ for all $t\in [0,\theta_M]$. 
 On the other hand,  since  $|p_M^1(z)|\leq |z|$ for all $z\in \R$,  for two constants $C_1,C_2>0$ not depending on $M>0$ we have $|F_M(v,m,n,h)|\leq C_1+ C_2 |v|$ for every $(v,m,n,h)\in \R \times [0,1]^3$.  Using this control on the right hand side of the  equations for $V^{(1,M)},\dots, V^{(N,M)}$ in \eqref{eq:truncated-HH-model}
and Gronwall's lemma we get $$\E\left[ |X_{t\wedge \theta_M}| \right] \leq C(t),$$
for some constant $C(t)>0$ not depending on $M$. This yields $M\P\left[ \theta_M<t  \right]\leq C(t) $, 
whence  $\P\left[ \theta<\infty \right] =0$  letting $M$ and then $t \nearrow \infty$ . The statement follows. 
\end{proof}

\begin{rem}

 \begin{itemize}
\item[i)] The arguments given in the previous proof also show that each of the functions $\sigma_x$ is locally Lipschitz on $\R\times [0,1]$. 
\item[ii)] The same proof also works for some extensions of our model. For instance, if independent Brownian motions are added to each of the voltage processes. 
\end{itemize}
\end{rem}

\medskip 

 We next show that under the additional Hypothesis \ref{hyp:V0}, each of the voltage processes is bounded uniformly in time and in $N$. Below and in all the sequel we denote
$$ \Vmax_{t,\infty}:= \max_{i=1,\dots,N} \sup_{s\in [t,\infty)} | V_s^{(i)}|.$$
We also set
$$\Rmax:= \max_{r,s,u \in[0,1]} |I+\gNa\VNa r +\gK \VK s+ \gL \VL + \JCh\Vrev u |. $$

\begin{prop}\label{prop:boundedness-of-V^i}  Under Hypothesis \ref{hyp:V0},   for every $N\geq 1$ and $t\geq 0$   we have a.s.  
$$ 
\left|\meanV_t \right|  \leq \Vmax_0  e^{-\gL t} +\frac{ 2 \Rmax}{\gL}(1-e^{-\gL t})
$$ 
and 
\begin{equation}\label{eq:uniform-bound-for-voltages}
\Vmax_{t,\infty}\leq V^*_t:= \frac{4\Rmax}{\gL}+2 \Vmax_0 e^{-\gL t}.
\end{equation}
 As a consequence  for every $N\geq 1$, there exists at least one invariant law $\mu^N_{\infty}$ for the solution to \eqref{eq:HH-model}, namely there exists a solution $(X_t, t\geq 0 )$  to \eqref{eq:HH-model} such that $X_t$  has law $\mu^N_{\infty}$ for all $t\geq 0$ as soon as $X_0$ has law $\mu^N_{\infty}$. Moreover, this invariant measure is exchangeable.
\end{prop}

\begin{rem}\label{rem:after-prop-boundedness-of-V^i}
~
\begin{itemize}
\item[i)]  The bound $V_t^* $  on $ \Vmax_{t,\infty}$  is in general not optimal. For instance, if $\Vmax_0<  \frac{ 2 \Rmax}{g_L}$, one can choose $\Vmax_0>  \frac{ 2 \Rmax}{g_L + \JE}$ and get  from the last identity in \eqref{eq:boundsV} that   $ \Vmax_{0,\infty}\leq \frac{4 \Rmax}{g_L}<V_0^*$. However, in order to state a synchronization result that holds  for a general class of initial conditions  $V_0$, the   fact that the bound  $V^*_t$ does not depend on the electrical connectivity  $\JE$ and that  $V^*_{\infty}:= \lim_{t \to \infty} V^*_t= \frac{4 \Rmax}{g_L}$ does not depend on the initial condition will be crucial. See point i) in Remark \ref{comments_post_synchro} for a related discussion. 

\item[ii)]   If  point 2) of Hypothesis \ref{hyp:V0}   does not hold, by  slightly modifying the arguments of Lemma \ref{prop:boundedness-of-V^i} we still   can get the a.s.\ bound 
$$ \left|V_t^{(i)}\right|\leq \frac{4\Rmax}{\gL}+2  \frac{  | V_0|}{\sqrt{N}}  e^{-\gL t}  \, $$
 implying a uniform in $N$ bound  for $\E( \Vmax_{t,\infty})$ if for instance all the random variables $V_0^{i,N}$, $i=1,\dots,N$, $N\geq 1$   are equal in law and have finite second moment. However, we have not been able to  fully extend our results to such a framework. 
 \item[iii)]\new{ The same arguments also show that a bound like \eqref{eq:uniform-bound-for-voltages}  holds with $\Vmax_{t,\infty}$ replaced by
 $$\widehat{\Vmax}_{t,\infty}:= \max_{i=1,\dots,N} \sup_{s\in [t,\infty)} |\widehat{ V}_s^{(i)}|.$$
 That is, the voltages obtained with the EPE scheme are also uniformly bounded. 
    }
\end{itemize}
\end{rem}

\new{ In the proof of Proposition \ref{prop:boundedness-of-V^i} and later,} we will make use of the 
 the following version of Gronwall's lemma (see  e.g. Ambrosio et al. \cite[p. 88]{Ambrosio:2008aa}).
\begin{lem}\label{lem:gronwallAmbrisio}
Let $\theta:[0,+\infty)\to\R$ be a locally absolutely continuous function and $a, b \in L_{\text{loc}}^1([0,+\infty))$ be given functions satisfying, for $\lambda\in\R$,
$$\frac{d}{dt}\theta^2(t)+2\lambda \theta^2(t)\leq a(t) + 2b(t)\theta(t)\;\text{for }\mathcal{L^1}-a.e.\;t>0. $$
Then for every $T>0$ we have
$$e^{\lambda T}|\theta(T)|\leq \left( \theta^2(0)+ \sup_{t\in[0,T]}\int_{0}^{t}{e^{2\lambda s}a(s)ds}\right)^{1/2}+2\int_{0}^{T}{e^{\lambda t}|b(t)|dt}.$$
\end{lem}

\begin{proof}[Proof of Proposition \ref{prop:boundedness-of-V^i}]
Setting 
$$R^{(i)}_s:=  I+\gNa\VNa\left[ m^{(i)}_s\right]^3h^{(i)}_s+\gK \VK \left[ n^{(i)}_s\right]^4 + \gL \VL + \JCh\Vrev\meanY_s,   \mbox{ and }$$
$$A^{(i)}_s: = \gNa \left[ m^{(i)}_s\right]^3h^{(i)}_s+\gK  \left[ n^{(i)}_s\right]^4 + \gL  + \JCh \bar{y}_s^N ,$$
the dynamics of the potential can be written as
\begin{equation*}
V_t^{(i)} = V_0^{(i)} + \int_{0}^{t}{R^{(i)}_s - A^{(i)}_sV_s^{(i)} -  \JE V^{(i)}_s+ \JE\meanV_sds}.
\end{equation*}
Therefore, we get
\begin{equation}\label{eq:dinamic-for-V-boudedness}
\left( V_t^{(i)}\right)^2 = \left( V_0^{(i)}\right)^2 + 2\int_{0}^{t}{R^{(i)}_s  V^{(i)}_s - A^{(i)}_s\left( V_s^{(i)}\right)^2 -  \JE\left(  V^{(i)}_s\right)^2+ \JE\meanV_s V^{(i)}_sds}.
\end{equation}
and 
$$| V_t|^2 = |V_0|^2 + 2\int_{0}^{t}{\sum_{i=1}^N\left[ R^{(i)}_s  V^{(i)}_s - A^{(i)}_s\left( V_s^{(i)}\right)^2\right] -  \JE|V_s|^2+ N\JE(\meanV_s)^2ds}.$$
Notice that 
\begin{align*}
(\meanV_s)^2&=\frac{1}{N^2}\sum_{i,j=1}^{N}V^{(i)}_sV^{(j)}_s \leq \frac{1}{2N^2}\sum_{i,j=1}^{N}(V^{(i)}_s)^2 + (V^{(j)}_s)^2 = \frac{1}{N}|V_s|^2, 
\end{align*}
which yields 
$$\frac{d}{dt}| V_t|^2  + 2\gL | V_t|^2 \leq  2|R_t||  V_t|.$$
By Lemma \ref{lem:gronwallAmbrisio} we  deduce that
$$| V_t| \leq | V_0| e^{-\gL t} + 2 e^{-\gL t}\int_{0}^{t}{e^{\gL s}|R_s|ds}.$$
Since   $|R^{(i)}_s|\leq \Rmax, $
 we then get
$$ \left|\meanV_t \right|  \leq  \frac{  | V_t|}{\sqrt{N}}  \leq\frac{  | V_0|}{\sqrt{N}}   e^{-\gL t} +\frac{ 2 \Rmax}{\gL}(1-e^{-\gL t})\leq  \Vmax_0  e^{-\gL t} +\frac{ 2 \Rmax}{\gL}(1-e^{-\gL t}).
$$
which is the first desired inequality. Using this in \eqref{eq:dinamic-for-V-boudedness} yields
\begin{align*}
\frac{d}{dt}(V_t^{(i)})^2 + 2(\gL+\JE)(V_t^{(i)})^2&\leq2|R^{(i)}_t + \JE\meanV_t||V_t^{(i)}|\\
& \leq 2 \left( \Rmax + \JE \left( \Vmax_0 -\frac{ 2 \Rmax}{\gL}\right)e^{-\gL t} +\frac{ 2 \JE\Rmax}{\gL}\right)|V_t^{(i)}|.
\end{align*}
Applying once again Lemma \ref{lem:gronwallAmbrisio}, we obtain
\begin{align}\label{eq:boundsV}
\left|V_t^{(i)}\right|&\leq \Vmax_0 e^{-(\gL+\JE)t}\nonumber\\
&\quad+2e^{-(\gL+\JE)t}\int_{0}^{t}{e^{(\gL+\JE)s} \left( \Rmax\left(\frac{\gL + 2 \JE}{\gL}\right) + \JE \left( \Vmax_0 -\frac{ 2 \Rmax}{\gL}\right)e^{-\gL s}\right) ds}\nonumber\\
&\quad=\Vmax_0 e^{-(\gL+\JE)t}+2\Rmax\left(\frac{\gL + 2 \JE}{\gL(\gL+\JE)}\right)(1- e^{-(\gL+\JE)t})\nonumber \\
&\quad \quad   +2 \left( \Vmax_0 -\frac{ 2 \Rmax}{\gL}\right)\left( e^{-\gL t} -e^{-(\gL+\JE)t}\right)\nonumber\\
&\quad=\frac{2\Rmax}{\gL}\left(\frac{\gL + 2 \JE}{\gL+\JE}\right)+2 \left( \Vmax_0 -\frac{ 2 \Rmax}{\gL}\right)e^{-\gL t}\nonumber \\
&\quad \quad+ \left(\frac{ 2 \Rmax}{\gL+\JE} - \Vmax_0  \right) e^{-(\gL+\JE)t}\nonumber\\
&\leq \frac{4\Rmax}{\gL}+2 \Vmax_0 e^{-\gL t}=V_t^* 
\end{align} 
which implies the asserted bounds on $ \Vmax_{t,\infty}$. 
\smallskip

Let  us now deduce the existence of an invariant distribution which is exchangeable. Let $P_t^N$ denote the semigroup associated to the solution of  \eqref{eq:HH-model}, that is for  each $\mathcal{X}\in(\R\times[0,1]^4)^N$  and  $B$ Borel set of $(\R\times[0,1]^4)^N$, 
$$P_t^N(\mathcal{X},B)=\P\left(X_t\in B \vert X_0 = \mathcal{X}\right).$$
Consider also the probability measure $R_T^N(\lambda)$ on $(\R\times[0,1]^4)^N$, defined for any  law $\lambda$  as
$$R_T^N(\lambda)(B)=\int_{ (\R\times[0,1]^4)^N} \left(\frac{1}{T}\int_{0}^{T} P_t^N(\mathcal{X},B) dt \right) \lambda(d\mathcal{X}).$$
Since the voltage component  is uniformly bounded in time, by \eqref{eq:boundsV}, the solution to \eqref{eq:HH-model} lies in the compact set $([-4\tfrac{\Rmax}{\gL}-2\Vmax_0, 4\tfrac{\Rmax}{\gL}+2\Vmax_0]\times[0,1]^4)^N$, and then for any $(T_M)\nearrow\infty$, and any $\lambda$ with compact support, the sequence $(R_{T_M}^N(\lambda),M\geq 0)$ is tight and has a subsequence weakly converging to some probability measure $\mu_{\infty}^N$. According to Krylov-Bogoliubov Theorem,  $\mu_{\infty}^N$ is invariant for $P_t^N$.

Let us now  choose  and exchangeable initial  law $\lambda$. For  any measurable and bounded function $\psi$, the identity 
$$  \int_{(\R\times[0,1]^4)^N}P_t^N(\mathcal{X},dy) \psi(y) \lambda(d\mathcal{X}) =   \int_{(\R\times[0,1]^4)^N}P_t^N(\mathcal{X},dy) (\psi\circ \pi)(y) \lambda(d\mathcal{X}).$$
for any $N$-permutation $\pi$ of the coordinates follows directly from the exchangeable structure of the system of equations  \eqref{eq:HH-model}. 
Therefore, $R_{T_M}^N(\lambda)$ is  exchangeable for any $T_M$, and the corresponding $\mu_{\infty}^N$ is  exchangeable  as  the weak limit of exchangeable measures.

\end{proof}

\section{Synchronization: proof of Theorem  \ref{theo:synchro-empiricalVar} \textit{a)}}\label{sec:synch}

In the sequel,  for any locally bounded real function $f$ on $\R$ and  each $R>0$ we will write
$$ \|f \|_{\infty,R}:= \sup_{v\in[-R,R]}|f (v)|.$$

 We will repeatedly use a  simple  control of the increments of  the function $F$, stated in  next  lemma for  convenience:
 
 \begin{lem} We have 
\begin{equation}\label{eq:boundForTheIncrementOfTheDriftF}
\begin{aligned}
\left( F(V_1,m_1,n_1,h_1)-F(V_2,m_2,n_2,h_2)\right)(V_1-V_2) &\leq - \gL(V_1-V_2)^2\\
&\quad  + 4\gK|V_2-\VK| | n_1 - n_2||V_1-V_2| \\
&\quad + 3\gNa|V_2-\VNa||m_1-m_2||V_1-V_2|\\
&\quad + \gNa|V_2-\VNa| |h_1-h_2||V_1-V_2|.
\end{aligned}
\end{equation}
for every  $m_i,n_i,h_i\in[0,1]$ and $V_i\in \R$,   $i=1,2$.
 \end{lem}
 \begin{proof}
Since
$$x^4-y^4 = (x^2+y^2)(x+y)(x-y),$$
and
$$x^3u-y^3v = u(x^2+xy+y^2)(x-y)+y^3(u-v),$$
we get
\begin{equation*}
\begin{aligned}
F(V_1,m_1,n_1,h_1)-F(V_2,m_2,n_2,h_2) &= - (\gK n^4_1+\gNa m^3_1h_1+\gL)(V_1-V_2)\\
&\quad  - \gK(V_2-\VK)( n^2_1 + n^2_2)( n_1 + n_2) ( n_1 - n_2) \\
&\quad - \gNa(V_2-\VNa)h_1 (m^2_1+m_1m_2+m_2^2)(m_1-m_2)\\
&\quad - \gNa(V_2-\VNa) m^3_2(h_1-h_2)
\end{aligned}
\end{equation*}
and the asserted bound follows.  
\end{proof}
 
The following result is the core of the proof of Theorem  \ref{theo:synchro-empiricalVar}:

\begin{prop}\label{prop:ElectricalSynchronization}
 For each  $V^*>0 $, there are  constants $\JE^*>0$  and $\lambda^*>0$ not depending on $N$ nor on $\sigma$ such that for each  $\JE>\JE^*$ and any solution $X$  of \eqref{eq:HH-model}  satisfying   $\Vmax_{0,\infty}\leq V^*$,  one has
\begin{equation*}
\E\left( |X_t^{(i)} -X_t^{(j)} |^2\right)\leq\E\left(  |X_0^{(i)} -X_0^{(j)} |^2 \right) e^{-\lambda^* t} + \sigma^2 \frac{2 C^*_{\zeta,\rho}}{\lambda^*} \, \quad \forall \,t\geq 0, 
\end{equation*}
for all $ i,j\in\{1,\ldots,N\}$,  where $$C^*_{\zeta,\rho}=\sum_{x=m,n,h,y}  \|\rho_x \vee \zeta_x \|_{\infty,V^*} <\infty . $$
 \end{prop}

\begin{proof} 
Let us write $\D V_t = V_t^{(i)}-V_t^{(j)}$ and $\D x_t =  x_t^{(i)}-x_t^{(j)} $. Thanks to the bound \eqref{eq:boundForTheIncrementOfTheDriftF}, we have
\begin{align*}
(\D V_t)^2&= (\D V_0)^2+ 2 \int_{0}^{t}{[F(V^{(i)}_s,m_s^{(i)},n_s^{(i)},h_s^{(i)})-F(V^{(j)}_s,m_s^{(j)},n_s^{(j)},h_s^{(j)})]\D V_sds}\\
&\quad-\int_{0}^{t}{(2\JE +2\JCh\meanY_s ) (\D V_s )^2ds}\\
&\leq (\D V_0)^2+  \int_{0}^{t}{  8\gK|V^{(j)}_s-\VK| | \D n_s||\D V_s| +  6\gNa|V^{(j)}_s-\VNa|| \D m_s||\D V_s|ds}\\
&\quad + \int_{0}^{t}{   2\gNa|V^{(j)}_s-\VNa| | \D h_s||\D V_s|ds}-\int_{0}^{t}{(2\gL+ 2\JE +2\JCh\meanY_s ) (\D V_s )^2ds}\\
&\leq  (\D V_0)^2 +\int_{0}^{t}{\eps_m\left( \D m_s\right)^2+\eps_n \left( \D n_s\right)^2+\eps_h \left( \D h_s\right)^2 ds}  \\
&\quad-\int_{0}^{t}{\left( 2\gL + 2\JE +2\JCh\meanY_s  - \frac{9 M_\text{Na}^2}{\eps_m}-\frac{16M_\text{K}^2}{\eps_n} -\frac{M_\text{Na}^2}{\eps_h}  \right)\left( \D V_s\right)^2ds},
\end{align*}
where we have used Young's inequality: $ab\leq   \eps_x a^2 +   \frac{b^2}{4 \eps_x }$  for $x=m,n,h,y$,  with $\eps_x >0$ to be chosen later, and where  we have set
$$M_\text{Na} = \gNa\max_{v\in[-V^*,V^*]}|v-\VNa|,\;\;M_\text{K} =\gK\max_{v\in[-V^*,V^*]}|v-\VK|.$$

On the other hand, for the channel types $x=m,n,h,y$, we have
\begin{align*}
\E\left[ ( \D x_t)^2\right]&=\E\left[ (\D x_0)^2\right]+2\int_{0}^{t}{\E\left[  (1-x^{(i)}_t)(\rho_x(V_t^{(i)})-\rho_x(V_t^{(j)}))\D x_s\right]ds}\\
&\quad -2\int_{0}^{t}{\E\left[ x^{(i)}_t(\zeta_x(V_t^{(i)})-\zeta_x(V_t^{(j)}))\D x_s\right]ds}\\
&\quad-2\int_{0}^{t}{\E\left[ \left( \rho_x(V_s^{(j)}) + \zeta_x(V_s^{(j)}) \right)(\D x_s)^2\right]ds}\\
&\quad+\int_{0}^{t}{\E\left[ \sigma_x^2( V_s^{(j)}, x^{(j)}_s) + \sigma_x^2(V_s^{(i)},x^{(i)}_s)\right]ds}.
\end{align*}
By our assumptions,  for all $t\geq 0$ we have fror $k=i,j$, 
$$\sigma_x^2(V^{(k)}_t,x_t)\leq \sigma^2{\left( (1-x^{(k)}_t)\rho_x(V^{(k)}_t)+x^{(k)}_t\zeta_x(V^{(k)}_t)\right)}\leq \sigma^2 \|\rho_x \vee \zeta_x \|_{\infty,V^*} .$$
Using Young's inequality in the same way as before yields 
\begin{align*}
\E\left[ ( \D x_t)^2\right]&\leq \E\left[ (\D x_0)^2\right]+\int_{0}^{t}{\E\left[  \frac{(L^*_{\rho_x}+L^*_{\zeta_x})^2}{\eps_x}(\D V_s)^2\right]ds}\\
&\quad-  \left(2\eta_x-\eps_x\right) \int_{0}^{t}{\E\left[ (\D x_s)^2\right]ds}+ 2  t\, \sigma^2 \|\rho_x \vee \zeta_x \|_{\infty,V^*},
\end{align*}
where $L_f^*$ denotes the Lipschitz constant  on $[-V^*,V^*]$ of a locally Lipschitz function  $f$, and where
$$\eta^*_x:=    \inf_{v\in[-V^*,V^*]}\left\{\rho_x(v)+\zeta_x(v)\right\}>0 .$$
Adding up, we get 
\begin{align*}
\E\left( |X_t^{(i)} -X_t^{(j)} |^2\right)&\leq \E\left[ |X_0^{(i)} -X_0^{(j)} |^2\right]\\
&\quad -\int_{0}^{t}\E\Big[ \Big(2\gL + 2\JE - \frac{9 M_\text{Na}^2+ (L^*_{\rho_m}+L^*_{\zeta_m})^2}{\eps_m} -  \frac{16 M_\text{K}^2 +(L^*_{\rho_n}+L^*_{\zeta_n})^2}{\eps_n} \\
&\quad\quad\qquad-  \frac{M_\text{Na}^2 +(L^*_{\rho_h}+ L^*_{\zeta_h})^2}{\eps_h}-  \frac{(L^*_{\rho_y}+L^*_{\zeta_y})^2}{\eps_y} \Big)(\D V_s)^2\Big]ds\\
&\quad-   \left(2\eta_m -2\eps_m\right) \int_{0}^{t}{\E\left[(\D m_s)^2\right]ds}-  \left(2\eta_n -2\eps_n\right) \int_{0}^{t}{\E\left[(\D n_s)^2\right]ds}\\
&\quad-\left(2\eta_h -2\eps_h\right) \int_{0}^{t}{\E\left[ (\D h_s)^2\right]ds}- \left(2\eta_y -\eps_y\right) \int_{0}^{t}{\E\left[ (\D y_s)^2\right]ds}\\
&\quad+ 2 t\, \sigma^2 C^*_{\zeta,\rho}.
\end{align*}
Define now $\lambda^*$ as the optimal value of the problem
$$ \max_{J,\eps_m,\eps_n,\eps_h,\eps_y  >0} \Psi(J,\eps_m,\eps_n,\eps_h,\eps_y ) \, ,$$
where
\begin{multline*}
 \Psi(J,\eps_m,\eps_n,\eps_h,\eps_y ):= \\
 \min\Bigg\{2\gL + 2J - \frac{12M_\text{Na}^2+ (L^*_{\rho_h}+L^*_{\zeta_h})^2}{\eps_m} -  \frac{16M_\text{K}^2 +(L^*_{\rho_n}+L^*_{\zeta_n})^2}{\eps_n} -  \frac{M_\text{Na}^2 +(L^*_{\rho_h}+L^*_{\zeta_h})^2}{\eps_h}-  \frac{(L^*_{\rho_y}+L^*_{\zeta_y})^2}{\eps_y}, \\
\quad\quad\quad\quad2\eta^*_m -2\eps_m,\,2\eta^*_n -2\eps_n, 2\eta^*_h -2\eps_h, 2\eta^*_y -\eps_y\Bigg \}. 
\end{multline*}
Notice that $\lambda^*$ is strictly positive since $\Psi(J,\eps_m,\eps_n,\eps_h,\eps_y )$  can be made so by  taking small enough $\eps_x>0$ for $x=m,n,h,y$ and then large enough $J>0$.  Calling $\JE^*$ the  smallest $J>0$ such that $(J,\eps_m,\eps_n,\eps_h,\eps_y )\in \arg\max \Psi$,  it follows that for every $\JE>\JE^*$, 
$$\E\left( |X_t^{(i)} -X_t^{(j)} |^2\right)\leq \E\left[ |X_0^{(i)} -X_0^{(j)} |^2\right]-\lambda^* \int_{0}^{t}{\E\left( |X_s^{(i)} -X_s^{(j)} |^2\right)}+ 2 t\, \sigma^2C_{\zeta,\rho}.$$
Applying Lemma \ref{lem:gronwallAmbrisio}, we obtain
$$\sqrt{\E\left( |X_t^{(i)} -X_t^{(j)} |^2\right)} \leq e^{-\frac{\lambda^* t}{2}}\left( \E\left[ |X_0^{(i)} -X_0^{(j)} |^2\right] + \frac{(e^{\lambda^* t}-1)}{\lambda^*}2 \sigma^2C^*_{\zeta,\rho} \right)^{1/2},$$
and the desired result.
\end{proof}

The next result removes the dependance of the previous one on the bound $V^*$,   at the price of ensuring   exponentially  fast synchronization  only from some time instant $t_0\geq  0$ on. It will then be  easy to deduce  part a)  of  Theorem  \ref{theo:synchro-empiricalVar}.
\begin{theo}\label{theo:ElectricalSynchronization}
There are  constants $\JE^0>0$ and  $\lambda^0>0$ not depending on $N\geq 1$, on $\sigma\geq 0$ nor on the initial data, and  $t_0\geq 0$ not depending on $N\geq 1$ nor on $\sigma\geq 0$, such that for each  $\JE>\JE^0$  the solution $X$ of \eqref{eq:HH-model}  satisfies, for every $t\geq t_0$, 
\begin{equation*}
\E\left( |X_t^{(i)} -X_t^{(j)} |^2\right)\leq\E\left(  |X_{t_0}^{(i)} -X_{t_0}^{(j)} |^2 \right) e^{-\lambda^0( t-t_0)} + \sigma^2 \frac{2 C^0_{\zeta,\rho}}{\lambda^0} \, , \quad \forall \, i,j\in\{1,\ldots,N\}, 
\end{equation*}
 where
$$C^0_{\zeta,\rho}:=\sum_{x=m,n,h,y}  \|\rho_x \vee \zeta_x \|_{\infty,\frac{ 5\Rmax}{g_L}} 
<\infty  \, .$$
\end{theo}
\begin{proof}
Fix  $\epsilon_0 \in (0,1)$, take $t_0\geq 0$ such that $2 \Vmax_0 e^{-\gL t_0}\leq   \epsilon_0 \frac{\Rmax}{g_L}  $ and, conditionally on the sigma-field generated by $(X_s:s \leq t_0)$, apply Proposition \ref{prop:ElectricalSynchronization}  to the shifted process $X':=(X_{t+t_0}:t\geq 0)$  with $V^*=V^*_{t_0}\leq (4+\epsilon_0)\frac{ \Rmax}{g_L}\leq 5\frac{ \Rmax}{g_L} $. The proof is  then achieved  taking expectation in the obtained inequality.  
\end{proof}

We can now finish the proof of  Theorem  \ref{theo:synchro-empiricalVar}.\ a). Here and in the sequel 
we denote by   $\varV_t$ and $\varX_t$ the empirical variance of  voltages  and  $x$ type channels at time $t$,  respectively: 
$$\varV_t = \frac{1}{N}\sum_{i=1}^N \left( V_t^{(i)}-\meanV_t\right)^2   \;\; \mbox{ and }\;  \varX_t = \frac{1}{N}\sum_{i=1}^N\left( x_t^{(i)}-\meanX_t\right)^2.$$
\begin{proof}[Proof of  Theorem  \ref{theo:synchro-empiricalVar}.\ a)]   Applying in the conclusion of  Theorem \ref{theo:ElectricalSynchronization} the elementary identity  
$$\frac{1}{N^2}\sum_{i,j=1}^N{(\alpha_i-\alpha_j)^2} = \frac{2}{N}\sum_{k=1}^N{(\alpha_k-\bar{\alpha}^N)^2} \quad   \mbox{ for every  }\alpha_1,\ldots,\alpha_N\in \R,  $$
with  $\bar{\alpha}^N=\frac{1}{N} \sum_{i=1}^N \alpha_i$
we get 
\begin{equation*}
\E\left(  \varV_t + \sum_{x=mn,n,h,y} \varX_t   \right)\leq\E\left(  \varV_{t_0} +\sum_{x=mn,n,h,y} \varX_{t_0} \right)e^{-\lambda^0( t-t_0)} + \sigma^2 \frac{C^0_{\zeta,\rho}}{\lambda^0}
\end{equation*}
in the general case. If,  additionally, exchangeability of the initial condition is assumed, the path law of system \eqref{eq:HH-model} is exchangeable,  by pathwise uniqueness. The asserted inequality follows. 

\end{proof}
\begin{rem}\label{comments_post_synchro}
$  $
 \begin{itemize}
\item[i)] Theorems  \ref{theo:synchro-empiricalVar} 
and \ref{theo:ElectricalSynchronization} show that, for large enough $\JE$,  synchronization of the network \eqref{eq:HH-model} always  occurs, as long as the  initial voltage $V_0$ is bounded, but regardless of its actual values. 
 More precisely, the time $t_0>0$  which
 depends on $\Vmax_0$, on $\frac{\Rmax}{g_L}$  and on some arbitrary choice of the parameter $\epsilon_0>0$, but not on $\JE$, is one possible  time after which we can grant  that  the  voltage trajectories  stay in  some fixed  interval  not depending on $V_0$.  Then, after $t_0$ and  if  $\JE$  was chosen large enough, synchronization occurs at least at the exponential rate $\lambda^0$ which depends on coefficients of the system \eqref{eq:HH-model} but no longer on the initial data.  In turn,  for large enough $\JE$, Proposition \ref{prop:ElectricalSynchronization} ensures synchronization from $t_0=0$ on   but only if    $\Vmax_0$ is small enough. 

\item[ii)]
Notice that  the function $\Psi$ in the proof of Proposition \ref{prop:ElectricalSynchronization} (and hence the constant $\lambda^*$ therein)  increases when its parameter  $V^*$ decreases, whereas $C^*_{\zeta,\rho}$ decreases when $V^*$ does. Therefore, letting $\epsilon_0\to 0$ (or $t_0\to \infty$) yields the best (by this approach) bounds  for the $\limsup$ in Theorem  \ref{theo:synchro-empiricalVar}. Moreover, the largest possible exponential rate $\lambda^0>0$  and the smallest possible interaction strength $\JE^0\geq 0 $  that can be  obtained   (but not necessarily attained) in Theorems  \ref{theo:synchro-empiricalVar}  and  \ref{theo:ElectricalSynchronization}  by our approach  are $\lambda^*$ and $\JE^*$ corresponding to $V^*=\frac{4 \Rmax}{g_L}$. These choices are  certainly  not optimal in general.
\end{itemize}
 \end{rem}

\section{Synchronized dynamics: proof of Theorem  \ref{theo:synchro-empiricalVar} \textit{b)}}

Our  next goal is to prove part b)  of  Theorem  \ref{theo:synchro-empiricalVar}. 
\begin{rem}\label{boundVt_1*}
Proceeding in a similar way as in the proof of Proposition \ref{prop:boundedness-of-V^i}  one checks  that  the process  \eqref{eq:HatProcess} satisfies
$\frac{d}{dt}| \widehat{V}_t|^2_2  + 2\gL | \widehat{V}_t |^2_2 \leq  2\Rmax |  \widehat{V}_t|, $
which now yields, for any $t\geq t_1$, 
\begin{equation*}| \widehat{V}_t  |\leq | \widehat{V}_{t_1}|  e^{-\gL( t-t_1) } +\frac{ 2 \Rmax}{\gL}(1-e^{-\gL ( t-t_1)}). 
\end{equation*}
Applying on $\bar{V}^N_{t_1}= \widehat{V}_{t_1}$      the  first bound  in Lemma \ref{prop:boundedness-of-V^i}  we get that $| \widehat{V}_t  |\leq  \Vmax_0 e^{-\gL  t } +\frac{ 2 \Rmax}{\gL}$ for every $t\geq t_1$.  Thus, if 
 $t_0\geq 0$ is  chosen as  in Theorem \ref{theo:ElectricalSynchronization}, we deduce that
\begin{equation}\label{boundhatV}
\max\left\{ \sup_{s\in [t_1,\infty)} | \bar{V}_s |, \sup_{s\in [t_1,\infty)} | \widehat{V}_s |\right\} \leq   \frac{V^*_{t_0}}{2} \leq\frac{ 5\Rmax}{2 g_L} \, .
\end{equation} 
\end{rem}

We first prove
\begin{prop}\label{prop:HatProcessIsCloseToMeanProcess} Let   $t_0$  be as in Theorem \ref{theo:ElectricalSynchronization} and $\delta>0$. There are constants $K_{1,\delta}, K_{2,\delta}>0$ increasingly depending  on $\delta>0$, but not depending on $N$ nor on the initial condition,  such that for each $t_1\geq t_0$,  
\begin{equation}\label{eq:barXhatX}
\begin{aligned}
\E\left( \sup_{t_1\leq t \leq t_1+\delta}|\bar{X}^{N,t_1}_t-\hX^N_t|^2\right) &\leq  \left(\left[ (V^*_{t_0} )^2 + 4\right] e^{-\lambda^0(t_1-t_0)}  +  \frac{\sigma^2 C^0_{\zeta,\rho}}{\lambda^0 }\right) \, \delta K_{1,\delta}  + \delta  K_{2,\delta}  \frac{\sigma^2  }{N} C^0_{\zeta,\rho} . 
  \end{aligned}
\end{equation}
\end{prop}

\begin{proof} For notational simplicity we write  in the proof $\hX_{t}^{N}: = \hX_{t}^{N,t_1} $. Notice that the average process satisfies the dynamics
\begin{align*}
\meanV_t&=  \meanV_{t_1} + \int_{t_1}^{t} \frac{1}{N}\sum_{i=1}^N{F(V_s^{(i)}, m_s^{(i)},n_s^{(i)},h_s^{(i)})} - \JCh \meanY_s (\meanV_s-\Vrev)ds \\
\meanX_t &= \meanX_{t_1}+\frac{1}{N}\sum_{j=1}^N\int_{t_1}^{t}\rho_x(V^{(j)}_s)(1-x^{(j)}_s)  -\zeta_x(V^{(j)}_s)x^{(j)}_s	+\frac{1}{N}\sum_{i=1}^N\int_{t_1}^{t}{\sigma_x(V_s^{(j)},x_s^{(j)})dW_s^{x,j}}.
\end{align*}
Therefore, after some manipulations, we get that
\begin{align*}
\left( \meanV_t - \hV_t\right)^2&=\int_{t_1}^{t}  \left[ \frac{1}{N}\sum_{i=1}^N{F(V_s^{(i)}, m_s^{(i)},n_s^{(i)},h_s^{(i)})-F(\meanV_s, \meanM_s,\meanN_s,\meanH_s)}  \right]^2 ds\\
&\quad+2\int_{t_1}^{t}  \left[F(\meanV_s, \meanM_s,\meanN_s,\meanH_s) - F(\hV^{N}_s,\hm_s^N,\hn_s^N,\hh_s^N) \right](\meanV_s-\hV^{N}_s) ds\\
&\quad + \int_{t_1}^{t}{(1-2\JCh \meanY_s )(\meanV_s-\hV^{N}_s)^2 +2\JCh (\hV^{N}_s-\Vrev)(\meanY_s- \hy_s^N)(\meanV_s-\hV^{N}_s) ds}\\
&= I_1+I_2+I_3.
\end{align*}
By Jensen's inequality and the bound \eqref{boundhatV} we have
\begin{align*}
I_1&\leq \int_{t_1}^{t}  \frac{1}{N}\sum_{i=1}^N{\left[ F(V_s^{(i)}, m_s^{(i)},n_s^{(i)},h_s^{(i)})-F(\meanV_s, \meanM_s,\meanN_s,\meanH_s)\right]^2}  ds\\
&\leq \int_{t_1}^{t}  \frac{1}{N}\sum_{i=1}^N{4\left[\left( \gK (n^{(i)}_s)^4+\gNa (m^{(i)}_s)^3 h^{(i)}_s+\gL\right)(V_s^{(i)}-\meanV_s)  \right]^2}  ds\\
&\quad+\int_{t_1}^{t}  \frac{1}{N}\sum_{i=1}^N{4\left[\gNa(\meanV_s-\VNa)h^{(i)}_s \left( (m^{(i)}_s)^2 +m^{(i)}_s\meanM_s+(\meanM_s)^2 \right)(m_s^{(i)}-\meanM_s)  \right]^2}  ds\\
&\quad+\int_{t_1}^{t}  \frac{1}{N}\sum_{i=1}^N{4\left[\gNa(\meanV_s-\VK)\left( (n^{(i)}_s)^2 +(\meanN_s)^2\right)\left( n^{(i)}_s +\meanN_s\right)(n_s^{(i)}-\meanN_s)  \right]^2}  ds\\
&\quad+ \int_{t_1}^{t}  \frac{1}{N}\sum_{i=1}^N{4\left[\gNa(\meanV_s-\VNa)(\meanM_s)^3 (h_s^{(i)}-\meanH_s)  \right]^2}  ds\\
&\leq K_V^1 \int_{t_1}^{t}{\varV_s+\varM_s+\varN_s+\varH_sds},
\end{align*}
with $K_V^1$ explicitly depending on 
$ \sup_{v\in[-\frac{ 5\Rmax}{2 g_L} ,\frac{ 5\Rmax}{2 g_L} ]} \max\{  |v-\VNa|,  |v-\VK| \}$,  $g_{\text{K}}$ and $g_{\text{Na}}$.
 Meanwhile, using \eqref{eq:boundForTheIncrementOfTheDriftF} we get
\begin{align*}
I_2&\leq \int_{t_1}^{t}  
-2\gL(\meanV_s-\hV^{N}_s)^2+4\gK|\hV^{N}_s-\VK|\left((\meanV_s-\hV^{N}_s)^2+(\meanN_s-\hn^{N}_s)^2 \right)\\
&\quad\quad +3\gNa|\hV^{N}_s-\VNa|\left((\meanV_s-\hV^{N}_s)^2+(\meanM_s-\hm^{N}_s)^2 \right)\\
&\quad\quad+\gNa|\hV^{N}_s-\VNa|\left((\meanV_s-\hV^{N}_s)^2+(\meanH_s-\hh^{N}_s)^2 \right) ds \\
&\leq K_V^2\int_{t_1}^{t}  
(\meanV_s-\hV^{N}_s)^2+(\meanN_s-\hn^{N}_s)^2+(\meanM_s-\hm^{N}_s)^2 +(\meanH_s-\hh^{N}_s)^2  ds,
\end{align*}
with $K_V^2$ also depending on  those quantities  and on $g_L$. 
By similar arguments, we get
$$I_3 \leq K_V^3\int_{t_1}^{t}  
(\meanV_s-\hV^{N}_s)^2+(\meanY_s-\hy^{N}_s)^2 ds$$
for some $K_V^3$ depending on $\JCh $ and on $\sup_{v\in[-\frac{ 5\Rmax}{2 g_L} ,\frac{ 5\Rmax}{2 g_L} ]}|v-V_\text{rev} |$. We thus get:
$$\left( \meanV_t - \hV_t\right)^2 \leq  K_V^1 \int_{t_1}^{t}{\left[ \varV_s+\varM_s+\varN_s+\varH_s\right]ds} + \tilde{K}_V\int_{t_1}^{t}{|\bar{X}^N_s-\hX^N_s|^2ds} $$
for some explicit $\tilde{K}_V$ a.s., from where
\begin{equation}\label{eq:Bound-meanV-hatV}
\E\left[ \sup_{t_1\leq s\leq t}\left( \meanV_s - \hV_s\right)^2\right] \leq  K_V^1 \int_{t_1}^{t}{\E\left[ \varV_s+\varM_s+\varN_s+\varH_s\right]ds} + \tilde{K}_V\int_{t_1}^{t}{\E\left[ \sup_{t_1\leq u\leq s}|\bar{X}^N_u-\hX^N_u|^2\right]ds}.
\end{equation}
On the other hand, for $x$ type channels we get
\begin{align*}
\meanX_t -\hx^{N}_t&=\frac{1}{N}\sum_{j=1}^N\int_{t_1}^{t}{ \rho_x(V^{(j)}_s)(1-x^{(j)}_s)  -\zeta_x(V^{(j)}_s)x^{(j)}_s -\rho_x(\meanV_s)(1-\meanX_s)  +\zeta_x(\meanV_s)\meanX_s}ds\\
&\quad +\int_{t_1}^{t}\rho_x(\meanV_s)(1-\meanX_s)  -\zeta_x(\meanV_s)\meanX_s- \rho_x(\hV^{N}_s)(1-\hx^{N}_s)  +\zeta_x(\hV^{N}_s)\hx^{N}_sds\\
&\quad+\frac{1}{N}\sum_{j=1}^N\int_{t_1}^{t} \sigma_x(V_s^{(j)},x_s^{(j)}) dW_s^{x,j}.
\end{align*}
For $t\in (t_1,t_1+\delta)$ we  deduce:
\begin{align*}
(\meanX_t -\hx^{N}_t)^2&\leq 3\delta  \int_{t_1}^{t} \frac{1}{N}\sum_{j=1}^N\left( \rho_x(V^{(j)}_s)(1-x^{(j)}_s)  -\zeta_x(V^{(j)}_s)x^{(j)}_s-\rho_x(\meanV_s)(1-\meanX_s)  +\zeta_x(\meanV_s)\meanX_s\right)^2 \, ds \\
&\quad + 3\delta \int_{0}^{t} \left(\rho_x(\meanV_s)(1-\meanX_s)  -\zeta_x(\meanV_s)\meanX_s- \rho_x(\hV^{N}_s)(1-\hx^{N}_s)  +\zeta_x(\hV^{N}_s)\hx^{N}_s\right)^2  \, ds\\
&\quad+    3\left( \frac{1}{N}\sum_{i=1}^N\int_{t_1}^{t}{\sigma_x(V_s^{(j)},x_s^{(j)})dW_s^{x,j}}\right)^2. \\
\end{align*}
The previous   yields, 
\begin{align*}
\E\bigg[ \sup_{t_1\leq s\leq t} & (\meanX_s -\hx^{N}_s)^2\bigg]\\
&\leq 3\delta \int_{t_1}^{t}{ \E\left[ \frac{1}{N}\sum_{j=1}^N\left( \rho_x(V^{(j)}_s)(1-x^{(j)}_s)  -\zeta_x(V^{(j)}_s)x^{(j)}_s-\rho_x(\meanV_s)(1-\meanX_s)  +\zeta_x(\meanV_s)\meanX_s\right)^2\right]ds}\\
&\quad +3\delta \int_{0}^{t}\E\left[ \left(\rho_x(\meanV_s)(1-\meanX_s)  -\zeta_x(\meanV_s)\meanX_s- \rho_x(\hV^{N}_s)(1-\hx^{N}_s)  +\zeta_x(\hV^{N}_s)\hx^{N}_s\right)^2\right]ds\\
&\quad+3\E\left[\sup_{t_1\leq s\leq t} \left( \frac{1}{N}\sum_{i=1}^N\int_{t_1}^{s}{\sigma_x(V_u^{(j)},x_u^{(j)}) dW_u^{x,j}}\right)^2\right]\\
&= I_1+I_2+I_3.
\end{align*}
Denoting by  $L_{f,R}$   a Lipschitz constant of a function  $f$ on $[- R,R]$ and using 
standard arguments, we get that
$$  I_1 \leq K_x\delta (L_{\rho_x,  \frac{ 5\Rmax}{2 g_L} }^2+L_{\rho_x+\zeta_x, \frac{ 5\Rmax}{2 g_L} }^2 )\int_{t_1}^{t}{ \E\left[\varV_s + \varX_s   \right]ds} $$
and that
$$ I_2\leq K_x	\delta (L_{\rho_x, \frac{ 5\Rmax}{2 g_L}  }^2+L_{\rho_x+\zeta_x, \frac{ 5\Rmax}{2 g_L} }^2 )\int_{t_1}^{t}{ \E\left[ (\meanV_s-\hV^{N}_s)^2+(\meanX_s-\hx^{N}_s)^2   \right]ds}
$$
for all  $t\in (t_1,t_1+\delta)$. By Doob's inequality, we  moreover obtain
\begin{equation*}
I_3 \leq 3 \cdot 4 \E\left[ \frac{1}{N^2}\sum_{i=1}^N\int_{t_1}^{t}{ \sigma^2 _x(V_s^{(j)},x_s^{(j)}) ds}\right] \leq \frac{12  \sigma^2 \delta}{N}   \|\rho_x \vee \zeta_x \|_{\infty,  \frac{ 5\Rmax}{2 g_L}} . \end{equation*}
 Summarizing, for the $x$-type channel we have shown that for all   $t\in (t_1,t_1+\delta)$,
\begin{equation}\label{eq:bound-for-x-mean-hat}\E\left[ \sup_{t_1\leq s\leq t}(\meanX_s -\hx^{N}_s)^2\right]  \leq \delta K_x\int_{t_1}^{t}{\E\left[ \varV_s+\varX_s+\sup_{t_1\leq u\leq s}|\bar{X}^N_u-\hX^N_u|^2 \right]  ds} +  \frac{ 12\sigma^2 \delta}{N} \|\rho_x \vee \zeta_x \|_{\infty, \frac{ 5\Rmax}{2 g_L}}
\end{equation}
for some constants $K_x>0$. Putting together \eqref{eq:Bound-meanV-hatV} and \eqref{eq:bound-for-x-mean-hat} we get for  all   $t\in (t_1,t_1+\delta)$ and some constants $K_1,K_2>0$, 
\begin{multline*}
\E\left[ \sup_{t_1\leq s\leq t}|\bar{X}^N_s-\hX^N_s|^2\right]   \leq (1+\delta) K_1\int_{t_1}^{t_1+\delta}{\E\left(  \varV_s+ \sum_{x=m,n,h,y}  \varX _s\right)  ds}   + \frac{12  \sigma^2 \delta}{N} C^0_{\zeta,\rho}  \\ + 
(1+\delta) K_2\int_{t_1}^{t}{\E\left[\sup_{t_1\leq u\leq s}|\bar{X}^N_u-\hX^N_u|^2\right]  ds},  \quad  \quad  \quad   \quad  \quad 
\end{multline*}
from where,  using Gronwall's inequality, we deduce:  
\begin{equation*}
\E\left( \sup_{t_1\leq t \leq t_1+\delta}|\bar{X}^N_t-\hX^N_t|^2\right) \leq e^{K_2(1+\delta)} \left( K_1(1+\delta) 
   \int_{t_1}^{t_1+\delta} \E\left( \varV_s + \sum_{x=m,n,h,y} \varX_s \right)ds +  \frac{12  \sigma^2\delta }{N} C^0_{\zeta,\rho}  \right). 
\end{equation*}
We can now use  Theorem  \ref{theo:synchro-empiricalVar}  to bound the integral  on the r.h.s.  With $K_{1,\delta} =e^{K_2(1+\delta)}K_1(1+\delta) $ and  $K_{1,\delta} =12 e^{K_2(1+\delta)}$ we get, 
 for all $t_1\geq t_0$, that
\begin{equation*}
\begin{aligned}
\E\left( \sup_{t_1\leq t \leq t_1+\delta}|\bar{X}^N_t-\hX^N_t|^2\right) &\leq \E\left( \varV_{t_0}+ \sum_{x=m,n,h,y}  \varX_{t_0}  \right)\frac{1}{\lambda^0 }(1-e^{-\lambda^0 \delta}) e^{-\lambda^0( t_1-t_0) }K_{1,\delta}\\
&\quad+ \frac{\sigma^2 C^0_{\zeta,\rho}}{\lambda^0 } \delta K_{1,\delta} + \delta  K_{2,\delta}  \frac{\sigma^2  }{N} C^0_{\zeta,\rho}  \\
&\leq \left(  \left[(V^*_{t_0} )^2 + 4\right] e^{-\lambda^0( t_1-t_0) } + \frac{\sigma^2 C^0_{\zeta,\rho}}{\lambda^0 }\right) \delta K_{1,\delta} + \delta  K_{2,\delta}  \frac{\sigma^2  }{N} C^0_{\zeta,\rho}  \\
\end{aligned}
\end{equation*}
since $ \varV_{t_0}\leq (V^*_{t_0} )^2$. 
 \end{proof}

\begin{proof}[Proof of  Theorem  \ref{theo:synchro-empiricalVar}.\ b)]  
Notice  on hand that,  for each $t\geq t_1$, we always have the bounds $$ | X^{(i)}_t- \hX^{t_1,N}_t|^2\leq 2 \varV_{t}+ 2 |\widehat{V}_t- \bar{V}_t^N|^2+ 4\leq 4(V^*_{t_0})^2 + 4 \leq K_0:= 4\left(  \frac{ 5\Rmax}{g_L} \right)^2+4, $$
 thanks to  \eqref{boundhatV} and that $V^*_{t_0} \leq  \frac{ 5\Rmax}{g_L} $.  On the other hand, combining Proposition \ref{prop:HatProcessIsCloseToMeanProcess}  with Theorem  \ref{theo:synchro-empiricalVar}.\ a) we get for every $t\in [t_1,t_1+\delta]$ that
$$  \E\left(   | X^{(i)}_t- \hX^{t_1,N}_t|^2\right)\leq   2 \bigg[     \left( K_0' e^{-\lambda_0(t-t_0)} + \sigma^2 \frac{C^0_{\zeta,\rho}}{\lambda^0}\right) (1+ \delta K_{1,\delta})+  \delta  K_{2,\delta}  \frac{\sigma^2  }{N} C^0_{\zeta,\rho} \bigg] $$
with $K_0'=\left(  \frac{ 5\Rmax}{g_L} \right)^2+4$. The statement follows.
\end{proof}

\section{Propagation of Chaos and synchronization for the McKean-Vlasov limit: proofs of Theorem \ref{theo:propchaos_McKeanVlasov_eq} and   Corollary  \ref{coro:synchro-McKeanVlasov}}\label{sec:chaos}
We first address the asymptotic behavior of the flow of empirical measures \eqref{eq:empirical_measure} when $N\to \infty$ and the proof of Theorem \ref{theo:propchaos_McKeanVlasov_eq}. In particular, we will prove the  propagation of chaos property for  system \eqref{eq:HH-model}. Following the classic pathwise approach developed 
in  \cite{sznitman1991topics} and \cite{meleard1996asymptotic},  we first establish: 

\begin{theo}\label{theo:prop_chais_pathwise} 
Under the assumptions of Theorem \ref{theo:propchaos_McKeanVlasov_eq}, we have:
\begin{itemize}
\item[a)]  Let $W^{x}, x=m,n,h,y$ be independent  standard Brownian motions and $(V_0,m_0,n_0,h_0,y_0)$  an independent random vector with law $\mu_0$. 
There is existence and uniqueness, pathwise and in law, of a solution $\tX= (\tV_t,\tm_t,\tn_t,\th_t,\ty_t, t\geq 0 )$ to the nonlinear stochastic differential  equation  (in the sense of McKean)  with values in $\R\times [0,1]^4$:
\begin{equation}\label{eq:NonLinearProcess} 
\begin{aligned}
\tV_t &= V_0+ \int_{0}^{t}{F(\tV^{}_s,\tm_s,\tn_s,\th_s)ds}-\int_{0}^{t}{ \JE (\tV^{}_s-\E[\tV_s])ds}- \int_{0}^{t}{\JCh \E[\ty_s](\tV^{}_s-\Vrev) ds},\\
 \tx^{}_t& = x^{}_0+\int_{0}^{t}\rho_x(\tV^{}_s)(1-\tx^{}_s)  -\zeta_x(\tV^{}_s)\tx^{}_sds + \int_{0}^{t}{\sigma_x(\tV_s^{},\tx_s^{})dW_s^{x}},\;\;x=m,n,h,y \,
 \end{aligned}
\end{equation}
such that for all $t\geq0$, $|\tV_t|\leq{4\Rmax}/{\gL} + 2\Vmax_0 e^{-\gL t}$ almost surely. 
\item[b)]  $(\mu_t:=\mbox{law}(\tilde{X}_t): t\geq 0)$ is a weak solution globally defined  in
 $C((0,+\infty];   {\cal P}_2(\R\times [0,1]^4)) $    of the McKean-Vlasov equation \eqref{eq:Non-Linear-PDE}. 
 
 \item[c)]  For each $T>0$, let  $\tX^{(i)}= \left((\tV^{(i)}_t,\tm^{(i)}_t,\tn^{(i)}_t,\th^{(i)}_t,\ty^{(i)}_t):t\in [0,T]\right)$,  $i=1,\ldots,N$ be independent copies of the nonlinear process \eqref{eq:NonLinearProcess} each of them driven by the same Brownian motions $(W^{x,i},\ x=m,n,h,y)$  and with same initial conditions $X^{(i)}_0=\tX^{(i)}_0 $ as the $N$-particle system \eqref{eq:HH-model}.  Then, there
is a constant $C(T)>0$ such that for every $N\geq 1$ and $i\in \{1,\dots, N\}$, 
$$\E\left[ \sup_{0\leq t\leq T}| X^{(i)}_t-\tX^{(i)}_t |^2 \right]\leq \frac{C(T)}{N}.$$
\end{itemize}
\end{theo}

\begin{proof} 
The statements a), b) and c) would be   standard if the coefficients 
in each of the $N$ components of  \eqref{eq:HH-model}  were replaced by  globally Lipschitz  functions  of $X^{(i)}_s$ and $X^{(j)}_s$, see Theorems 2.2 and 2.3 in  \cite{meleard1996asymptotic}.
In particular, with functions $p^j_M$ and $F_M$ defined for fixed $M>0$ as in Lemma \ref{lem:well-posedHH}, for any $T>0$ there is existence and uniqueness, pathwise and in law, of a solution to the nonlinear  stochastic differential equation  on $[0,T]$:
\begin{equation}\label{eq:NonLinearProcess_M} 
\begin{aligned}
\tV_t^M &= V_0+ \int_{0}^{t}{F_M(\tV^M_s,\tm_s^M,\tn_s^M,\th_s^M)ds}-\int_{0}^{t}{ \JE (\tV^M_s-\E[\tV^M_s])ds}\\
 & \quad - \int_{0}^{t}{\JCh \E[p_M^1(\ty_s)](p_M^1(\tV^M_s)-\Vrev) ds},\\
 \tx^M_t & = x^{}_0+\int_{0}^{t}\rho_x(p^1_M(\tV^M_s)(1-p_M^1(\tx^M_s))  -\zeta_x(p_M^2(\tV^M_s))p_M^1(\tx^M_s) ds\\
 &\quad  
  + \int_{0}^{t}{\sigma_x(p_M^1(\tV_s^M),\tx_s^M)dW_s^{x}},\;\;x=m,n,h,y.  \,
 \end{aligned}
\end{equation}
Moreover, letting 
 $\tX^{(i,M)}= \left((\tV^{(i,M)}_t,\tm^{(i,M)}_t,\tn^{(i,M)}_t,\th^{(i,M)}_t,\ty^{(i,M)}_t):t\in [0,T]\right)$,  $i=1,\ldots,N$ be independent copies of the nonlinear process \eqref{eq:NonLinearProcess_M} driven by the same Brownian motions $(W^{x,i},\ x=m,n,h,y)$  and with same initial conditions $X^{(i)}_0=\tX^{(i)}_0 $ as the  system  $(X^{(1,M)},\dots, X^{(N,M)}) $ defined in \eqref{eq:truncated-HH-model}, we obtain that 
 $$\E\left[ \sup_{0\leq t\leq T}| X^{(i,M)}_t-\tX^{(i,M)}_t |^2 \right]\leq \frac{C_M(T)}{N}$$
 for every $N\geq 1$ and $i\in \{1,\dots, N\}$, and some constant $C_M(T)>0$.

 We notice now that, by  Proposition \ref{prop:boundedness-of-V^i}, for $M>0$ large enough the system $(X^{(1)},\dots, X^{(N)}) $  is also a solution to the system of equations \eqref{eq:truncated-HH-model}.  Pathwise uniqueness of the latter yields for all such $M>0$ that  $(X^{(1)},\dots, X^{(N)}) =(X^{(1,M)},\dots, X^{(N,M)}) $  on $[0,T]$, from where
 \begin{equation}\label{estim_prop_chaos_M}
 \E\left[ \sup_{0\leq t\leq T}| X^{(i)}_t-\tX^{(i,M)}_t |^2 \right]\leq \frac{C_M(T)}{N}
 \end{equation}
 for every $N\geq 1$ and $i\in \{1,\dots, N\}$. Furthermore, for any $M'>0$ 
 $$\P\left(  \sup_{0\leq t\leq T}  \tx^{(i,M)}_t  \geq M'+\varepsilon \right)\leq \P\left(  \sup_{0\leq t\leq T}  x^{(i,M)}_t  \geq M' \right) + \frac{2 C_M(T)}{N \varepsilon^2}$$
 Taking $M'=1$, letting $N\to \infty$ and then $\varepsilon\to 0$ we deduce that $ \tx^{(i,M)}_t \leq 1$ a.s. for every $t\in [0.T]$ and $i\in \N $. In a similar way,   $ \tx^{(i,M)}_t \geq  0$ and $|\tilde{V}^{(i,M)}_t | \leq   \Vmax_{t,\infty}$ hold a.s. for every $t\in [0, T]$ and $i\in \N $. This implies that for $M>0$ large enough but fixed,  a solution to \eqref{eq:NonLinearProcess_M}  also solves \eqref{eq:NonLinearProcess}, and proves the existence part in a).
 
 \smallskip 

We show now that any solution have uniform in time bounded compact support, from which uniqueness in part a) will immediately  follow.  We shall first consider a solution $(U_t,q^m_t,q^n_t,q^h_t,q^y_t)$ of \eqref{eq:NonLinearProcess}  with explosion time $\xi$, and we will show that it  coincides with  $(\tV^{M}_t,\tm^{M}_t,\tn^{M}_t,\th^{M}_t,\ty^{M}_t, t\geq 0)$ for a $M$ big enough.  For $M>1$, we define  $\tau_M = \inf\{t\geq0: \max\{|U_t|,|q^m_t|,|q^n_t|,|q^h_t|,|q^y_t|\}\geq M\}$.
Then  we observe that  the coefficients of \eqref{eq:NonLinearProcess} applied to $(U_t,q^m_t,q^n_t,q^h_t,q^y_t,  0 \leq t \leq  \tau_M)$ coincide with the truncated coefficients of \eqref{eq:NonLinearProcess_M} 
and thanks to the uniqueness property for \eqref{eq:NonLinearProcess_M}  we  conclude that almost surely
$$(U,q^m,q^n,q^h,q^y)_{t\wedge\tau_M} = (\tV^{M},\tm^{M},\tn^{M},\th^{M},\ty^{M})_{t\wedge\tau_M}.$$
In particular, we observe that  $q^x_{t \wedge \tau_M} \in [0,1]$ for $x=m,n,h,y$,  and  that  $\tau_M = \inf\{t\geq0:|U_t|\geq M\}$ for $M > 1$.  Moreover the second order moment $ \E[U_{t\wedge\tau_M}^2]$ is uniformly bounded in $M$, since 
 \begin{align*}
U_{t\wedge\tau_M}^2 &= V_0^2+ 2\int_{0}^{t\wedge \tau_M}{U_sF(U_s,q^m_s,q^n_s,q^h_s)ds}-2\int_{0}^{t\wedge \tau_M}{ \JE U_s(U_s-\E[U_s])ds}\\
&\quad- 2\int_{0}^{t\wedge \tau_M}{\JCh \E[q^y_s]U_s(U_s-\Vrev) ds},
\end{align*}
 from where, it  is easy to show that 
 $$\E(U_{t\wedge\tau_M}^2) \leq C_1 + C_2\int_{0}^{t}{\E(U_{s\wedge\tau_M}^2)},$$
 and therefore, thanks to Gronwall's inequality
  $$\E(U_{t\wedge\tau_M}^2) \leq C_1e^{C_2 t}.$$
  On the other hand $\E(U_{t\wedge\tau_M}^2) = \E(U_{t}^2\ind{\tau_M>t})+ M^2\P(\tau_M \leq t)$ and then we can conclude for all $t\geq0$ and all $M\geq 1$
  $$\P(\tau_M \leq t) \leq \frac{C_1e^{C_2t}}{M^2}.$$
Since $\tau_M\nearrow \xi$, we conclude that for all $t$ $\P(\xi \leq t)=0$, from where it follows that $\xi$ is almost surely infinite. 
 
Now, since $(U_t,q^m_t,q^n_t,q^h_t,q^y_t)$ has no explosion, we   apply  Proposition 3.3 in \cite{Bossy2015} to get that  almost surely $q^x_{t} \in [0,1]$ for any $t>0$. Using this, we derive a more precise bound for the second order moment: 
\begin{align*}
\E(U_{t}^2) &\leq \E(V_0^2)+ 2\int_{0}^{t}{\sqrt{\E(R_s^2)}\sqrt{\E(U_s^2)}-\gL\E(U_s^2) ds},
\end{align*}
where as in the proof of Proposition \ref{prop:boundedness-of-V^i}, 
\begin{align*}
 R_s
  & \leq   \Rmax:= \max_{a,b,c  \in [0,1]}{| I + \gL \VL+\gK\VK a+\gNa\VNa b+\JCh\Vrev c}|. 
\end{align*}
 Applying one more time Lemma \ref{lem:gronwallAmbrisio} we conclude
$$\sqrt{ \E(U_{t}^2)} \leq \sqrt{\E(V_0^2)} e^{-\gL t} + \frac{2\Rmax}{\gL}(1-e^{-\gL t}).$$
Thus, the second moment of any solution of \eqref{eq:NonLinearProcess} is uniformly bounded in time.  Moreover, since the initial condition $V_0$ is bounded, proceeding exactly as in the proof of Proposition \ref{prop:boundedness-of-V^i} we obtain  that
$$|U_t| \leq \frac{4\Rmax}{\gL} + 2\Vmax_0 e^{-\gL t},$$
with the same bound  $\Vmax_0$ for $V_0$. In conclusion, solutions of \eqref{eq:NonLinearProcess} are non explosive, even more  they are uniformly bounded in time. Choosing   $M>4\Rmax/\gL + 2\Vmax_0$, we get $\tau_M=\infty$ almost surely,  and 
for any $t\geq0$,
$$(U,q^m,q^n,q^h,q^y)_{t} = (\tV^{M},\tm^{M},\tn^{M},\th^{M},\ty^{M})_{t}.$$
Hence equation \eqref{eq:NonLinearProcess} has a unique solution.

Part  b) derives from a direct and easy application of the Ito's formula to compute 
$$\E[\psi(\tX_t)] = \int_{\R\times [0,1]^4} \psi(x) \mu_t(dx) $$ 
for a $C^\infty_c$ test function $\psi$,  thanks to the fact that the Lebesgue integrals on the  right hand side of the Itô formula will be all bounded, since the supports of the laws $(\mu_t: t\geq 0)$ are contained in some compact set, and by  continuity of coefficients. 

Part c) is immediate taking large enough $M$ in  \eqref{estim_prop_chaos_M}. 

 \end{proof}

We are now in position to prove 
\begin{proof}[Proof of  of Theorem  \ref{theo:propchaos_McKeanVlasov_eq}]
a) We write $\mathcal{C}_T:=C([0,T], \R\times [0,1]^4)$. 
Part c) of Theorem \ref{theo:prop_chais_pathwise}  implies that for each  $T>0$ and $k\geq 1$ the convergence $\mbox{Law}(X^{(1)},\dots X^{(k)})\to\mu^{\otimes k}$  with  $\mu=\mbox{Law}(\tilde{X}^{(1)})$  holds on the space  $\mathcal{C}_T^k$ as $N\to \infty$. By  Proposition 2.2.\ in \cite{sznitman1991topics} or Proposition 4.2.\ in \cite{meleard1996asymptotic}, this implies that the empirical measure 
\begin{equation*} 
\mu^N:=\frac{1}{N}\sum_{i=1}^N \delta_{X_{\cdot}^{(i)}}\in  {\cal P}(\mathcal{C}_T), 
\end{equation*}
with ${\cal P}(\mathcal{C}_T)$ denoting the space of probability measures on $\mathcal{C}_T$ endowed with the weak topology, converges in law to the (deterministic) probability measure $\mu$.  The  first assertion of the theorem follows then from the fact that the  mapping  associating with $\nu \in   {\cal P}(\mathcal{C}_T)$ its flow $(\nu_t:t\in [0,T])\in C([0,T];{\cal P}( \R\times [0,1]^4))$ of one-dimensional time-marginals laws is continuous,  together with  part b) of Theorem \ref{theo:prop_chais_pathwise} (notice  that $C([0,T];{\cal P}( \R\times [0,1]^4))$ can be replaced by $C([0,T]; {\cal P}_2( \R\times [0,1]^4))$ since all the random measures involved have a common compact support).

\medskip 

b) We observe first that for each $t\geq 0$ one has
$$      \E\left(\mathcal{W}_2^2(\mu^N_t,\mu_t )  \right) \leq 2    \E\left(\mathcal{W}_2^2(\mu^N_t,\tilde{\mu}^N_t )  \right)+  2  \E\left(\mathcal{W}_2^2(\tilde{\mu}^N_t,\mu_t ) \right),$$
where $\tilde{\mu}^N_t$ is the empirical measure  of any  random i.i.d. sample  of the law $\mu_t$ constructed in the same probability as $\mu^N_t$. Taking 
$\tilde{\mu}^N_t:=\frac{1}{N}\sum_{i=1}^N \delta_{\tilde{X}_t^{(i)}}$
with $\tX^{(i)}_t$,  $i=1,\ldots,N$ the processes defined in part c) of Theorem \ref{theo:prop_chais_pathwise}  we get for every $t\in [0,T]$  that
$$      \E\left(\mathcal{W}_2^2(\mu^N_t,\mu_t )  \right) \leq 2  \frac{C(T)}{N} +  2  \E\left(\mathcal{W}_2^2(\tilde{\mu}^N_t,\mu_t )\right) .$$
On the other hand, we have    $\sup_{t\in [0,T]}( \int |z|^q\mu_t(dz))^{1/q} <\infty$  for each $q\geq 1$, using for instance   the bound  obtained at the end of the proof of Theorem \ref{theo:prop_chais_pathwise}.  We can therefore apply Theorem 1 in \cite{fournier-guillin2013} with $p=2$, $d=5$ and a sufficiently large $q>2$,  to get that $ \E\left(\mathcal{W}_2^2(\tilde{\mu}^N_t,\mu_t )\right) \leq C N^{-2/5}$. The second assertion thus follows.

\medskip 
c)  In order to prove uniqueness  for the McKean-Vlasov equation \eqref{eq:Non-Linear-PDE}, we adapt to our setting a generic argument going back at least to G\"artner \cite{gartner}.   Assume for  a while that for each compactly supported  $\nu_0\in {\cal P}( \R\times [0,1]^4))$  and  $(\nu^*_t:t\in [0,T])\in C([0,T],{\cal P}_2( \R\times [0,1]^4))$ 
the linear Fokker-Planck equation 
\begin{equation}\label{eq:Linear-PDE}
\begin{aligned}
\partial_t\nu_t & = \partial_v\left(\Phi (\langle (\nu^*_t)^V\rangle ,\langle (\nu^*_t)^y\rangle ,  \cdot,\cdot)\nu_t \right)+ \sum_{x=m,n,h,y} \frac{1}{2}\sigma^2 \, \partial_{u_x u_x}^2 \left(a_x\nu_t\right)-\partial_{u_x}\left( b_x\nu_t\right)
\end{aligned}
\end{equation}
has at most one weak solution with supports  bounded uniformly in $t\in [0,T]$. By similar arguments as in Lemma \ref{lem:well-posedHH}, strong well-posedness holds for the stochastic differential equation: 
\begin{equation}\label{eq:AuxiliarProcess} 
\begin{aligned}
V_t^* &= V^*_0+ \int_{0}^{t}{F(V^*_s,m_s^*,n_s^*,h_s^*)ds}-\int_{0}^{t}{ \JE (V^*_s-\langle (\nu^*_s)^V \rangle)ds}- \int_{0}^{t}{\JCh \langle (\nu^*_s)^y \rangle (V^*_s-\Vrev) ds},\\
 x^*_t& = x^*_0+\int_{0}^{t}\rho_x(V^*_s)(1-x^*_s)  -\zeta_x(V^*_s)x^*_sds + \int_{0}^{t}{\sigma_x(V_s^*,x_s^*)dW_s^{x}},\;\;x=m,n,h,y,
\end{aligned}
\end{equation}
with $(V_0^* ,m_0^* ,n_0^* ,h_0^* ,y_0^* )$ independent of the Brownian motions $W^x$ and with law  $\nu_ 0$. Moreover,   one can check that  $x^*_t\in [0,1]$  a.s. for all $t\in [0,T]$ and   that the process $(V_t^*:t\in [0,T])$ is bounded.  It follows using It\^o's formula that a unique weak solution to equation \eqref{eq:Linear-PDE}  with uniformly  bounded supports   does exist, and  is given by $\nu_ t=\mbox{ law} (V_t^* ,m_t^* ,n_t^* ,h_t^* ,y_t^* )$ for all  $t\in [0,T]$. Now, any solution $(\mu_t:t\in [0,T])$   in $C([0,T],{\cal P}( \R\times [0,1]^4))$ of \eqref{eq:Non-Linear-PDE} with uniformly bounded supports also solves the linear equation \eqref{eq:Linear-PDE} with  $(\nu^*_t:t\in [0,T])=(\mu_t:t\in [0,T])$. This yields,    for all  $t\in [0,T]$,  that $\mu_t=\mbox{ law} (V_t^* ,m_t^* ,n_t^* ,h_t^* ,y_t^* )$, for the process  defined as in \eqref{eq:AuxiliarProcess},  with $\nu^*_s=\mu_s$ for all $ s\in [0,T]$. In other words,   this process solves the nonlinear stochastic differential equation \eqref{eq:NonLinearProcess}. From Theorem \ref{theo:prop_chais_pathwise}  we conclude that 
$(\mu_t:t\in [0,T])=(\mbox{law}(\tilde{X}_t):t\in [0,T])$, that is,  there is uniqueness of solutions in $C([0,T],{\cal P}( \R\times [0,1]^4))$ of  \eqref{eq:Non-Linear-PDE} having uniformly bounded support.

Hence, in order to conclude the proof of Theorem  \ref{theo:propchaos_McKeanVlasov_eq} it is enough to show that, given  functions $\alpha,\beta\in C([0,T],\R)$ and $\nu_0\in {\cal P}_2( \R\times [0,1]^4)$ there is at most one solution   $(\nu_t:t\in [0,T])\in C([0,T],{\cal P}( \R\times [0,1]^4))$ with support bounded uniformly in $[0,T]$,  to the distribution formulation of equation \eqref{eq:Linear-PDE}
 \begin{equation}\label{eq:weakeq}
\begin{split}
\int \psi(t,v,u) \nu_t(dv,du) & = \int \psi(0,v,u) \nu_0(dv,du) - \int_0^t  \int \bigg[ \Phi (\alpha_s ,\beta_s , v, u ) \partial_v \psi(s,v,u)   \\
&  \qquad   +  \big( \partial_s + \sum_{x=m,n,h,y}  \frac{1}{2} \sigma^2  \, a_x  \partial_{u_x u_x}^2  + b_x \partial_{u_x} \big) \psi(s,v,u)  \bigg] \nu_s (dv,du)  \, ds
\end{split}
\end{equation}
for all $t\in [0,T]$ and for an extended class of test function $\psi\in C^{1,1,2}_b([0,T]\times \R\times [0,1]^4)$. Let $\rho_x'$ and $\zeta_x'$ denote compactly supported functions coinciding with   $\rho_x$ and $\zeta_x$ on some compact set  ${\cal K}\subset \R$ containing the supports of the measures $\nu^V_t$ for $ t\in [0,T]$, and define $\sigma_x'$, $a'_x$ and $b_x'$ in terms of them in a similar way as $\sigma_x$, $a_x$ and $b_x$ were defined in terms of   $\rho_x$ and $\zeta_x$. 
For a given $t>0$, consider the following Cauchy problem in $\R^5$ : for all $(s,v,u)\in [0,t)\times \R\times \R^4$, 
\begin{equation}\label{cauchy}
\begin{aligned}
  \big( \partial_s-  \Phi (\alpha_s ,\beta_s , v, u ) \partial_v +  \sum_{x=m,n,h,y} \frac{1}{2} \sigma^2  \, a'_x  \partial_{u_x u_x}^2  + b'_x \partial_{u_x} \big) f_t(s,v,u)=&0,\\
f_t(t,v,u) = & \, \psi(v,u).
\end{aligned}
\end{equation}
By the Feynman-Kac formula  (see e.g. Karatzas and Shreve \cite{Karatzas:1991uq}),  if a solution $f_t\in C_b([0,t] \times \R^5 )\cap C_b^{1,1,2}([0,t)\times \R\times \R^4)$  exists, then  it is given  by 
\begin{equation}\label{eq:FKformula}
f_t(s,v,u):=\E(\psi(X^{s,v,u}_t)) \, 
\end{equation}
 where $(X_r^{s,v,u}:=(V_r,m_r,n_r,h_r,y_r)\,: r\in [s,t])$ is the unique (pathwise and in law) solution in $[s,t]$ of the stochastic differential equation: 
\begin{equation*} 
\begin{aligned}
V_r &= v+ \int_{s}^{r}{F(V_{\theta},m_{\theta},n_{\theta},h_{\theta})d\theta}-\int_{s}^{r}{ \JE (V_{\theta}-\alpha_{\theta}) ds}- \int_{s}^{t}{\JCh\beta_{\theta} (V_{\theta}-\Vrev) d\theta},\\
 x_r& = u_x+\int_{s}^{r}\rho'_x(V_{\theta})(1-x_{\theta})  -\zeta'_x(V_{\theta})x_{\theta}d\theta + \int_{s}^{r}{\sigma'_x(V_{\theta},x_{\theta})dW_{\theta}^{x}},\;\;x=m,n,h,y.
 \end{aligned}
\end{equation*}
Moreover, for $v$ chosen in some fixed compact set, this solution  is bounded independently of $s\in [0,t]$, and one has $x_r\in [0,1]$ for all $r\in [s,t]$.   Hence, under the assumption that $\sigma>0$, $\rho_x$ and $\zeta_x$ are of class $C^2(\R)$,  one can moreover prove, following the lines of  Friedman \cite[p.124]{Friedman-06}, that the function $f_t$ defined by  \eqref{eq:FKformula} actually is of class $C_b^{1,1,2}([0,t)\times \R\times\R^4)$ and solves the Cauchy problem \eqref{cauchy}. Putting $\psi=f_t$ in \eqref{eq:weakeq} yields
\begin{equation*}
\int \psi(v,u) \nu_t(dv,du)  = \int \E(\psi(X^{s,v,u}_t))  \nu_0(dv,du) 
\end{equation*}
for all $\psi\in C_0^2(\R^5 )$, which uniquely determines $\nu_t.$ Notice that when $\sigma=0$, the required  
regularity for $\phi$ and for $f$ turns from $C^{1,1,2}$  to  $C^{1,1,1}$ and the  Feymann Kac formula   in the argument can be replaced by the characteristics formula.  The proof of part c)  is complete. 

d) This is immediate from parts b) and d) of Theorem \ref{theo:prop_chais_pathwise} . 
  \end{proof}

\begin{proof}[Proof of  Corollary  \ref{coro:synchro-McKeanVlasov}]

Recall  first that, for any $\nu\in {\cal P}_2(\R\times [0,1]^4)$ and $w\in \R\times [0,1]^4)$, one has 
$$ \mathcal{W}_2^2(\nu,\delta_w) = \int | z-w|^2 \nu(dz).$$
Moreover, for every $t\geq t_1$ and $N\geq 1$  it holds by exchangeability that:  
$$  \E\left(   | X^{(i)}_t- \hX^{t_1,N}_t|^2\right)=  \E\left(  \mathcal{W}_2^2(\mu^N_t,\delta_{ \hX^{ t_1,N}_t  }) \right)  . $$
 Therefore, it is enough to prove that,  for any $t_1\geq 0$,
\begin{equation*}
\sup_{t_1\leq t\leq t_1+\delta}  \E\left|  \mathcal{W}_2^2(\mu_t,\delta_{ \hX^{ t_1,\infty}_t  })   -   \mathcal{W}_2^2(\mu^N_t,\delta_{ \hX^{ t_1,N}_t  })    \right|  \to 0
\end{equation*}
as $N\to \infty$.  Given  $t\geq t_1$ and $N\geq 1$, let $\pi_t^N(dz,dz')$ be  a coupling between $\mu_t$ and $\mu^N_t$.
Then,  for some constant $C>0$ not depending on $t\geq t_1$ nor on $N\geq  1$, we have
\begin{equation*}
\begin{split}
 \left| \mathcal{W}_2^2(\mu_t,\delta_{ \hX^{ t_1,\infty}_t  })   -   \mathcal{W}_2^2(\mu^N_t,\delta_{ \hX^{ t_1,N}_t  }) \right| =  & \left| \int \pi_t^N(dz,dz')  \left[ | z-   \hX^{ t_1,\infty}_t |^2 - | z'-\hX^{ t_1,N}_t  |^2\right]\right| \\ 
 \leq  &\,  C  \left[ \int   | z-  z'  |\pi_t^N(dz,dz') + |  \hX^{ t_1,\infty}_t-\hX^{ t_1,N}_t  |\right] \\ 
\end{split}
\end{equation*}
since the supports of $\mu_t$ and $\mu_t^N$ and the processes  $\hX^{ t_1,\infty}_t$  and  $\hX^{ t_1,N}_t$  are uniformly bounded in $t\geq t_1$ and $N$.  The latter property also allows us to write the dynamics in \eqref{eq:HatProcess}  and \eqref{eq:HatProcess_infty} using globally Lipschitz coefficients. Thanks  to Gronwall's lemma this yields the estimates 
$$  \sup_{t_1\leq t\leq t_1+\delta}  |  \hX^{ t_1,\infty}_t-\hX^{ t_1,N}_t  |\leq C_{\delta}  |  \hX^{ t_1,\infty}_{t_1}-\hX^{ t_1,N}_{t_1}  |=  C_{\delta}  | \langle \mu_{t_1}\rangle-  \langle \mu^N_{t_1} \rangle |\leq C_{\delta}  \int   | z-  z'  |\pi_{t_1}^N(dz,dz') $$
for some constant $C_{\delta}>0$ not depending on $N$.  Since  $ \int   | z-  z'  |\pi_{t}^N(dz,dz') \leq \left(  \int   | z-  z'  |^2\pi_{t}^N(dz,dz')\right)^{1/2}  $,  by taking the above couplings to be optimal  for $ \mathcal{W}_2$,   we get the estimate
\begin{equation*}
\sup_{t_1\leq t\leq t_1+\delta}  \E\left| \mathcal{W}_2^2(\mu_t,\delta_{ \hX^{ t_1,\infty}_t  })   -   \mathcal{W}_2^2(\mu^N_t,\delta_{ \hX^{ t_1,N}_t  })  \right| \leq  C' \sup_{t_1\leq t\leq t_1+\delta}   \E^{1/2}\left(  \mathcal{W}_2^2(\mu_t, \mu^N_t) \right) 
\end{equation*}
for some $C'>0$. We conclude thanks to Theorem  \ref{theo:propchaos_McKeanVlasov_eq}.
\end{proof}
\new{
\section{Strong Convergence Rate Result for the Exponential Projective Euler Scheme (EPES)}\label{sec:proof-numerical-scheme}

The main object of this section is to prove the convergence of the numerical scheme presented in Section \ref{sec:numerical-experiments} to the model \eqref{eq:HH-model} and establish the following rate of  convergence

\begin{prop}\label{theo:ConvNS} Asumme Hypothesis \ref{hyp:V0}, if $\chi(x)=O(x(1-x))$, then there exists a constant $C$ depending on the parameters of the system, but independent of $\dt$, such that for any $i=1,\ldots,N$:
\begin{align*}
 \E\left[\left(V_t^{(i)}- \widehat{V}_t^{(i)}\right)^2\right] +\sum_{x=m,n,h,y}\E\left[ |x_t^{(i)}-\hx_t^{(i)}|^2\right] & \leq C\dt.
\end{align*}
\end{prop}

 We decompose the proof of this proposition in several preliminary results.

The next result follows from the uniform bound for $\widehat{V}_t^{(i)}$ (see iii) in Remark \ref{rem:after-prop-boundedness-of-V^i}) and some standard arguments on local approximation of SDEs, so we omit the proof.
\begin{lem}\label{prop:local-errors} Under Hypothesis \ref{hyp:V0}, there exists a constant $C$ depending on the parameters of the system, but independent of $\dt$ such that
$$\sup_{i=1,\ldots,N}\E\left[\left(\widehat{V}_t^{(i)}- \widehat{V}^{(i)}_{\eta(t)}\right)^2\right]\leq C\dt^2,\qquad\sup_{i=1,\ldots,N}\E\left[\left( \cxi_{t}-  \widehat{x}^{(i)}_{\eta(t)}\right)^2\right]\leq C\dt.$$
\end{lem}

Next we establish a  the key step in the convergence of the scheme, namely  that, with extremely high probability, the processes $\widehat{x}^{(i)}$ and $\cxi$ coincide. 

\begin{lem}\label{prop:probability-of-projection}
Asumme Hypothesis \ref{hyp:V0}, if $\chi(x)=O(x(1-x))$, then there exists a constant $C$ depending on the parameters of the system, but independent of $\dt$, such that 
$$\sup_{i=1,\ldots,N}\sum_{x=m,n,h,y} \P\left(\cxi_t \notin[0,1]\right) \leq  \exp\left(-\frac{C}{\dt}\right).$$
\end{lem}

\begin{rem} It is not difficult to see that
\begin{align*}
\E\left[\left( \cxi_{t} -  \widehat{x}^{(i)}_{t}\right)^2\right]  &=  \E\left[\left( \cxi_{t}-  \widehat{x}^{(i)}_{t}\right)^2\ind{\{\cxi_t \notin[0,1]\}}\right]    \\
  &     \leq \sqrt{\sup_{j=1,\ldots,N}\E\left[ 2(\check{x}^j_{t})^2+1\right]\P\left(\cxi_t \notin[0,1]\right)}  \leq K \exp\left(-\frac{C}{2\dt}\right). 
\end{align*}
Notice that the RHS above tends to zero faster than any power of $\dt$ when $\dt\to0$.
\end{rem}

\begin{proof}[Proof of Lemma \ref{prop:probability-of-projection}] 
We first notice that conditional to $\F_{\eta(t)}$, $\cxi$ corresponds to an Ornstein-Uhlenbeck process, therefore its law is Gaussian with known conditional mean and conditional variance given by 
\begin{align*}
  \E_{ \eta(t)}\left[  \cxi_t \right]&=      \widehat{x}^{(i)}_{\eta(t)}\exp\left(-\left(\rho_x+\zeta_x\right)(\widehat{V}^{(i)}_{\eta(t)})(t-\eta(t))\right)\\  
  & \quad+ \frac{\rho_x(\widehat{V}^{(i)}_{\eta(t)})}{\left(\rho_x+\zeta_x\right)(\widehat{V}^{(i)}_{\eta(t)})}\left(1-\exp\left(-\left(\rho_x+\zeta_x\right)(\widehat{V}^{(i)}_{\eta(t)})(t-\eta(t))\right)\right),\\
 \Var_{\eta(t)}\left[  \cxi_t \right] & =   \frac{\sigma^2_x(\widehat{V}^{(i)}_{\eta(t)},  \widehat{x}_{\eta(t)}^{(i)})}{2\left(\rho_x+\zeta_x\right)(\widehat{V}^{(i)}_{\eta(t)})}\left(1-\exp\left(-2\left(\rho_x+\zeta_x\right)(\widehat{V}^{(i)}_{\eta(t)})(t-\eta(t))\right)\right).
\end{align*}
Observe that the conditional variance is strictly positive if $t>\eta(t)$,  $\widehat{x}^{(i)}_{\eta(t)}\neq0$ and $\widehat{x}^{(i)}_{\eta(t)}\neq1$. Since for $\widehat{x}^{(i)}_{\eta(t)}=0$ or $\widehat{x}^{(i)}_{\eta(t)}=1$ the diffusions coefficient vanish, and in that case the solution to the  ODE for $\cxi$ remains in $[0,1]$ almost surely, we can restrict ourselves to the case $\widehat{x}^{(i)}_{\eta(t)}\in (0,1)$.

Using the inequality for Gaussian concentration, conditional to $\F_{\eta(t)}$,  we have
\begin{align*}
 \P_{\eta(t)}\left(\cxi_t \leq 0\right) &=     \P_{\eta(t)}\left(\frac{\cxi_t-\E_{ \eta(t)}\left[  \cxi_t \right]}{\sqrt{\Var_{\eta(t)}\left[  \cxi_t \right]}} \leq \frac{-\E_{ \eta(t)}\left[  \cxi_t \right]}{\sqrt{\Var_{\eta(t)}\left[  \cxi_t \right]}}\right)   \leq \frac{1}{2}\exp\left(- \frac{\E_{ \eta(t)}\left[  \cxi_t \right]^2}{\Var_{\eta(t)}\left[  \cxi_t \right]}\right).
\end{align*}
Since for $t$ small enough 
$$1-\exp\left(-2\left(\rho_x+\zeta_x\right)(\widehat{V}^{(i)}_{\eta(t)})(t-\eta(t))\right) \leq 2\left(\rho_x+\zeta_x\right)(\widehat{V}^{(i)}_{\eta(t)})(t-\eta(t)), $$
and $t-\eta(t)\leq \dt$, we can bound the conditional variance, and then it follows
\begin{align*}
 \P_{\eta(t)}\left(\cxi_t \leq 0\right) & \leq \frac{1}{2}\exp\left(- \frac{\E_{ \eta(t)}\left[  \cxi_t \right]^2}{\sigma^2_x(\widehat{V}^{(i)}_{\eta(t)},  \widehat{x}_{\eta(t)}^{(i)})\dt}\right).
\end{align*}
On the other hand, $\E_{ \eta(t)}\left[  \cxi_t \right]$ is a weighted mean between to quantities in $[0,1]$, therefore
$$  \E_{ \eta(t)}\left[  \cxi_t \right]\geq      \widehat{x}^{(i)}_{\eta(t)} \wedge \frac{\rho_x(\widehat{V}^{(i)}_{\eta(t)})}{\left(\rho_x+\zeta_x\right)(\widehat{V}^{(i)}_{\eta(t)})},$$
hence
\begin{equation}\label{eq:bound1-projection}
\begin{aligned}
 \P_{\eta(t)}\left(\cxi_t \leq 0\right) & \leq \frac{1}{2}\exp\left( \frac{-\rho_x^2(\widehat{V}^{(i)}_{\eta(t)})}{\left(\rho_x+\zeta_x\right)^2(\widehat{V}^{(i)}_{\eta(t)})\sigma^2_x(\widehat{V}^{(i)}_{\eta(t)},  \widehat{x}_{\eta(t)}^{(i)})\dt}\right)\ind{\left\{ \widehat{x}_{\eta(t)}^{(i)} \geq   \frac{\rho_x(\widehat{V}^{(i)}_{\eta(t)})}{\left(\rho_x+\zeta_x\right)(\widehat{V}^{(i)}_{\eta(t)})} \right\}}\\
 &\quad +\frac{1}{2}\exp\left(- \frac{\left(\widehat{x}_{\eta(t)}^{(i)}\right)^2}{\sigma^2_x(\widehat{V}^{(i)}_{\eta(t)},  \widehat{x}_{\eta(t)}^{(i)})\dt}\right)\ind{\left\{ \widehat{x}_{\eta(t)}^{(i)} \leq   \frac{\rho_x(\widehat{V}^{(i)}_{\eta(t)})}{\left(\rho_x+\zeta_x\right)(\widehat{V}^{(i)}_{\eta(t)})} \right\}}.
\end{aligned}
\end{equation}

To bound the first exponential in the right-hand side of the last inequality,  we notice that since the process $\widehat{V}^{(i)}$ is uniformly bounded and $\sigma$ is bounded, we can easily exhibit a constant $C_1>0$ independent of $i$ such that
$$\frac{\rho_x^2(\widehat{V}^{(i)}_{\eta(t)})}{\left(\rho_x+\zeta_x\right)^2(\widehat{V}^{(i)}_{\eta(t)})\sigma^2_x(\widehat{V}^{(i)}_{\eta(t)},  \widehat{x}_{\eta(t)}^{(i)})}\geq C_1.$$
For the second term in the right-hand side of \eqref{eq:bound1-projection}, since $x^2/\chi(x)^2$ is bounded from below in $(0,1)$, there exists $C_2>0$, such that
$$\frac{\left(\widehat{x}_{\eta(t)}^{(i)}\right)^2}{\sigma^2_x(\widehat{V}^{(i)}_{\eta(t)},  \widehat{x}_{\eta(t)}^{(i)})} \geq C_2,$$
from which we conclude
$$ \P\left(\cxi_t \leq 0\right)  \leq \exp\left(-\frac{C_1\wedge C_2}{\dt}\right).$$
An analogous computation shows 
$$ \P\left(\cxi_t \geq 1\right)  \leq \exp\left(-\frac{C_1\wedge C_2}{\dt}\right).$$
\end{proof}

The last preliminary step in the proof of Proposition \ref{theo:ConvNS} is the following 
\begin{lem}\label{lem:conv-num-step1} Under hypotheses of Proposition \ref{theo:ConvNS}, consider
$$u(t):= \E\left[\left(V_t^{(i)}- \widehat{V}_t^{(i)}\right)^2\right] +\sum_{x=m,n,h,y}\E\left[ |x_t^{(i)}-\cxi_t|^2\right].$$
Then there exists a constant $C$ depending on the parameters of the system, but independent of $\dt$, such that
\begin{equation}\label{eq:bound-error-t-from-tk}
u(t) \leq  \Bigg( u(\eta(t))  +  C\dt^2\Bigg)e^{C\dt}.
\end{equation}
\end{lem}

\begin{proof}
Thanks to the boundedness of the processes, drift and diffusion coefficients, that is $b_x$ and $\sigma_x$, behave like Lipschitz functions, just as in the proof of Lemma \ref{lem:well-posedHH}. Then, thanks to It\^o formula and pivoting with in drift and diffusion with the point $( \widehat{V}^{(i)}_{s},\cxi_s)$
\begin{align*}
\E&\left[(x^{(i)}_t- \cxi_t)^2 \right] \\
 & \leq  \E\left[  ( x^{(i)}_{\eta(t)}- \widehat{x}^{(i)}_{\eta(t)})^2\right]  +2\int_{{\eta(t)}}^{t}\E\left[(x^{(i)}_s- \cxi_s)\left(b_x(V^{(i)}_s,x^{(i)}_s)-b_x( \widehat{V}^{(i)}_{s},\cxi_s)\right)\right] ds\\
&\quad +2\int_{{\eta(t)}}^{t}\E\left[(x^{(i)}_s- \cxi_s)\left(b_x( \widehat{V}^{(i)}_{s},\cxi_s) -b_x( \widehat{V}^{(i)}_{\eta(t)},\cxi_s)\right)\right] ds\\
&\quad+ \int_{\eta(t)}^{t}{2\E\left[\left(\sigma_x(V_s^{(i)},x_s^{(i)})- \sigma_x( \widehat{V}_{s}^{(i)}, \cxi_s)\right)^2\right]  +2\E\left[\left(\sigma_x( \widehat{V}_{s}^{(i)},\cxi_s) - \sigma_x( \widehat{V}_{\eta(t)}^{(i)}, \widehat{x}_{\eta(t)}^{(i)})\right)^2\right]ds}.
\end{align*}
from where the Lipchitz property of the coefficients, Lemma \ref{prop:local-errors} to bound the terms involving the local error and some classical arguments lead to
\begin{align*}
\E\left[(x^{(i)}_t- \cxi_t)^2 \right]  & \leq  \E\left[  ( x^{(i)}_{\eta(t)}-\cxi_{\eta(t)})^2\right]  +C\int_{{\eta(t)}}^{t} \E\left[  ( x^{(i)}_{s}-\cxi_{s})^2\right]  + \E\left[  ( V^{(i)}_{s}- \widehat{V}^{(i)}_{s})^2\right] ds + C\dt^2.
\end{align*}

On the other hand, for the voltage error we obtain first  the a.s. bound
\begin{align*}
 \left(V_t^{(i)}- \widehat{V}_t^{(i)}\right)^2
 &\leq \left( V^{(i)}_{\eta(t)}- \widehat{V}^{(i)}_{\eta(t)}\right)^2  + C\int_{{\eta(t)}}^{t} |V_s^{(i)}- \widehat{V}_s^{(i)}|^2 +  \sum_{x=m,n,h} |x_s^{(i)}-\widehat{x}_{\eta(t)}^{(i)}|^2 ds \\
& \quad   + C\int_{\eta(t)}^t |V_s^{(i)}- \widehat{V}_s^{(i)}|^2 + \frac{1}{N}\sum_{j=1}^{N} |V_s^{(j)} - \widehat{V}_{\eta(t)}^{(j)}|^2  ds,\\
 &\quad + C \int_{\eta(t)}^t |V_s^{(i)}-\hV_s^{(i)}|^2+ (V_s^{(i)}+\Vrev)^2\frac{1}{N}\sum_{j=1}^{N}{|y^{(j)}_s  -  \widehat{y}^{(j)}_{\eta(t)}|^2 } ds. 
\end{align*}
Thanks to the exchangeability of the particles, it follows that
$$\E\left[\frac{1}{N}\sum_{j=1}^{N}{|y^{(j)}_s  -  \widehat{y}^{(j)}_{\eta(t)}|^2 }\right] = \E\left[\left(y^{(i)}_s  -  \widehat{y}^{(i)}_{\eta(t)}\right)^2\right],\;\;\; \E\left[\frac{1}{N}\sum_{j=1}^{N}{|V_s^{(j)} - \widehat{V}_{\eta(t)}^{(j)}|^2  }\right] = \E\left[\left(V^{(i)}_s  -  \widehat{V}^{(i)}_{\eta(t)}\right)^2\right],$$
and then, since the processes are uniformly bounded, we get that
\begin{align*}
 \E&\left[\left(V_t^{(i)}- \widehat{V}_t^{(i)}\right)^2\right]\\
 &\leq  \E\left[\left( V^{(i)}_{\eta(t)}- \widehat{V}^{(i)}_{\eta(t)}\right)^2\right]  + C\int_{{\eta(t)}}^{t}\E\left[\left(V_s^{(i)}- \widehat{V}_s^{(i)}\right)^2\right] ds+  \\
& \quad   + C\int_{\eta(t)}^t   \E\left[\left(V^{(i)}_s  -  \widehat{V}^{(i)}_{\eta(t)}\right)^2\right] +\sum_{x=m,n,h,y}\E\left[ |x_s^{(i)}-\widehat{x}_{\eta(t)}^{(i)}|^2\right] ds\\
 &\leq  \E\left[\left( V^{(i)}_{\eta(t)}- \widehat{V}^{(i)}_{\eta(t)}\right)^2\right]  + C\int_{{\eta(t)}}^{t}\E\left[ |V_s^{(i)}- \widehat{V}_s^{(i)}|^2\right] ds+  \\
& \quad   + C\int_{\eta(t)}^t   \E\left[\left(V^{(i)}_s  -  \widehat{V}^{(i)}_{s}\right)^2\right] + C\dt^2 +\sum_{x=m,n,h,y}\E\left[ |x_s^{(i)}-\cxi_{s}|^2\right] + C\dt ds\\
 &\leq  \E\left[\left( V^{(i)}_{\eta(t)}- \widehat{V}^{(i)}_{\eta(t)}\right)^2\right] + C\int_{{\eta(t)}}^{t}\E\left[ \left(V_s^{(i)}- \widehat{V}_s^{(i)}\right)^2\right]+\sum_{x=m,n,h,y}\E\left[ |x_s^{(i)}-\cxi_{s}|^2\right] ds+ C\dt^2.
\end{align*}
We can summarize the previous computations as
\begin{align*}
 u(t) &\leq u(\eta(t)) + C\int_{{\eta(t)}}^{t}u(s)ds+ C\dt^2,   
\end{align*}
from where we conclude thanks to Gronwall's inequality.
\end{proof}
\begin{proof}[Proof of Proposition \ref{theo:ConvNS} ] 

From the previous Lemma, denoting $u_k = u(t_k)$ we obtain the following recurrence relationship:
$$u_{k+1} \leq A u_k + B,\quad u_0=0,\quad A= e^{C\dt} ,\;\; B= C\dt^2e^{C\dt}.$$
Iterating this inequality, it is easy to conclude that
$$u_{k+1} \leq B\frac{A^{k+1}-1}{A-  1}\leq  C\dt^2\frac{\left(e^{C\dt}\right)^{k+1}}{e^{C\dt}-  1}=C\dt^2\frac{e^{Ct_{k+1}}}{e^{C\dt}-  1}.$$
But when $\dt\to0$, we have $e^{C\dt}-  1 \sim C\dt$, and therefore $u_{k+1} \leq  C\dt.$
Inserting this in \eqref{eq:bound-error-t-from-tk}, we conclude

\begin{align*}
 \E&\left[\left(V_t^{(i)}- \widehat{V}_t^{(i)}\right)^2\right] +\sum_{x=m,n,h,y}\E\left[ |x_t^{(i)}-\hx_t^{(i)}|^2\right]\\
  &\leq  \E\left[\left(V_t^{(i)}- \widehat{V}_t^{(i)}\right)^2\right] +\sum_{x}\E\left[ |x_t^{(i)}-\cxi_t|^2\right]    +  \E\left[ |\cxi_t-\hx_t^{(i)}|^2\right]    \leq C\dt + \sum_{x}\P\left(\cxi_t \notin[0,1]\right), 
\end{align*}
from where the statement follows,  applying Lemma \ref{prop:probability-of-projection}.

\end{proof}

}

\bibliographystyle{plain}
\bibliography{biblio}

\newpage

\end{document}